\documentclass{aart}

\usepackage{titlesec}
\titleformat{\subsection}[runin]{\normalfont\bfseries}{\thesubsection.}{.5em}{}[.]\titlespacing{\subsection}{0pt}{2ex plus .1ex minus .2ex}{.8em}
\titleformat{\subsubsection}[runin]{\normalfont\itshape}{\thesubsubsection.}{.3em}{}[.]\titlespacing{\subsubsection}{0pt}{1ex plus .1ex minus .2ex}{.5em}

\usepackage[labelfont=sc,font=small,labelsep=period]{caption}
\usepackage[letterpaper, hmargin=1in, top=1.5in, bottom=1.9in, footskip=0.7in]{geometry}

\usepackage{amsmath} %[intlimits]
\usepackage{amssymb}
\usepackage{amsfonts}
\usepackage{latexsym}
\usepackage{amsthm}
\usepackage{amsxtra}
\usepackage{amscd}
\usepackage{bbm}
\usepackage{mathrsfs}
\usepackage{bm}

\usepackage{graphicx, color}
\usepackage[nottoc,notlof,notlot]{tocbibind} 
\usepackage{cite} % Enabling `[1-4]` - style citing
\usepackage{xspace} %command to insert a space after a latex command

\numberwithin{equation}{section}
\numberwithin{figure}{section}
\renewcommand{\r}{\mathrm}  %upright
\renewcommand{\cal}{\mathcal}

\newcommand{\ol}[1]{\overline{#1} \!\,} %overline
\newcommand{\wh}{\widehat}
\newcommand{\wt}{\widetilde}
\newcommand{\f}[1]{\boldsymbol{\mathrm{#1}}} %bold
\newcommand{\ii}{\mathrm{i}}
\newcommand{\dd}{\mathrm{d}}
\newcommand{\col}{\mathrel{\mathop:}}
\newcommand{\st}{\,\col\,}
\newcommand{\deq}{\mathrel{\mathop:}=}

\newcommand{\eqdist}{\overset{d}{=}}

\newcommand{\simdist}{\overset{d}{\sim}}
\renewcommand{\leq}{\leqslant}
\renewcommand{\geq}{\geqslant}
\renewcommand{\le}{\leqslant}
\renewcommand{\ge}{\geqslant}

\newcommand{\ind}[1]{\f 1 (#1)}
\newcommand{\indb}[1]{\f 1 \pb{#1}}

\newcommand{\indBB}[1]{\f 1 \pBB{#1}}

\renewcommand{\epsilon}{\varepsilon}

\renewcommand{\P}{\mathbb{P}}
\newcommand{\E}{\mathbb{E}}
\newcommand{\R}{\mathbb{R}}
\newcommand{\C}{\mathbb{C}}
\newcommand{\N}{\mathbb{N}}
\newcommand{\Z}{\mathbb{Z}}

\newcommand{\p}[1]{({#1})}
\newcommand{\pb}[1]{\bigl({#1}\bigr)}
\newcommand{\pB}[1]{\Bigl({#1}\Bigr)}
\newcommand{\pbb}[1]{\biggl({#1}\biggr)}
\newcommand{\pBB}[1]{\Biggl({#1}\Biggr)}

\newcommand{\q}[1]{[{#1}]}
\newcommand{\qb}[1]{\bigl[{#1}\bigr]}
\newcommand{\qB}[1]{\Bigl[{#1}\Bigr]}
\newcommand{\qbb}[1]{\biggl[{#1}\biggr]}
\newcommand{\qBB}[1]{\Biggl[{#1}\Biggr]}
\newcommand{\qa}[1]{\left[{#1}\right]}

\newcommand{\qq}[1]{[\![{#1}]\!]}

\newcommand{\h}[1]{\{{#1}\}}
\newcommand{\hb}[1]{\bigl\{{#1}\bigr\}}
\newcommand{\hB}[1]{\Bigl\{{#1}\Bigr\}}

\newcommand{\hBB}[1]{\Biggl\{{#1}\Biggr\}}

\newcommand{\abs}[1]{\lvert #1 \rvert}
\newcommand{\absb}[1]{\bigl\lvert #1 \bigr\rvert}

\newcommand{\absbb}[1]{\biggl\lvert #1 \biggr\rvert}
\newcommand{\absBB}[1]{\Biggl\lvert #1 \Biggr\rvert}

\newcommand{\norm}[1]{\lVert #1 \rVert}
\newcommand{\normb}[1]{\bigl\lVert #1 \bigr\rVert}

\newcommand{\normbb}[1]{\biggl\lVert #1 \biggr\rVert}

\newcommand{\scalar}[2]{\langle{#1} \mspace{2mu}, {#2}\rangle}

\newcommand{\scalarbb}[2]{\biggl\langle{#1} \,\mspace{2mu},\, {#2}\biggr\rangle}

\DeclareMathOperator{\diag}{diag}
\DeclareMathOperator{\tr}{Tr}

\DeclareMathOperator{\supp}{supp}

\DeclareMathOperator{\re}{Re}
\DeclareMathOperator{\im}{Im}

\DeclareMathOperator{\dist}{dist}

\DeclareMathOperator{\spn}{Span}

\newcommand{\beqa}{\begin{eqnarray}}
\newcommand{\eeqa}{\end{eqnarray}}
\newcommand{\e}{\varepsilon}

\newcommand{\bv}{{\bf{v}}}
\newcommand{\bw}{{\bf{w}}}

\newcommand{\be}{\begin{equation}}
\newcommand{\ee}{\end{equation}}

\newcommand{\txt}[1]{\text{\rm{#1}}}

\definecolor{darkred}{rgb}{0.8,0,0.3}
\definecolor{darkblue}{rgb}{0,0.3,0.8}

\usepackage{ifthen}
\def\comment#1{\ifthenelse{\isodd{\value{page}}}{\marginpar{\raggedright\scriptsize{\textcolor{darkred}{#1}}}}{\marginpar{\raggedleft\scriptsize{\textcolor{darkred}{#1}}}}}

\theoremstyle{plain} %plain, definition, remark
\newtheorem{theorem}{Theorem}[section]
\newtheorem*{theorem*}{Theorem}
\newtheorem{lemma}[theorem]{Lemma}
\newtheorem*{lemma*}{Lemma}
\newtheorem{corollary}[theorem]{Corollary}
\newtheorem*{corollary*}{Corollary}
\newtheorem{proposition}[theorem]{Proposition}
\newtheorem*{proposition*}{Proposition}
\newtheorem{definition}[theorem]{Definition}
\newtheorem*{definition*}{Definition}

\theoremstyle{definition} %plain, definition, remark
\newtheorem{example}[theorem]{Example}
\newtheorem*{example*}{Example}
\newtheorem{remark}[theorem]{Remark}

\newtheorem*{remark*}{Remark}
\newtheorem*{remarks*}{Remarks}

\begin{document}

\title{On the principal components of sample covariance matrices}

\author{Alex Bloemendal\footnote{Harvard University, alexb@math.harvard.edu.} \and
Antti Knowles\footnote{ETH Z\"urich, knowles@math.ethz.ch. Partially supported by Swiss National Science Foundation grant 144662.}
 \and Horng-Tzer Yau\footnote{Harvard University, htyau@math.harvard.edu. Partially supported by NSF Grant DMS-1307444 and Simons investigator fellowship.}  \and Jun Yin\footnote{University of Wisconsin, jyin@math.wisc.edu. Partially supported by NSF Grant DMS-1207961.}}

%\date{April 2, 2014}

\maketitle

\begin{abstract}
We introduce a class of $M \times M$ sample covariance matrices $\cal Q$ which subsumes and generalizes several previous models. The associated population covariance matrix $\Sigma = \E \cal Q$ is assumed to differ from the identity by a matrix of bounded rank. All quantities except the rank of $\Sigma - I_M$ may depend on $M$ in an arbitrary fashion. We investigate the principal components, i.e.\ the top eigenvalues and eigenvectors, of $\cal Q$. We derive precise large deviation estimates on the generalized components $\scalar{\f w}{\f \xi_i}$ of the outlier and non-outlier eigenvectors $\f \xi_i$.  Our results also hold near the so-called BBP transition, where outliers are created or annihilated, and for degenerate or near-degenerate outliers. We believe the obtained rates of convergence to be optimal. In addition, we derive the asymptotic distribution of the generalized components of the non-outlier eigenvectors. A novel observation arising from our results is that, unlike the eigenvalues, the eigenvectors of the principal components contain information about the \emph{subcritical} spikes of $\Sigma$.

The proofs use several results on the eigenvalues and eigenvectors of the uncorrelated matrix $\cal Q$, satisfying $\E \cal Q = I_M$, as input: the isotropic local Marchenko-Pastur law established in \cite{BEKYY}, level repulsion, and quantum unique ergodicity of the eigenvectors. The latter is a special case of a new universality result for the joint eigenvalue-eigenvector distribution.
\end{abstract}

%\newpage
%\tableofcontents

\section{Introduction} \label{sec:intro}

In this paper we investigate $M \times M$ sample covariance matrices of the form
\begin{equation} \label{def_calQ}
\cal Q \;=\; \frac{1}{N} A A^* \;=\; \pbb{\frac{1}{N} \sum_{\mu = 1}^N A_{i \mu} A_{j \mu}}_{i,j = 1}^M\,,
\end{equation}
where the \emph{sample matrix} $A = (A_{i \mu})$ is a real $M \times N$ random matrix. The main motivation to study such models stems from multivariate statistics. Suppose we are interested in the statistics of $M$ mean-zero variables $\f a = (a_1, \dots, a_M)^*$ which are thought to possess a certain degree of interdependence. Such problems of multivariate statistics commonly arise in population genetics, economics, wireless communication, the physics of mixtures, and statistical learning \cite{Johnstone2, Paul, BBP}. The goal is to unravel the interdependencies among the variables $\f a$ by finding the \emph{population covariance matrix}
\begin{equation} \label{Sigma_mean_zero}
\Sigma \;=\; \E \f a \f a^* \;=\; (\E a_i a_j)_{i,j = 1}^M\,.
\end{equation}
To this end, one performs a large number, $N$, of repeated, independent measurements, called ``samples'', of the variables $\f a$. Let $A_{i\mu}$ denote the value of $a_i$ in the $\mu$-th sample. Then the sample covariance matrix \eqref{def_calQ} is the empirical mean approximating the population covariance matrix $\Sigma$.

In general, the mean of the variables $\f a$ is nonzero and unknown. In that case, the population covariance matrix \eqref{Sigma_mean_zero} has to be replaced with the general form
\begin{equation*}
\Sigma \;=\; \E \qb{(\f a - \E \f a) (\f a - \E \f a)^*}\,.
\end{equation*}
Correspondingly, one has to subtract from $A_{i \mu}$ the empirical mean of the $i$-th row of $A$, which we denote by $[A]_i \deq \frac{1}{N} \sum_{\mu = 1}^N A_{i \mu}$. Hence, we replace \eqref{def_calQ} with
\begin{equation} \label{def_calQ_mean}
\dot {\cal Q} \;\deq\; \frac{1}{N - 1} A (I_N - \f e \f e^*) A^* \;=\; \pbb{\frac{1}{N - 1} \sum_{\mu = 1}^N (A_{i \mu} - [A]_i) (A_{j \mu} - [A]_j)}_{i,j = 1}^M\,,
\end{equation}
where we introduced the vector
\begin{equation} \label{def_vec_e}
\f e \;\deq\; N^{-1/2} (1,1, \dots, 1)^* \;\in\; \R^N\,.
\end{equation}
Since $\dot{\cal Q}$ is invariant under the shift $A_{i\mu} \mapsto A_{i \mu} + m_i$ for any deterministic vector $(m_i)_{i = 1}^M$, we may assume without loss of generality that $\E A_{i \mu} = 0$. We shall always make this assumption from now on.

It is easy to check that $\E \cal Q = \E \dot{\cal Q} = \Sigma$. Moreover, we shall see that the principal components of $\cal Q$ and $\dot{\cal Q}$ have identical asymptotic behaviour. For simplicity of presentation, in the following we focus mainly on $\cal Q$, bearing in mind that every statement we make on $\cal Q$ also holds verbatim for $\dot{\cal Q}$ (see Theorem \ref{thm: Q dot} below).

By the law of large numbers, if $M$ is fixed and $N$ taken to infinity, the sample covariance matrix $\cal Q$ converges almost surely to the population covariance matrix $\Sigma$. In many modern applications, however, the population size $M$ is very large and obtaining samples is costly. Thus, one is typically interested in the regime where $M$ is of the same order as $N$, or even larger. In this case, as it turns out, the behaviour of $\cal Q$ changes dramatically and the problem becomes much more difficult. In \emph{principal component analysis}, one seeks to understand the correlations by considering the \emph{principal components}, i.e.\ the top eigenvalues and associated eigenvectors, of $\cal Q$. These provide an effective low-dimensional projection of the high-dimensional data set $A$, in which the significant trends and correlations are revealed by discarding superfluous data.

The fundamental question, then, is how the principal components of $\Sigma = \E \cal Q$ are related to those of $\cal Q$.

\subsection{The uncorrelated case} \label{sec: uncorr}
In the ``null'' case, the variables $\f a$ are uncorrelated and $\Sigma = I_M$ is the identity matrix. The global distribution of the eigenvalues is governed by the \emph{Marchenko-Pastur law} \cite{MP}. More precisely, defining the dimensional ratio
\begin{equation} \label{def_phi}
 \phi \;\equiv\; \phi_N \;\deq\;\frac MN\,,
\end{equation}
the empirical eigenvalue density of the rescaled matrix $Q = \phi^{-1/2} \cal Q$ has the same asymptotics for large $M$ and $N$ as
\begin{equation} \label{MP_law}
\frac{\sqrt{\q{(x-\gamma_-)(\gamma_+-x)}_+}}{2 \pi \sqrt{\phi} \, x}\, \dd x + (1-\phi^{-1})_+ \, \delta(\dd x)\,,
\end{equation}
where we defined
\begin{equation} \label{def:gamma_pm}
\gamma_\pm  \;\deq\; \phi^{1/2}+ \phi^{-1/2} \pm 2
\end{equation}
to be the edges of the limiting spectrum. Hence, the unique nontrivial eigenvalue $1$ of $\Sigma$ spreads out into a bulk spectrum of $\cal Q$ with diameter $4 \phi^{1/2}$. Moreover, the local spectral statistics are universal; for instance, the top eigenvalue of $\cal Q$ is distributed according to the Tracy-Widom-1 distribution \cite{TW1, TW2, Joh1, Johnstone}. Finally, the eigenvectors of $\cal Q$ are uniformly distributed on the unit sphere of $\R^M$; following \cite{BoY1}, we call this property the \emph{quantum unique ergodicity} of the eigenvectors of $\cal Q$, a term borrowed from quantum chaos. We refer to Theorem \ref{thm:univ_H} and Remark \ref{rem: universality} below for precise statements.

\subsection{Examples and outline of the model} \label{sec:examples}
The problem becomes much more interesting if the variables $\f a$ are correlated. Several models for correlated data have been proposed in the literature, starting with the Gaussian spiked model from the seminal paper of Johnstone \cite{Johnstone}. Here we propose a general model which includes many previous models as special cases. We motivate it using two examples.

\begin{itemize}
\item[(1)]
Let $\f a = T \f b$, where the entries of $\f b$ are independent with zero mean and unit variance, and $T$ is a deterministic $M \times M$ matrix. 
This may be interpreted as an observer studying a complicated system whose randomness is governed by many independent internal variables $\f b$. The observer only has access to the external variables $\f a$, which may depend on the internal variables $\f b$ in some complicated and unknown fashion. Assuming that this dependence is linear, we obtain $\f a = T \f b$. The sample matrix for this model is therefore $A = T B$, where $B$ is an $M \times N$ matrix with independent entries of unit variance. The population covariance matrix is $\Sigma = T T^*$.
\item[(2)]
Let $r \in \N$ and set
\begin{equation*}
\f a \;=\; \f z + \sum_{l = 1}^r y_l \f u_l\,.
\end{equation*}
Here $\f z \in \R^M$ is a vector of ``noise'', whose entries are independent with zero mean and unit variance. The ``signal'' is given by the contribution of $r$ terms of the form $y_l \f u_l$, whereby $y_1, \dots, y_r$ are independent, with zero mean and unit variance, and $\f u_1, \dots, \f u_r \in \R^M$ are arbitrary deterministic vectors. The sample matrix is
\begin{equation*}
A \;=\; Z + \sum_{l = 1}^r \f u_l \f y_l^*\,,
\end{equation*}
where, writing $Y \deq [\f y_1, \dots, \f y_r] \in \R^{N \times r}$, the $(M + r) \times N$ matrix $B \deq \binom{Z}{Y^*}$ has independent entries with zero mean and unit variance. Writing $U \deq [\f u_1, \dots, \f u_r] \in \R^{M \times r}$, we therefore have
\begin{equation*}
A \;=\; T B \,, \qquad T \;\deq\; (I_M, U)\,.
\end{equation*}
The population covariance matrix is $\Sigma = T T^* = I_M + U U^*$.
\end{itemize}

Below we shall refer to these examples as Examples (1) and (2) respectively.
Motivated by them, we now outline our model. Let $B$ be an $(M + r) \times N$ matrix whose entries are independent with zero mean and unit variance. We choose a deterministic $M \times (M + r)$ matrix $T$, and set $\cal Q = \frac{1}{N} TBB^* T^*$.
We stress that we do not assume that the underlying randomness is Gaussian. Our key assumptions are (i) $r$ is bounded; (ii) $\Sigma - I_M$ has bounded rank; (iii) $\log N$ is comparable to $\log M$; (iv) the entries of $B$ are independent, with zero mean and unit variance, and have a sufficient number of bounded moments. The precise assumptions are given in Section \ref{sec: model} below. We emphasize that everything apart from $r$ and the rank of $\Sigma - I_M$ is allowed to depend on $N$ in an arbitrary fashion.

As explained around \eqref{def_calQ_mean}, in addition to $\cal Q$ we also consider the matrix $\dot {\cal Q} = \frac{1}{N - 1} TB(I_N - \f e \f e^*)B^* T^*$,
whose principal components turn out to have the same asymptotic behaviour as those of $\cal Q$.

\subsection{Definition of model} \label{sec: model}

In this section we give the precise definition of our model and introduce some basic notations. For convenience, we always work with the rescaled sample covariance matrix
\begin{equation} \label{def_Q}
Q \;=\; \phi^{-1/2} \cal Q\,.
\end{equation}
The motivation behind this rescaling is that, as observed in \eqref{MP_law}, it ensures that the bulk spectrum of $Q$ has asymptotically a fixed diameter, $4$, for arbitrary $N$ and $M$.
%Accordingly, in the precise definition of $Q$ below, we replace the random matrix $B$ from \eqref{model_intr} with its rescaled version $X = (MN)^{-1/4} B$.

We always regard $N$ as the fundamental large 
parameter, and write $M \equiv M_N$. Here, and throughout the following, 
in order to unclutter notation we omit the argument $N$ in quantities, 
such as $M$, that depend on it. In other words, every symbol that is not explicitly a constant is in fact a sequence indexed by $N$.
We assume that $M$ and $N$ satisfy the bounds
\begin{equation} \label{NM gen}
N^{1/C} \;\leq\; M \;\leq\; N^C
\end{equation}
for some positive constant $C$. %We recall the dimensional ratio $\phi \deq M/N$ from \eqref{def_phi}.
%, and emphasize that it may depend on $N$ and need not converge in $(0, \infty)$.

Fix a constant $r = 0,1,2,3,\dots$. Let $X$ be an $(M + r) \times N$ random matrix and $T$ an $M \times (M + r)$ deterministic matrix. For definiteness, and bearing the motivation of sample covariance matrices in mind, we assume that the entries of $X$ and $T$ are real. However, our method also trivially applies to complex-valued $X$ and $T$, with merely cosmetic changes to the proofs. We consider the $M \times M$ matrix
\begin{equation} \label{wtH_cov}
Q \;\deq\; T X X^* T^*\,.
\end{equation}
Since $TX$ is an $M \times N$ matrix, we find that $Q$ has
\begin{equation}
K \;\deq\; M \wedge N
\end{equation}
nontrivial (i.e.\ nonzero) eigenvalues.

We define the population covariance matrix
\begin{equation} \label{def_Sigma}
\Sigma \;\equiv\; \Sigma(T) \;\deq\; T T^* \;=\; \sum_{i = 1}^M \sigma_i \f v_i \f v_i^* \;=\; I_M + \phi^{1/2} \sum_{i = 1}^M d_i \f v_i \f v_i^*\,,
\end{equation}
where $\{\f v_i\}_{i = 1}^M$ is a real orthonormal basis of $\R^M$ and $\{\sigma_i\}_{i = 1}^M$ are the eigenvalues of $\Sigma$. Here we introduce the representation
\begin{equation*}
\sigma_i \;=\; 1 + \phi^{1/2} d_i
\end{equation*}
for the eigenvalues $\sigma_i$. We always order the values $d_i$ such that
\begin{equation*}
d_1 \;\geq\; d_2 \;\geq\; \cdots \;\geq\; d_M\,.
\end{equation*}
We suppose that $\Sigma$ is positive definite, so that each $d_i$ lies in the interval
\begin{equation} \label{def_cal_D}
\cal D \;\deq\; (-\phi^{-1/2}, \infty)\,.
\end{equation}
Moreover, we suppose that $\Sigma - I_M$ has bounded rank, i.e.
\begin{equation} \label{def_calR}
\cal R \;\deq\; \h{i \col d_i \neq 0}
\end{equation}
has bounded cardinality, $\abs{\cal R} = O(1)$. We call the couples $((d_i, \f v_i))_{i \in \cal R}$ the \emph{spikes} of $\Sigma$.

We assume that the entries $X_{i \mu}$ of $X$ are independent (but not necessarily identically distributed) random variables satisfying
\begin{equation} \label{cond on entries of X}
\E X_{i \mu}\;=\;0\,,\qquad \E X_{i \mu}^2\;=\;\frac{1}{\sqrt {N M}}\,.
\end{equation}
In addition, we assume that,  for all $p \in \N$, the random variables $(N M)^{1/4} X_{i\mu}$ have a uniformly bounded $p$-th moment. In other words, we assume that there is a constant $C_p$ such that
\begin{equation} \label{moments of X-1}
\E \absb{(N M)^{1/4} X_{i\mu}}^p \;\leq\; C_p\,.
\end{equation}
The assumption that \eqref{moments of X-1} hold for all $p \in \N$ may be easily relaxed. For instance, it is easy to check that our results and their proofs remain valid, after minor adjustments, if we only require that \eqref{moments of X-1} holds for all $p \leq C$ for some large enough constant $C$. We do not pursue such generalizations further.

Our results concern the eigenvalues of $Q$, denoted by
\begin{equation*}
\mu_1 \;\geq\; \mu_2 \;\geq\; \cdots \;\geq\; \mu_M\,,
\end{equation*}
and the associated unit eigenvectors of $Q$, denoted by
\begin{equation*}
\f \xi_1, \f \xi_2, \dots, \f \xi_M \;\in\; \R^M\,.
\end{equation*}

\subsection{Sketch of behaviour of the principal components of $Q$} \label{sec:beh_ev}
To guide the reader, we now give a heuristic description of the behaviour of principal components of $Q$. The spectrum of $Q$ consists of a \emph{bulk spectrum} and of \emph{outliers}---eigenvalues separated from the bulk. The bulk contains an order $K$ eigenvalues, which are distributed on large scales according to the Marchenko-Pastur law \eqref{MP_law}. In addition, if $\phi > 1$ there are $M - K$ trivial eigenvalues at zero. Each $d_i$ satisfying $\abs{d_i} > 1$ gives rise to an outlier located near its \emph{classical location}
\begin{equation} \label{classical_loc}
\theta(d) \;\deq\; \phi^{1/2} + \phi^{-1/2} + d + d^{-1}\,.
\end{equation}
Any $d_i$ satisfying $\abs{d_i} < 1$ does not result in an outlier. We summarize this picture in Figure \ref{fig:outliers}.
\begin{figure}[ht!]
\begin{center}
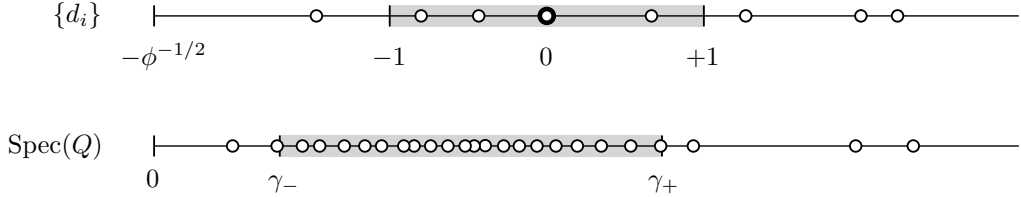
\end{center}
\caption{A typical configuration of $\{d_i\}$ (above) and the resulting spectrum of $Q$ (below). An order $M$ of the $d_i$'s are zero, which is symbolized by the thicker dot at $0$. Any $d_i$ inside the grey interval $[-1,1]$ does not give rise to an outlier, while any $d_i$ outside the grey interval gives rise to an outlier located near its classical location $\theta(d_i)$ and separated from the bulk $[\gamma_-, \gamma_+]$. \label{fig:outliers}} 
\end{figure}
The creation or annihilation of an outlier as a $d_i$ crosses $\pm 1$ is known as the \emph{BBP phase transition} \cite{BBP}. It takes place on the scale\footnote{We use the symbol $\asymp$ to denote quantities of comparable size; see ``Conventions'' at the end of this section for a precise definition.} $\absb{\abs{d_i} - 1}  \asymp  K^{-1/3}$. This scale has a simple heuristic explanation (we focus on the right edge of the spectrum). Suppose that $d_1 \in (0,1)$ and all other $d_i$'s are zero. Then the top eigenvalue $\mu_1$ exhibits universality, and fluctuates on the scale $K^{-2/3}$ around $\gamma_+$ (see Theorem \ref{thm:univ_H} and Remark \ref{rem:univ_Q} below). Increasing $d_1$ beyond the critical value $1$, we therefore expect $\mu_1$ to become an outlier when its classical location $\theta(d_1)$ is located at a distance greater than $K^{-2/3}$ from $\gamma_+$. By a simple Taylor expansion of $\theta$, the condition $\theta(d_1) - \gamma_+ \gg K^{-2/3}$ becomes $d_1 - 1 \gg K^{-1/3}$.

Next, we outline the distribution of the outlier eigenvectors. Let $\mu_i$ be an outlier with associated eigenvector $\f \xi_i$. Then $\f \xi_i$ is concentrated on a cone \cite{Paul,BGN,Mes} with axis parallel to $\f v_i$, the corresponding eigenvector of the population covariance matrix $\Sigma$. More precisely, assuming that the eigenvalue $\sigma_i = 1 + \phi^{1/2} d_i$ of $\Sigma$ is simple, we have\footnote{We use the symbol $\approx$ to denote approximate equality with high probability in heuristic statements. In the precise statements of Section \ref{sec: model_results}, it will be replaced by the more precise notion of stochastic domination from Definition \ref{def:stocdom}.}
\begin{equation} \label{cone_condition}
\scalar{\f v_i}{\f \xi_i}^2 \;\approx\; u(d_i)\,,
\end{equation}
where we defined
\begin{equation} \label{aperture}
u(d_i) \;\equiv\; u_\phi(d_i) \;\deq\; \frac{\sigma_i}{\phi^{1/2} \theta(d_i)} (1 - d_i^{-2})
\end{equation}
for $d_i > 1$.
The function $u$ determines the aperture $2 \arccos \sqrt{u(d_i)}$ of the cone. Note that $u(d_i) \in (0,1)$ and $u(d_i)$ converges to $1$ as $d_i \to \infty$. See Figure \ref{fig:cone}.
\begin{figure}[ht!]
\begin{center}
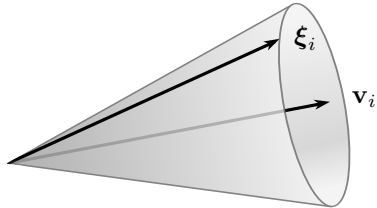
\end{center}
\caption{The eigenvector $\f \xi_i$ associated with an outlier $\mu_i$ is concentrated on a cone with axis parallel to $\f v_i$. The aperture of the cone is determined by $u(d_i)$ defined in \eqref{aperture}. \label{fig:cone}} 
\end{figure}

\subsection{Summary of previous related results} \label{sec:previous_work}
There is an extensive literature on spiked covariance matrices. So far most of the results have focused on the outlier eigenvalues of Example (1), with the nonzero $d_i$ independent of $N$ and $\phi$ fixed. Eigenvectors and the non-outlier eigenvalues have seen far less attention. 

For the uncorrelated case $\Sigma = I_M$ and Gaussian $X$ in \eqref{wtH_cov} with fixed $\phi$, it was proved in \cite{Joh1} for the complex case and in \cite{Johnstone} for the real case that the top eigenvalue, rescaled as $K^{2/3}(\mu_1 - \gamma_+)$, is asymptotically distributed according the Tracy-Widom law of the appropriate symmetry class \cite{TW1, TW2}. Subsequently, these results were shown to be universal, i.e.\ independent of the distribution of the entries of $X$, in \cite{Sosh2, PY}. The assumption that $\phi$ be fixed was relaxed in \cite{Pe1, ElK1}.

The study of covariance matrices with nontrivial population covariance matrix $\Sigma \neq I_M$ goes back to the seminal paper of Johnstone \cite{Johnstone}, where the Gaussian spiked model was introduced. The BBP phase transition was established by Baik, Ben Arous, and P\'ech\'e \cite{BBP} for complex Gaussian $X$, fixed rank of $\Sigma - I_M$, and fixed $\phi$. Subsequently, the results of \cite{BBP} were extended to the other Gaussian symmetry classes, such as real covariance matrices, in \cite{BV1,BV2}. The proofs of \cite{BBP, Pec} use an asymptotic analysis of Fredholm determinants, while those of \cite{BV1,BV2} use an explicit tridiagonal representation of $X X^*$; both of these approaches rely heavily on the Gaussian nature of $X$. See also \cite{PeBor} for a generalization of the BBP phase transition.

For the model from Example (1) with fixed nonzero $\{d_i\}$ and $\phi$, the almost sure convergence of the outliers was established in \cite{BS}. It was also shown in \cite{BS} that if $\abs{d_i} < 1$ for all $i$, the top eigenvalue $\mu_1$ converges to $\gamma_+$. For this model, a central limit theorem of the outliers was proved in \cite{BY1}. In \cite{BY2}, the almost sure convergence of the outliers was proved for a generalized spiked model whose population covariance matrix is of the block diagonal form $\Sigma = \diag (A,T)$, where $A$ is a fixed $r \times r$ matrix and $T$ is chosen so that the associated sample covariance matrix has no outliers.

In \cite{BGN}, the almost sure convergence of the projection of the outlier eigenvectors onto the finite-dimensional spike subspace was established, under the assumption that $\phi$ and the nonzero $d_i$ are fixed, and that $B$ and $T$ are both random and one of them is orthogonally invariant. In particular, the cone concentration from \eqref{cone_condition} was established in \cite{BGN}. In \cite{Paul}, under the assumption that $X$ is Gaussian and $\phi$ and the nonzero $d_i$ are fixed, a central limit theorem for a certain observable, the so-called sample vector, of the outlier eigenvectors was established. The result of \cite{Paul} was extended to non-Gaussian entries for a special class of $\Sigma$ in \cite{Shi}.

Moreover, in \cite{BGN2,NA1} result analogous to those of \cite{BGN} were obtained for the model from Example (2). Finally, a related class of models, so-called deformed Wigner matrices, have been the subject of much attention in recent years; we refer to \cite{SoshPert, SoshPert2, KY2, KY3} for more details; in particular, the joint distribution of all outliers was derived in \cite{KY3}.

\subsection{Overview of results} \label{sec:outl_res}

In this subsection we give an informal overview of our results. % on $Q$ defined in \eqref{wtH_cov}.
%We consider the general model $Q$ from Section \ref{sec: model}.
%For definiteness, we outline our results for $Q = \phi^{-1/2} \cal Q$, but they also hold verbatim for $\dot Q = \phi^{-1/2} \dot{\cal Q}$.

We establish results on the eigenvalues $\mu_i$ and the eigenvectors $\f \xi_i$ of $Q$. Our results consist of large deviation bounds and asymptotic laws. We believe that all of our large deviation bounds from Theorems \ref{thm: outlier locations}, \ref{thm:sticking}, \ref{thm: outlier eigenvectors}, \ref{thm: outlier eigenvectors 2}, and \ref{thm: non-outlier bound} are optimal (up to the technical conditions in the definition of $\prec$ given in Definition \ref{def:stocdom}). We do not prove this. However, we expect that, combining our method with the techniques of \cite{KY3}, one may also derive the asymptotic laws of all quantities on which we establish large deviation bounds, in particular proving the optimality of our large deviation bounds.

Our results on the eigenvalues of $Q$ consist of two parts. First, we derive large deviation bounds on the locations of the outliers (Theorem \ref{thm: outlier locations}). Second, we prove \emph{eigenvalue sticking} for the non-outliers (Theorem \ref{thm:sticking}), whereby each non-outlier ``sticks'' with high probability and very accurately to the eigenvalues of a related covariance matrix satisfying $\Sigma = I_M$ and whose top eigenvalues exhibit universality. As a corollary (Remark \ref{rem:univ_Q}), we prove that the top non-outlier eigenvalue of $Q$ has asymptotically the Tracy-Widom-1 distribution. This sticking is very accurate if all $d_i$'s are separated from the critical point $1$, and becomes less accurate if a $d_i$ is in the vicinity of $1$. Eventually, it breaks down precisely on the BBP transition scale $\abs{d_i - 1} \asymp K^{-1/3}$, at which the Tracy-Widom-1 distribution is known not to hold for the top non-outlier eigenvalue. These results generalize those from \cite[Theorem 2.7]{KY2}.

Next, we outline our results for the eigenvectors $\f \xi_i$ of $Q$. We consider the \emph{generalized components} $\scalar{\f w}{\f \xi_i}$ of $\f \xi_i$, where $\f w \in \R^M$ is an arbitrary deterministic vector. In our first result on the eigenvectors (Theorems \ref{thm: outlier eigenvectors} and \ref{thm: outlier eigenvectors 2}), we establish large deviation bounds on the generalized components of outlier eigenvectors (and, more generally, of the outlier spectral projections defined in \eqref{def_P_A} below). This result gives a quantitative version of the cone concentration from \eqref{cone_condition}, which in particular allows us to track the strength of the concentration in the vicinity of the BBP transition and for overlapping outliers. Our results also establish the complete delocalization of an outlier eigenvector $\f \xi_i$ in any direction orthogonal to the spike direction $\f v_i$, provided the outlier $\mu_i$ is well separated from the bulk spectrum and other outliers. We say that the vector $\f \xi_i$ is \emph{completely delocalized, or unbiased, in the direction $\f w$} if $\scalar{\f w}{\f \xi_i}^2 \prec M^{-1}$, where ``$\prec$'' denotes a high probability bound up to powers of $M^\epsilon$ (see Definition \ref{def:stocdom}).

If the outlier $\mu_i$ approaches the bulk spectrum or another outlier, the cone concentration becomes less accurate. For the case of two nearby outlier eigenvalues, for instance, the cone concentration \eqref{cone_condition} of the eigenvectors breaks down when the distributions of the outlier eigenvalues have a nontrivial overlap. In order to understand this behaviour in more detail, we introduce the deterministic projection
\begin{equation} \label{def_Pi_A}
\Pi_A \;\deq\; \sum_{i \in A} \f v_i \f v_i^*\,,
\end{equation}
where $A \subset \{1, \dots, M\}$. Then the cone concentration from \eqref{cone_condition} may be written as $\abs{\Pi_{\{i\}} \f \xi_i}^2 \approx u(d_i) \abs{\f \xi_i}^2$. In contrast, in the degenerate case $d_1 = d_2 > 1$ and all other $d_i$'s being zero, \eqref{cone_condition} is replaced with
\begin{equation} \label{cone_condition_deg}
\scalar{\f \xi_i}{\Pi_{\{1,2\}} \f \xi_j} \;\approx\; \delta_{ij} \, u(d_1) \abs{\f \xi_i} \abs{\f \xi_j}\,,
\end{equation}
where $i,j \in \{1,2\}$. We deduce that each $\f \xi_i$ lies on the cone
\begin{equation} \label{cone2}
\abs{\Pi_{\{1,2\}} \f \xi_i}^2 \;\approx\; u(d_1) \abs{\f \xi_i}^2\,,
\end{equation}
and that $\Pi_{\{1,2\}} \f \xi_1 \perp \Pi_{\{1,2\}} \f \xi_2$. Moreover, we prove that $\f \xi_i$ is completely delocalized in any direction orthogonal to $\f v_1$ and $\f v_2$. The interpretation is that $\f \xi_1$ and $\f \xi_2$ both lie on the cone \eqref{cone2}, that they are orthogonal on both the range and null space of $\Pi_{\{1,2\}}$, and that beyond these constraints their distribution is unbiased (i.e.\ isotropic). Finally, we note that the preceding discussion remains unchanged if one interchanges $\f \xi_i$ and $\f v_i$. We refer to Example \ref{ex:degen} below for more details.

In our second result on the eigenvectors (Theorem \ref{thm: non-outlier bound}), we establish delocalization bounds for the generalized components of non-outlier eigenvectors $\xi_i$. In particular, we prove complete delocalization of non-outlier eigenvectors in directions orthogonal to any spike $\f v_j$ whose value $d_j$ is near the critical point 1. In addition, we prove that non-outlier eigenvectors away from the edge are completely delocalized in all directions. The complete delocalization in the direction $\f v_j$ breaks down if $\abs{d_j - 1} \ll 1$. The interpretation of this result is that any spike $d_j$ near the BBP transition point $1$ causes all non-outlier eigenvectors $\f \xi_i$ near the upper edge of the bulk spectrum to have a bias in the direction $\f v_j$, in contrast to the completely delocalized case where $\f \xi_i$ is uniformly distributed on the unit sphere.

In our final result on the eigenvectors (Theorem \ref{thm: eigenvector law}), we give the asymptotic law of the generalized component $\scalar{\f w}{\f \xi_i}$ of a non-outlier eigenvector $\f \xi_i$. In particular, we prove that this generalized component is asymptotically Gaussian and has a variance predicted by the delocalization bounds from Theorem \ref{thm: non-outlier bound}. For instance, we prove that if $\abs{d_j - 1} \gg K^{-1/3}$ then
\begin{equation} \label{ev_bias}
\scalar{\f v_j}{\f \xi_i}^2 \;=\; \frac{\sigma_i}{M (d_j - 1)^2} \Theta\,,
\end{equation}
for all non-outlier indices $i$ that are not too large (see Theorem \ref{thm: eigenvector law} for a precise statement). Here $\Theta$ is a random variable that converges in distribution to a chi-squared variable. If $\f \xi_i$ were completely delocalized in the direction $\f v_j$, the right-hand side would be of order $M^{-1}$. Suppose for simplicity that $\phi$ is of order one. The bias of $\f \xi_i$ in the direction $\f v_j$ emerges as soon as $\abs{d_j - 1} \ll 1$, and reaches a magnitude of order $M^{-1/3}$ for $d_j$ near the BBP transition. This is much larger than the unbiased $M^{-1}$. Note that this phenomenon applies simultaneously to all non-outlier eigenvectors near the right edge: the right-hand side of \eqref{ev_bias} does not depend on $i$. Note also that the right-hand side of \eqref{ev_bias} is insensitive to the sign of $d_j - 1$. In particular, the bias is also present for \emph{subcritical} spikes. We conclude that even subcritical spikes are observable in the principal components. In contrast, if one only considers the eigenvalues of the principal components, the subcritical spikes cannot be detected; this follows from the eigenvalue sticking result in Theorem \ref{thm:sticking}.

Finally, the proofs of universality of the non-outlier eigenvalues and eigenvectors require the universality of $Q$ for the uncorrelated case $\Sigma = I_M$ as input. This universality result is given in Theorem \ref{thm:univ_H}, which is also of some independent interest. It establishes the joint, fixed-index, universality of the eigenvalues and eigenvectors of $Q$ (and hence, as a special case, the quantum unique ergodicity of the eigenvectors of $Q$ mentioned in Section \ref{sec: uncorr}). It works for all eigenvalue indices $i$ satisfying $i \leq K^{1 - \tau}$ for any fixed $\tau > 0$.

We conclude this subsection by outlining the key novelties of our work.
\begin{enumerate}
\item
We introduce the general models $Q$ from \eqref{wtH_cov} and $\dot Q$ from \eqref{def_dot_Q} below, which subsume and generalize several models considered previously in the literature\footnote{In particular, the current paper is the first to study the principal components of a realistic sample covariance matrix \eqref{def_calQ_mean} instead of the zero mean case \eqref{def_calQ}.}. We allow the entries of $X$ to be arbitrary random variables (up to a technical assumption on their tails). All quantities except $r$ and the rank of $\Sigma - I_M$ may depend on $N$. We make no assumption on $T$ beyond the bounded-rank condition of $T T^* - I_M$. The dimensions $M$ and $N$ may be wildly different, and are only subject to the technical condition \eqref{NM gen}.

\item
We study the behaviour of the principal components of $Q$ near the BBP transition and when outliers collide. Our results hold for generalized components $\scalar{\f w}{\f \xi_i}$ of the eigenvectors in arbitrary directions $\f w$.

\item
We obtain quantitative bounds (i.e.\ rates of convergence) on the outlier eigenvalues  and the generalized components of the eigenvectors. We believe these bounds to be optimal.  

\item
We obtain precise information about the non-outlier principal components. A novel observation is that, provided there is a $d_i$ satisfying $\abs{d_i - 1} \ll 1$ (i.e.\ $Q$ is near the BBP transition), all non-outlier eigenvectors near the edge will be biased in the direction of $\f v_i$. In particular, non-outlier eigenvectors, unlike non-outlier eigenvalues, retain some information about the subcritical spikes of $\Sigma$.

\item
We establish the joint, fixed-index, universality of the eigenvalues and eigenvectors for the case $\Sigma = I_M$. This result holds for any eigenvalue indices $i$ satisfying $i \leq K^{1 - \tau}$ for an arbitrary $\tau > 0$.
Note that previous works \cite{KY1, TV3} (established in the context of Wigner matrices) required either the much stronger condition $i \leq (\log K)^{C \log \log K}$ or a four-moment matching condition.
\end{enumerate}

We remark that the large deviation bounds derived in this paper also allow one to derive the joint distribution of the generalized components of the outlier eigenvectors; this will be the subject of future work.

\subsection*{Conventions}

The fundamental large parameter is $N$. All quantities that are not explicitly constant may depend on $N$; we almost always omit the argument $N$ from our notation.

We use $C$ to denote a generic large positive constant, which may depend on some fixed parameters and whose value may change from one expression to the next. Similarly, we use $c$ to denote a generic small positive constant.  For two positive quantities $A_N$ and $B_N$ depending on $N$ we use the notation $A_N \asymp B_N$ to mean $C^{-1} A_N \leq B_N \leq C A_N$ for some positive constant $C$. For $a < b$ we set $\qq{a,b} \deq [a,b] \cap \Z$. We use the notation $\f v = (v(i))_{i = 1}^M$ for vectors in $\R^M$, and denote by $\abs{\cdot}= \norm{\cdot}_2$ the Euclidean norm of vectors and by $\norm{\cdot}$ the corresponding operator norm of matrices. We use $I_M$ to denote the $M \times M$ identity matrix, which we also sometimes write simply as $1$ when there is no risk of confusion.

We use $\tau > 0$ in various assumptions to denote a positive constant that may be chosen arbitrarily small. A smaller value of $\tau$ corresponds to a weaker assumption. All of our estimates depend on $\tau$, and we neither indicate nor track this dependence.

\section{Results} \label{sec: model_results}

In this section we state our main results.
The following notion of a high-probability bound was introduced in \cite{EKY2}, and has been subsequently used in a number of works on random matrix theory. It provides a simple way of systematizing and making precise statements of the form ``$A$ is bounded with high probability by $B$ up to small powers of $N$''.

\begin{definition}[Stochastic domination]\label{def:stocdom}
Let
\begin{equation*}
A = \pb{A^{(N)}(u) \col N \in \N, u \in U^{(N)}} \,, \qquad
B = \pb{B^{(N)}(u) \col N \in \N, u \in U^{(N)}}
\end{equation*}
be two families of nonnegative random variables, where $U^{(N)}$ is a possibly $N$-dependent parameter set. 
We say that $A$ is \emph{stochastically dominated by $B$, uniformly in $u$,} if for all (small) $\epsilon > 0$ and (large) $D > 0$ we have
\begin{equation} \label{prec_def}
\sup_{u \in U^{(N)}} \P \qB{A^{(N)}(u) > N^\epsilon B^{(N)}(u)} \;\leq\; N^{-D}
\end{equation}
for large enough $N\geq N_0(\e, D)$.
Throughout this paper the stochastic domination will always be uniform in all parameters (such as matrix indices) that are not explicitly fixed. Note that $N_0(\e, D)$ may depend on the constants from \eqref{NM gen} and \eqref{moments of X-1} as well as any constants fixed in the assumptions of our main results.
If $A$ is stochastically dominated by $B$, uniformly in $u$, we use the notation $A \prec B$. Moreover, if for some complex family $A$ we have $\abs{A} \prec B$ we also write $A = O_\prec(B)$.
\end{definition}

\begin{remark} \label{rem:stochdomMN} 
Because of \eqref{NM gen}, all (or some) factors of $N$ in Definition \eqref{def:stocdom} could be replaced with $M$ without changing the definition of stochastic domination. 
\end{remark}

%The remainder of this section is devoted to main results of this paper.
%As explained in Section \ref{sec:intro}, we first state our results for the matrix $Q$ from \eqref{wtH_cov}, bearing in mind that they also hold for $\dot Q$ The precise statement for $\dot Q$ is given in Theorem \ref{thm: Q dot} below.

\subsection{Eigenvalue locations} \label{sec:eval_results}

We begin with results on the locations of the eigenvalues of $Q$. These results will also serve as a fundamental input for the proofs of the results on eigenvectors presented in Sections \ref{sec:out_results} and \ref{sec:non_outl_results}.

Recall that $Q$ has $M - K$ zero eigenvalues. We shall therefore focus on the $K$ nontrivial eigenvalues $\mu_1 \geq \cdots \geq \mu_K$ of $Q$. On the global scale, the eigenvalues of $Q$ are distributed according to the Marchenko-Pastur law \eqref{MP_law}. This may be easily inferred from the fact that \eqref{MP_law} gives the global density of the eigenvalues for the uncorrelated case $\Sigma = I_M$, combined with eigenvalue interlacing (see Lemma \ref{lem: interlacing} below). In this section we focus on \emph{local} eigenvalue information.

We introduce the set of \emph{outlier indices}
\begin{equation} \label{def: outlier indices}
\cal O \;\deq\; \hb{i \in \cal R \col \abs{d_i} \geq 1 + K^{-1/3}}\,.
\end{equation}
As explained in Section \ref{sec:beh_ev}, each $i \in \cal O$ gives rise to an outlier of $Q$ near the classical location $\theta(d_i)$ defined in \eqref{classical_loc}. In the definition \eqref{def: outlier indices}, the lower bound $1 + K^{-1/3}$ is chosen for definiteness; it could be replaced with $1 + a K^{-1/3}$ for any fixed $a > 0$. We denote by
\begin{equation} \label{def_r_pm}
s_{\pm} \;\deq\; \absb{\hb{i \in \cal O \col \pm d_i > 0}}
\end{equation}
the number of outliers to the left ($s_-$) and right ($s_+$) of the bulk spectrum.

For $d \in \cal D \setminus [-1,1]$ we define
\begin{equation*}
\Delta(d) \;\deq\;
\begin{cases}
\frac{\phi^{1/2} \theta(d)}{1 + (\abs{d} - 1)^{-1/2}} & \text{if } -\phi^{-1/2} < d < -1
\\
(d - 1)^{1/2} & \text{if } 1 < d \leq 2
\\
1 + \frac{d}{1 + \phi^{-1/2}} & \text{if } d \geq 2\,.
\end{cases}
\end{equation*}
The function $\Delta(d)$ will be used to give an upper bound on the magnitude of the fluctuations of an outlier associated with $d$. We give such a precise expression for $\Delta$ in order to obtain sharp large deviation bounds for all $d \in \cal D \setminus [-1,1]$. (Note that the discontinuity of $\Delta$ at $d = 2$ is immaterial since $\Delta$ is used as an upper bound with respect to $\prec$. The ratio of the right- and left-sided limits at $2$ of $\Delta$ lies in $[1,3]$.)

Our result on the outlier eigenvalues is the following.

\begin{theorem}[Outlier locations] \label{thm: outlier locations}
Fix $\tau > 0$. Then for $i \in \cal O$ we have the estimate
\begin{equation} \label{outlier locations}
\abs{\mu_{i} - \theta(d_i)} \;\prec\; \Delta(d_i) \, K^{-1/2}
\end{equation}
provided that $d_i > 0$ or $\abs{\phi - 1} \geq \tau$.

Furthermore, the extremal non-outliers $\mu_{s_+ + 1}$ and $\mu_{K - s_-}$ satisfy
\begin{equation} \label{ex+ bound}
\abs{\mu_{s_++1} - \gamma_+} \;\prec\; K^{-2/3}\,,
\end{equation}
and, assuming in addition that $\abs{\phi-1} \geq \tau$,
\begin{equation} \label{ex- bound}
\abs{\mu_{K-s_-} - \gamma_-} \;\prec\; K^{-2/3}\,. 
\end{equation}
\end{theorem}

\begin{remark}
Theorem \ref{thm: outlier locations} gives large deviation bounds for the locations of the outliers to the right of the bulk.
Since $\tau > 0$ may be arbitrarily small, Theorem \ref{thm: outlier locations} also gives the full information about the outliers to the left of the bulk except in the case $1 > \phi = 1 + o(1)$. Although our methods may be extended to this case as well, we exclude it here to avoid extraneous complications.
\end{remark}

\begin{remark}
By definition of $s_-$ and $\cal D$, if $\phi>1$ then $s_-=0$. Hence, by \eqref{ex- bound}, if $\phi > 1$ there are no outliers on the left of the bulk spectrum.
\end{remark}

\begin{remark}
Previously, the model from Example (1) in Section \ref{sec:examples} with fixed nonzero $\{d_i\}$ and $\phi$ was investigated in \cite{BS, BY1}. In \cite{BS}, it was proved that each outlier eigenvalue $\mu_i$ with $i \in \cal O$ convergences almost surely to $\theta(d_i)$. Moreover, a central limit theorem for $\mu_i$ was established in \cite{BY1}.
\end{remark}

The locations of the non-outlier eigenvalues $\mu_i$, $i \notin \cal O$, are governed by \emph{eigenvalue sticking}, whereby the eigenvalues of $Q$ ``stick'' with high probability to eigenvalues of a reference matrix which has a trivial population covariance matrix. The reference matrix is $Q$ from \eqref{wtH_cov} with uncorrelated entries. More precisely, we set
\begin{equation} \label{def_H}
H \;\deq\; Y Y^*\,, \qquad Y \;\deq\; (I_M, 0) O X\,,
\end{equation}
where $O \equiv O(T) \in \r O(M + r)$ is a deterministic orthogonal matrix. It is easy to check that $\E H = (NM)^{-1/2} I_M$, so that $H$ corresponds to an uncorrelated population. The matrix $O(T)$ is explicitly given in \eqref{O_dep_T} below. In fact, in Theorem \ref{thm:univ_H} below we prove the universality of the joint  distribution of non-bulk eigenvalues and eigenvectors of $H$. Here, by definition, we say that an index $i \in \qq{1,K}$ is \emph{non-bulk} if $i \notin \qq{K^{1 - \tau}, K - K^{1 - \tau}}$ for some fixed $\tau > 0$.  In particular, the asymptotic distribution of the non-bulk eigenvalues and eigenvectors of $H$ does not depend on the choice of $O$. Note that for the special case $r = 0$ the eigenvalues of $H$ coincide with those of $X X^*$. We denote by
\begin{equation*}
\lambda_1\;\geq\; \lambda_2 \;\geq\; \cdots \;\geq\; \lambda_M
\end{equation*}
the eigenvalues of $H$.

\begin{theorem}[Eigenvalue sticking] \label{thm:sticking}
Define
\begin{equation} \label{def_alpha}
\alpha_\pm \;\deq\; \min_{1 \leq i \leq M} \abs{d_i \mp 1}\,.
\end{equation}
Fix $\tau > 0$. 
Then we have for all $i \in \qq{1, (1 - \tau) K}$
\begin{equation} \label{sticking_1}
\abs{\mu_{i + s_+} - \lambda_i} \;\prec\; \frac{1}{K \alpha_+}\,.
\end{equation}
Similarly, if $\abs{\phi - 1} \geq \tau$ then we have for all $i \in \qq{\tau K, K}$
\begin{equation} \label{sticking_2}
\abs{\mu_{i - s_-} - \lambda_i} \;\prec\; \frac{1}{K \alpha_-}\,.
\end{equation}
\end{theorem}

\begin{remark} \label{rem: sticking_univ}
As outlined above, in Theorem \ref{thm:univ_H} below we prove that the asymptotic joint distribution of the non-bulk eigenvalues of $H$ is universal, i.e.\ it coincides with that of the Wishart matrix $H_{\txt{Wish}} = X X^*$ with $r = 0$ and $X$ Gaussian.
As an immediate corollary of Theorems \ref{thm:sticking} and \ref{thm:univ_H}, we obtain the universality of the non-outlier eigenvalues of $Q$ with index $i \leq K^{1 - \tau} \alpha_+^3$. This condition states simply that the right-hand side of \eqref{sticking_1} is much smaller than the scale on which the eigenvalue $\lambda_i$ fluctuates, which is $K^{-2/3} i^{-1/3}$. See Remark \ref{rem:univ_Q} below for a precise statement.
\end{remark}

\begin{remark} \label{rem:com_old_ev}
Theorem \ref{thm:sticking} is analogous to Theorem 2.7 of \cite{KY2}, where sticking was first established for Wigner matrices. Previously, eigenvalue sticking was established for a certain class of random perturbations of Wigner matrices in \cite{BGGM1,BGGM2}. We refer to \cite[Remark 2.8]{KY2} for a more detailed discussion.

Aside from holding for general covariance matrices of the form \eqref{wtH_cov}, Theorem \ref{thm:sticking} is stronger than its counterpart from \cite{KY2} because it holds much further into the bulk: in \cite[Theorem 2.7]{KY2}, sticking was established under the assumption that $i \leq (\log K)^{C \log \log K}$.
\end{remark}

\begin{remark}
The edge universality following from Theorem \ref{thm:sticking} (as explained in Remark \ref{rem: sticking_univ}) generalizes the recent result \cite{BaoPanZhou2}. There, for the model from Example (1) in Section \ref{sec:examples} with fixed nonzero $\{d_i\}$ and $\phi$, it was proved that if $d_i < 1$ for all $i$ and $\Sigma$ is diagonal, then $\mu_1$ converges (after a suitable affine transformation) in distribution to the Tracy-Widom-1 distribution.
\end{remark}

\subsection{Outlier eigenvectors}\label{sec:out_results}
We now state our main results for the outlier eigenvectors.  Statements of results on eigenvectors requires some care, since there is some arbitrariness in the definition of the eigenvector $\f \xi_i$ of $Q$. In order to get rid of the arbitrariness in the sign (or, in the complex case, the phase) of $\f \xi_i$ we consider products of generalized components,
\begin{equation*}
\scalar{\f v}{\f \xi_i} \scalar{\f \xi_i}{\f w}\,.
\end{equation*}
It is easy to check that these products characterize the eigenvector $\f \xi_i$ completely, up to the ambiguity of a global sign (or phase). More generally, one may consider the generalized components $\scalar{\f v}{(\cdot) \, \f w}$ of the (random) \emph{spectral projection}
\begin{equation} \label{def_P_A}
P_A \;\deq\; \sum_{i \in A} \f \xi_{i} \f \xi_{i}^*\,,
\end{equation}
where $A \subset \cal O$.

In the simplest case $A = \{i\}$ the generalized components of $P_A$ characterize the generalized components of $\f \xi_i$. The need to consider higher-dimensional projections arises if one considers degenerate or almost degenerate outliers. Suppose for example that $d_1 \approx d_2$ and all other $d_i$'s are zero. Then the cone concentration \eqref{cone_condition} fails, to be replaced with \eqref{cone_condition_deg}. The failure of the cone concentration is also visible in our results as a blowup of the error bounds. This behaviour is not surprising, since for degenerate outliers $d_1 = d_2$ it makes no sense to distinguish the associated spike eigenvectors $\f v_1$ and $\f v_2$; only the eigenspace matters. Correspondingly, we have to consider the orthogonal projection onto the eigenspace of the outliers in $A$. See Example \ref{ex:degen} below for a more detailed discussion.

For $i \in \qq{1,M}$ we define $\nu_i \geq 0$ through
\begin{equation*}
\nu_i \;\equiv\; \nu_i(A) \;\deq\;
\begin{cases}
\min_{j \notin A} \abs{d_i - d_j} & \text{if } i \in A
\\
\min_{j \in A} \abs{d_i - d_j} & \text{if } i \notin A\,.
\end{cases}
\end{equation*}
In other words, $\nu_i(A)$ is the distance from $d_i$ to either $\{d_i\}_{i \in A}$ or $\{d_i\}_{i \notin A}$, whichever it does not belong to. For a vector $\f w \in \R^M$ we also introduce the shorthand
\begin{equation*}
w_i \;\deq\; \scalar{\f v_i}{\f w}
\end{equation*}
to denote the components of $\f w$ in the eigenbasis of $\Sigma$.

For definiteness, we only state our results for the outliers on the right-hand side of the bulk spectrum. Analogous results hold for the outliers on the left-hand side. Since the behaviour of the fluctuating error term is different in the regimes $\mu_i - \gamma_+ \ll 1$ (near the bulk) and $\mu_i - \gamma_+ \gg 1$ (far from the bulk), we split these two cases into separate theorems.

\begin{theorem}[Outlier eigenvectors near bulk] \label{thm: outlier eigenvectors}
Fix $\tau > 0$. Suppose that $A \subset \cal O$ satisfies $1 + K^{-1/3} \leq d_i \leq \tau^{-1}$ for all $i \in A$. Define the deterministic positive quadratic form
\begin{equation*}
\scalar{\f w}{Z_A \f w} \;\deq\; \sum_{i \in A} u(d_i) w_i^2\,,
\end{equation*}
where we recall the definition \eqref{aperture} of $u(d_i)$.
Then for any deterministic $\f w \in \R^M$ we have
\begin{equation} \label{outlier evectors 1}
\scalar{\f w}{P_A \f w} \;=\; \scalar{\f w}{Z_A \f w}
+ O_\prec \qa{
\sum_{i \in A} \frac{w_i^2}{M^{1/2}(d_i - 1)^{1/2}} + \sum_{i = 1}^M \frac{\sigma_i w_i^2}{M \nu_i(A)^2}
+
\scalar{\f w}{Z_A \f w}^{1/2} \pBB{\sum_{i \notin A} \frac{\sigma_i w_i^2}{M \nu_i(A)^2}}^{1/2}}\,.
\end{equation}
Note that the last error term is zero if $\f w$ is in the subspace $\spn \{\f v_i\}_{i \in A}$ or orthogonal to it.
\end{theorem}

\begin{remark} \label{rem:off-diag_outlier}
Theorem \ref{thm: outlier eigenvectors} may easily also be stated for more general quantities of the form $\scalar{\f v}{P_A \f w}$. We omit the precise statement; it is a trivial corollary of \eqref{vi_vj_result 2} below, which holds under the assumptions of Theorem \ref{thm: outlier eigenvectors}.
\end{remark}

We emphasize that the set $A$ in Theorem \ref{thm: outlier eigenvectors} may be chosen at will. If all outliers are well-separated, then the choice $A = \{i\}$ gives the most precise information. However, as explained at the beginning of this subsection, the indices of outliers that are close to each other should be included in the same set $A$. Thus, the freedom to chose $\abs{A} \geq 2$ is meant for degenerate or almost degenerate outliers. (In fact, as explained after \eqref{cone_condition2} below, the correct notion of closeness of outliers is that of \emph{overlapping}.)

We consider a few examples.

\begin{example} \label{ex:1}
Let $A = \{i\}$ and $\f w = \f v_i$. Then we get from \eqref{outlier evectors 1}
\begin{equation} \label{Pi_example1}
\scalar{\f v_i}{\f \xi_i}^2 \;=\; u(d_i) + O_\prec \qBB{\frac{1}{M^{1/2} (d_i - 1)^{1/2}} + \frac{\sigma_i}{M \nu_i^2}}\,.
\end{equation}
This gives a precise version of the cone concentration from \eqref{cone_condition}. Note that the cone concentration holds provided the error is much smaller than the main term $u(d_i)$, which leads to the conditions
\begin{equation} \label{cone_condition2}
d_i - 1 \;\gg\; K^{-1/3} \qquad \txt{and} \qquad \nu_i \;\gg\; (d_i - 1)^{-1/2} K^{-1/2}\,;
\end{equation}
here we used that $d_i \asymp 1$ and $M \asymp (1 + \phi) K$.

We claim that both conditions in \eqref{cone_condition2} are natural and necessary. The first condition of \eqref{cone_condition2} simply means that $\mu_i$ is an outlier. The second condition of \eqref{cone_condition2} is a \emph{non-overlapping} condition. To understand it, recall from \eqref{outlier locations} that $\mu_i$ fluctuates on the scale $(d_i - 1)^{1/2} K^{-1/2}$. Then $\mu_i$ is a non-overlapping outlier if all other outliers are located with high probability at a distance greater than this scale from $\mu_i$. Recalling the definition of the classical location $\theta(d_i)$ of $\mu_i$, the non-overlapping condition becomes
\begin{equation} \label{non-overlapping1}
\min_{j \in \cal O \setminus \{i\}} \abs{\theta(d_j) - \theta(d_i)} \;\gg\; (d_i - 1)^{1/2} K^{-1/2}\,.
\end{equation}
After a simple estimate using the definition of $\theta$, we find that this is precisely the second condition of \eqref{cone_condition2}. The degeneracy or almost degeneracy of outliers discussed at the beginning of this subsection is hence to be interpreted more precisely in terms of overlapping of outliers. 

Provided $\mu_i$ is well-separated from both the bulk spectrum and the other outliers, we find that the error in \eqref{Pi_example1} is of order $M^{-1/2}$.
\end{example}

\begin{example} \label{ex:2}
Take $A = \{i\}$ and $\f w = \f v_j$ with $j \neq i$. Then we get from \eqref{outlier evectors 1}
\begin{equation} \label{del_orth}
\scalar{\f v_j}{\f \xi_i}^2 \;\prec\; \frac{\sigma_j}{M (d_i - d_j)^2}\,.
\end{equation}
Suppose for simplicity that $\phi \asymp 1$. Then, under the condition that $\abs{d_i - d_j} \asymp 1$, we find that $\f \xi_i$ is completely delocalized in the direction $\f v_j$. In particular, if $\nu_i \asymp 1$ then $\f \xi_i$ is completely delocalized in any direction orthogonal to $\f v_i$.

As $d_j$ approaches $d_i$ the delocalization bound from \eqref{del_orth} deteriorates, and eventually when $\mu_i$ and $\mu_j$ start overlapping, i.e.\ the second condition of \eqref{cone_condition2} is violated, the right-hand side of \eqref{del_orth} has the same size as the leading term of \eqref{Pi_example1}. This is again a manifestation of the fact that the individual eigenspaces of overlapping outliers cannot be distinguished.
\end{example}

\begin{example}\label{ex:degen}
Suppose that we have an $\abs{A}$-fold degenerate outlier, i.e.\ $d_i = d_j$ for all $i,j \in A$. Then from Theorem \ref{thm: outlier eigenvectors} and Remark \ref{rem:off-diag_outlier} (see the estimate \eqref{vi_vj_result 2}) we get, for all $i,j \in A$,
\begin{equation*}
\scalar{\f v_i}{P_A \f v_j} \;=\; \delta_{ij} u(d_i) + O_{\prec} \qBB{\frac{1}{M^{1/2} (d_i - 1)^{1/2}} + \frac{\sigma_i}{M \nu_i(A)^2}}\,.
\end{equation*}
Defining the $\abs{A} \times \abs{A}$ random matrix $M = (M_{ij})_{i,j \in A}$ through $M_{ij} \deq \scalar{\f v_i}{\f \xi_j}$, we may write the left-hand side as $(M M^*)_{ij}$. 
We conclude that $u(d_i)^{-1/2} M$ is approximately orthogonal, from which we deduce that $u(d_i)^{-1/2} M^*$ is also approximately orthogonal. In other words, we may interchange the families $\{\f v_i\}_{i \in A}$ and $\{\f \xi_i\}_{i \in A}$.
More precisely, we get
\begin{equation*}
\scalar{\f \xi_i}{\Pi_A \f \xi_j} \;=\; (M^* M)_{ij} \;=\; \delta_{ij} u(d_i) + O_{\prec} \qBB{\frac{1}{M^{1/2} (d_i - 1)^{1/2}} + \frac{\sigma_i}{M \nu_i(A)^2}}\,.
\end{equation*}
This is the correct generalization of \eqref{Pi_example1} from Example \ref{ex:1} to the degenerate case. The error term is the same as in \eqref{Pi_example1}, and its size and relation to the main term is exactly the same as in Example \ref{ex:1}. Hence the discussion following \eqref{Pi_example1} may be take over verbatim to this case.

In addition, analogously to Example \ref{ex:2}, for $i \in A$ and $j \notin A$ we find that \eqref{del_orth} remains true. This establishes the delocalization of $\f \xi_i$ in any direction within the null space of $\Pi_A$.

These estimates establish the general cone concentration, with optimal rate of convergence, for degenerate outliers outlined around \eqref{cone2}. The eigenvectors $\{\f \xi_i\}_{i \in A}$ are all concentrated on the cone defined by $\abs{\Pi_A \f \xi}^2 = u(d_i) \abs{\f \xi}^2$ (for some immaterial $i \in A$). Moreover, the eigenvectors $\{\f \xi_i\}_{i \in A}$ are orthogonal on both the range and null space of $\Pi_A$. Provided that the group $\{d_i\}_{i \in A}$ is well-separated from $1$ and all other $d_i$'s, the eigenvectors $\{\f \xi_i\}_{i \in A}$ are completely delocalized on the null space of $\Pi_A$.

We conclude this example by remarking that a similar discussion also holds for a group of outliers that is not degenerate, but nearly degenerate, i.e.\ $\abs{d_i - d_j} \ll \abs{d_i - d_k}$ for all $i,j \in A$ and $k \notin A$. We omit the details.
\end{example}

The next result is the analogue of Theorem \ref{thm: outlier eigenvectors} for outliers far from the bulk.

\begin{theorem}[Outlier eigenvectors far from bulk] \label{thm: outlier eigenvectors 2}
Fix $\tau > 0$. Suppose that $A \subset \cal O$ satisfies $d_i \geq 1 + \tau$ for all $i \in A$, and that there exists a positive $d_A$ such that $\tau d_A \leq d_i \leq \tau^{-1} d_A$ for all $i \in A$. Then for any deterministic $\f w \in \R^M$ we have
\begin{multline} \label{outlier evectors 2}
\scalar{\f w}{P_A \f w} \;=\; \scalar{\f w}{Z_A \f w}
+ O_\prec \Biggl[\frac{1}{M^{1/2}(\phi^{1/2} + d_A)} \sum_{i \in A} \sigma_i w_i^2 + \pbb{1 + \frac{\phi^{1/2} d_A^2}{\phi^{1/2} + d_A}} \sum_{i = 1}^M \frac{\sigma_i w_i^2}{M \nu_i(A)^2}
\\
+ \frac{d_A}{\phi^{1/2} + d_A}\pBB{\sum_{i \in A} \sigma_i w_i^2}^{1/2} \pBB{\sum_{i \notin A} \frac{\sigma_i w_i^2}{M \nu_i(A)^2}}^{1/2} \Biggr]\,.
\end{multline}
\end{theorem}

We leave the discussion on the interpretation of the error in \eqref{outlier evectors 2} to the reader; it is similar to that of Examples \ref{ex:1}, \ref{ex:2}, and \ref{ex:degen}.

\subsection{Non-outlier eigenvectors}\label{sec:non_outl_results}
In this subsection we state our results on the \emph{non-outlier} eigenvectors, i.e.\ on $\f \xi_a$ for $a \notin \cal O$.
Our first result is a delocalization bound. In order to state it, we define for $a \in \qq{1,K}$ the \emph{typical distance from $\mu_a$ to the spectral edges} $\gamma_\pm$ through
\begin{equation} \label{def_kappa_a}
\kappa_a \;\deq\; K^{-2/3} (a \wedge (K + 1 - a))^{2/3}\,.
\end{equation}
This quantity should be interpreted as a deterministic version of $\abs{\mu_a - \gamma_-}\wedge \abs{\mu_a - \gamma_+}$ for $a \notin \cal O$; see Theorem \ref{thm: cov-rig} below.

\begin{theorem}[Delocalization bound for non-outliers] \label{thm: non-outlier bound}
Fix $\tau > 0$.
For $a \in \qq{1, (1 - \tau)K} \setminus \cal O$ and deterministic $\f w \in \R^M$ we have
\begin{equation} \label{bulk_right}
\scalar{\f w}{\f \xi_a}^2 \;\prec\; \frac{\abs{\f w}^2}{M} + \sum_{i = 1}^M \frac{\sigma_i w_i^2}{M ((d_i - 1)^2 + \kappa_a)}\,.
\end{equation}
Similarly, if $\abs{\phi - 1} \geq \tau$ then
for $a \in \qq{\tau K, K} \setminus \cal O$ and deterministic $\f w \in \R^M$ we have
\begin{equation} \label{bulk_left}
\scalar{\f w}{\f \xi_a}^2 \;\prec\; \frac{\abs{\f w}^2}{M} + \sum_{i = 1}^M \frac{\sigma_i w_i^2}{M ((d_i + 1)^2 + \kappa_a)}\,.
\end{equation}
\end{theorem}

For the following examples, we take $\f w = \f v_i$ and $a \in \qq{1, (1 - \tau)K} \setminus \cal O$. Under these assumptions \eqref{bulk_right} yields
\begin{equation} \label{bulk_example}
\scalar{\f w}{\f \xi_a}^2 \;\prec\; \frac{1}{M} + \frac{\sigma_i}{M ((d_i - 1)^2 + \kappa_a)}\,.
\end{equation}

\begin{example}
Fix $\tau > 0$.
If $\abs{d_i - 1} \geq \tau$ ($d_i$ is separated from the transition point) or $a \geq \tau K$ ($\mu_a$ is in the bulk), then the right-hand side of \eqref{bulk_example} reads $(1 + \sigma_i) / M$. In particular, if the eigenvalue $\sigma_i$ of $\Sigma$ is bounded, $\f \xi_a$ is completely delocalized in the direction $\f v_i$.
\end{example}

\begin{example}
Suppose that $a \leq C$ ($\mu_a$ is near the edge), which implies that $\kappa_a \asymp K^{-2/3}$. Suppose moreover that $d_i$ is near the transition point $1$. Then we get
\begin{equation*}
\scalar{\f v_i}{\f \xi_a}^2 \;\prec\; \frac{\sigma_i}{M ((d_i - 1)^2 + K^{-2/3})}\,.
\end{equation*}
Therefore the delocalization bound for $\f \xi_a$ in the direction of $\f v_i$ becomes worse as $d_i$ approaches the critical point (from either side), from $(1 + \phi)^{1/2} M^{-1}$ for $d_i$ separated from $1$, to $(1 + \phi)^{-1/6} M^{-1/3}$ for $d_i$ at the transition point $1$.
\end{example}

Next, we derive the law of the generalized component $\scalar{\f w}{\f \xi_a}$ for non-outlier $a$. In particular, this provides a lower bound complementing the upper bound from Theorem \ref{thm: non-outlier bound}. Recall the definition \eqref{def_alpha} of $\alpha_+$.
\begin{theorem}[Law of non-outliers] \label{thm: eigenvector law}
Fix $\tau > 0$. Then, for any deterministic $a \in \qq{1, K^{1 - \tau} \alpha_+^3} \setminus \cal O$ and $\f w \in \R^M$, there exists a random variable $\Theta(a, \f w) \equiv \Theta_N(a, \f w)$ satisfying
\begin{equation*}
\scalar{\f w}{\f \xi_a}^2 \;=\; \sum_i \frac{\sigma_i w_i^2}{M (d_i - 1)^2}  \, \Theta(a, \f w)
\end{equation*}
and
\begin{equation*}
\Theta(a, \f w) \;\longrightarrow \chi_1^2
\end{equation*}
in distribution as $N \to \infty$, uniformly in $a$ and $\f w$. Here $\chi_1^2$ is a chi-squared random variable (i.e.\ the square of a standard normal).

An analogous statement holds near the left spectral edge provided $\abs{\phi - 1} \geq \tau$; we omit the details.
\end{theorem}

\begin{remark}
More generally, our method also yields the asymptotic joint distribution of the family
\begin{equation} \label{general_family}
\pB{\mu_{a_1}, \dots, \mu_{a_k}, \scalar{\f u_1}{\f \xi_{b_1}} \scalar{\f \xi_{b_1}}{\f w_1}, \dots, \scalar{\f u_k}{\f \xi_{b_k}} \scalar{\f \xi_{b_k}}{\f w_k}}
\end{equation}
(after a suitable affine rescaling of the variables, as in Theorem \ref{thm:univ_H} below), where $a_1, \dots, a_k, b_1, \dots, b_k \in \qq{1, K^{1 - \tau} \alpha_+^3} \setminus \cal O$. We omit the precise statement, which is a universality result: it says essentially that the asymptotic distribution of \eqref{general_family} coincides with that under the standard Wishart ensemble (i.e.\ an uncorrelated Gaussian sample covariance matrix). The proof is a simple corollary of Theorem \ref{thm:sticking}, Proposition \ref{prop: bulk_law}, Proposition \ref{prop: level repulsion}, and Theorem \ref{thm:univ_H}.
\end{remark}

\begin{remark}
The restriction $a \leq K^{1 - \tau} \alpha_+^3$ is the same as in Remarks \ref{rem: sticking_univ} and \ref{rem:univ_Q}. There, it is required for the eigenvalue sticking to be effective in the sense that the right-hand side of \eqref{sticking_1} is much smaller than the scale on which the eigenvalue $\lambda_a$ fluctuates. Here, it ensures that the distribution of the eigenvector $\f \xi_a$ is determined by the distribution of a single eigenvector of $H$ (see Proposition \ref{prop: bulk_law}).
\end{remark}

Finally, instead of $Q$ defined in \eqref{wtH_cov}, we may also consider
\begin{equation} \label{def_dot_Q}
\dot Q \;\deq\; \frac{N}{N - 1} T X (I_N - \f e \f e^*) X^* T^*\,,
\end{equation}
where the vector $\f e$ was defined in \eqref{def_vec_e}. All of our results stated for $Q$ also hold for $\dot Q$.
\begin{theorem} \label{thm: Q dot}
Theorems \ref{thm: outlier locations}, \ref{thm:sticking}, \ref{thm: outlier eigenvectors}, \ref{thm: outlier eigenvectors 2}, \ref{thm: non-outlier bound}, and \ref{thm: eigenvector law} hold with $\mu_i$ and $\f \xi_i$ denoting the eigenvalues and eigenvectors of $\dot Q$ instead of $Q$. For Theorem \ref{thm:sticking}, $\lambda_i$ denotes the eigenvalues of $\frac{N}{N - 1} Y (I_N - \f e \f e^*) Y^*$ instead of $Y Y^*$ from \eqref{def_H}.
\end{theorem}

\section{Preliminaries} \label{sec:prelim}

The rest of this paper is devoted to the proofs of the results from Sections \ref{sec:eval_results}--\ref{sec:non_outl_results}. To clarify the presentation of the main ideas of the proofs, we shall first assume that
\begin{equation} \label{simpl_assumption}
\txt{$r = 0$ \, and \, $T = \Sigma^{1/2}$\,.}
\end{equation}
We make the assumption \eqref{simpl_assumption} throughout Sections \ref{sec:prelim}--\ref{sec: QUE}. The additional arguments required to relax the assumption \eqref{simpl_assumption} are presented in Section \ref{sec:general_model}. Under the assumption \eqref{simpl_assumption} we have
\begin{equation} \label{Q_simplified}
Q \;=\; \Sigma^{1/2} X X^* \Sigma^{1/2}\,.
\end{equation}
Moreover, the extension of our results from $Q$ to $\dot Q$, and hence the proof of Theorem \ref{thm: Q dot}, is given in Section \ref{sec: Q dot}.

For an $M \times M$ matrix $A$ and $\f v, \f w \in \R^M$ we abbreviate
\begin{equation*}
A_{\f v \f w} \;\deq\; \scalar{\f v}{A \f w}\,.
\end{equation*}
We also write
\begin{equation*}
A_{\f v \f e_i} \;\equiv\; A_{\f v i}\,,
\qquad
A_{\f e_i \f v} \;\equiv\; A_{i \f v}\,,
\qquad
A_{\f e_i \f e_j} \;\equiv\; A_{i j}\,,
\end{equation*}
where $\f e_i \in \R^M$ denotes the $i$-th standard basis vector.

\subsection{The isotropic local Marchenko-Pastur law}

In this section we collect the key tool of our analysis: the isotropic Marchenko-Pastur law from \cite{BEKYY}.

It is well known that the empirical distribution of the eigenvalues of the $N\times N$ matrix $X^*X$ has the same asymptotics as the \emph{Marchenko-Pastur} law
\begin{equation}\label{def: rhog}
\varrho_\phi(\dd x) \;\deq\; \frac{\sqrt\phi}{2\pi}\frac{\sqrt{\q{(x-\gamma_-)(\gamma_+-x)}_+}}{x}\, \dd x + (1-\phi)_+ \, \delta(\dd x)\,,
\end{equation}
where we recall the edges $\gamma_\pm$ of the limiting spectrum defined in \eqref{def:gamma_pm}. Similarly, as noted in \eqref{MP_law}, the empirical distribution of the eigenvalues of the $M \times M$ matrix $X X^*$ has the same asymptotics as $\varrho_{\phi^{-1}}$.

Note that \eqref{def: rhog} is normalized so that its integral is equal to one. 
The Stieltjes transform of the Marchenko-Pastur law \eqref{def: rhog} is
\begin{equation} \label{S_MP}
m_\phi(z)\;\deq\; \int \frac{\varrho_\phi (\dd x)}{x-z} \;=\; \frac{\phi^{ 1/2}-\phi^{-1/2}-z+\ii \sqrt{(z-\gamma_-)(\gamma_+-z)}}{2\,\phi^{-1/2}\, z}\,,
\end{equation}
where the square root is chosen so that $m_\phi$ is holomorphic in the upper half-plane and satisfies $m_\phi(z) \to 0$ as $z \to \infty$. 
The function $m_\phi = m_\phi(z)$ is also characterized as the unique solution of the equation
\begin{equation} \label{identity for m MP}
m+\frac{1}{z+z\phi^{-1/2}m-(\phi^{ 1/2}-\phi^{-1/2})} \;=\; 0
\end{equation}
satisfying $\im m (z) > 0$ for $\im z >0$.
The formulas \eqref{def: rhog}--\eqref{identity for m MP} were originally derived for the case when $\phi=M/N$ is independent of $N$ 
 (or, more precisely, when $\phi$ has a limit in $(0,\infty)$ as $N \to \infty$).  Our results
allow $\phi$ to depend on $N$ under the constraint \eqref{NM gen}, so that $m_\phi$ and $\varrho_\phi$ may also depend on $N$ through $\phi$.

Throughout the following we use a spectral parameter
\begin{equation*}
z \;=\; E + \ii \eta\,,
\end{equation*}
with $\eta > 0$, as the argument of Stieltjes transforms and resolvents.
Define the resolvent
\begin{equation*}
G(z) \;\deq\; (X X^*-z)^{-1}\,.
\end{equation*}
For $z\in \C$, define $\kappa(z)$ to be the distance from $E=\re z$ to the spectral edges $\gamma_\pm$, i.e.
\begin{equation} \label{def kappa sc}
\kappa \;\equiv\; \kappa(z) \;\deq\; \abs{\gamma_+-E}\wedge\abs{\gamma_- -E}\,.
\end{equation}
Throughout the following we regard the quantities $E(z)$, $\eta(z)$, and $\kappa(z)$ as functions of $z$ and usually omit the argument unless it is needed to avoid confusion.

Sometimes we shall need the following notion of high probability.
\begin{definition} \label{def: high probability}
An $N$-dependent event $\Xi \equiv \Xi_N$ holds with \emph{high probability} if $1 - \ind{\Xi} \prec 0$.
\end{definition}

Fix a (small) $\omega \in (0,1)$
and define the domain
\begin{equation} \label{def_S_theta}
\f S \;\equiv\; \f S(\omega, K)  \;\deq\; \hb{z \in \C \col  \kappa \leq   \omega^{-1} \,,\, K^{-1+\omega} \leq \eta \leq  \omega^{-1} \,,\, \abs{z} \geq \omega}\,.
\end{equation}
Beyond the support of the limiting spectrum, one has stronger control all the way down to the real axis. For fixed (small) $\omega > 0$ define the region
\begin{equation} \label{def_S_theta_wt}
\wt{\f S} \;\equiv\; \wt {\f S}(\omega, K) \;\deq\; \hb{z \in \C \col E \notin [\gamma_-, \gamma_+]\,,\, K^{-2/3 + \omega} \leq \kappa \leq \omega^{-1} \,,\, \abs{z} \geq \omega\,,\, 0 < \eta \leq \omega^{-1}}
\end{equation}
of spectral parameters separated from the asymptotic spectrum by $K^{-2/3 + \omega}$, which may have an arbitrarily small positive imaginary part $\eta$. Throughout the following we regard $\omega$ as fixed once and for all, and do not track the dependence of constants on $\omega$.

\begin{theorem}[Isotropic local Marchenko-Pastur law \cite{BEKYY}] \label{thm: IMP gen}
Suppose that \eqref{cond on entries of X}, \eqref{NM gen}, and \eqref{moments of X-1} hold. 
Then
\begin{equation}\label{bound: Rij isotropic gen}
\absb{\scalar{\f v}{G(z) \f w} - m_{\phi^{-1}}(z) \scalar{\f v}{\f w}} \;\prec\; \sqrt{\frac{\im m_{\phi^{-1}}(z)}{M \eta}} + \frac{1}{M \eta}
\end{equation}
uniformly in $z \in \f S$ and any deterministic unit vectors $\f v, \f w \in \R^M$.
Moreover,
\begin{equation}\label{bound: Rij isotropic outside sc gen}
\absb{\scalar{\f v}{G(z) \f w} - m_{\phi^{-1}}(z) \scalar{\f v}{\f w}} \;\prec\; \sqrt{\frac{\im m_{\phi^{-1}}(z)}{M\eta}} \;\asymp\; \frac{1}{1 + \phi} (\kappa + \eta)^{-1/4} K^{-1/2}
\end{equation}
uniformly in $z \in \wt{\f S}$ and any deterministic unit vectors $\f v, \f w \in \R^M$.
\end{theorem}

\begin{remark} \label{rem:all_z}
The probabilistic estimates
\eqref{bound: Rij isotropic gen} and \eqref{bound: Rij isotropic outside sc gen} of Theorem \ref{thm: IMP gen}  may be strengthened to hold simultaneously for all $z \in \f S$ and for all $z\in \wt{\f S}$, respectively.
For instance, \eqref{bound: Rij isotropic outside sc gen} may be strengthened to
\begin{equation*}
\P \qBB{\bigcap_{z \in \wt {\f S}} \hBB{\absb{\scalar{\f v}{G(z) \f w} - m_{\phi^{-1}}(z) \scalar{\f v}{\f w}} \leq N^\epsilon \frac{1}{1 + \phi} (\kappa + \eta)^{-1/4} K^{-1/2}}} \;\geq\; 1 - N^{-D}\,,
\end{equation*}
for all $\epsilon > 0$, $D > 0$, and $N \geq N_0(\epsilon, D)$. See \cite[Remark 2.6]{BEKYY}.
\end{remark}

The next results are on the nontrivial (i.e.\ nonzero) eigenvalues of $H \deq XX^*$ as well as the corresponding eigenvectors. The matrix $H$ has $K$
nontrivial eigenvalues, which we order according to
\begin{equation} \label{def_lambda}
\lambda_1 \;\geq\; \lambda_2 \;\geq\; \cdots \;\geq\; \lambda_K\,.
\end{equation}
(The remaining $M - K$ eigenvalues of $H$ are zero.)
Moreover, we denote by
\begin{equation} \label{def_zeta}
\f \zeta_1, \f \zeta_2 , \dots, \f \zeta_K \in \R^M
\end{equation}
the unit eigenvectors of $H$ associated with the nontrivial eigenvalues $\lambda_1 \geq \lambda_2 \geq \dots \geq \lambda_{K}$.

\begin{theorem}[Isotropic delocalization \cite{BEKYY}] \label{thm: I deloc gen}
Fix $\tau > 0$, and suppose that \eqref{cond on entries of X}, \eqref{NM gen}, and \eqref{moments of X-1} hold. Then for $i \in \qq{1,K}$ we have
\begin{equation}\label{smfy gen}
\scalar{\f \zeta_i}{\f v}^2 \;\prec\; M^{-1} 
\end{equation}
if either $i \leq (1 - \tau) K$ or $\abs{\phi - 1} \geq \tau$.
\end{theorem}

The following result is on the rigidity of the nontrivial eigenvalues of $H$.
Let $\gamma_1 \geq \gamma_2 \geq \cdots \geq \gamma_K$ be the \emph{classical eigenvalue locations according to $\varrho_{\phi}$} (see \eqref{def: rhog}), defined through
\begin{equation} \label{def:gamma_alpha}
\int_{\gamma_i}^\infty \varrho_{\phi}(\dd x) \;=\; \frac{i}{N}\,.
\end{equation}

\begin{theorem}[Eigenvalue rigidity \cite{BEKYY}]\label{thm: cov-rig}
Fix $\tau > 0$,  and suppose that \eqref{cond on entries of X}, \eqref{NM gen}, and \eqref{moments of X-1} hold.
Then for $i \in \qq{1,M}$ we have
\begin{equation}\label{rigidity1}
\absb{\lambda_i-\gamma_i}   \;\prec\; \pb{i \wedge (K +1 - i)}^{-1/3}K^{-2/3}
\end{equation}
if $i \leq (1 - \tau) K$ or $\abs{\phi - 1} \geq \tau$.
\end{theorem}

\subsection{Link to the semicircle law} \label{sec: sc}
It will often be convenient to replace the Stieltjes transform $m_\phi(z)$ of $\varrho_{\phi}(\dd x)$ with the Stieltjes transform $w_\phi(z)$ of the measure
\begin{equation} \label{sc_trans}
\phi^{1/2} x \varrho_{\phi^{-1}}(\dd x) \;=\; \frac{1}{2\pi}\sqrt{\q{(x-\gamma_-)(\gamma_+-x)}_+}\, \dd x\,.
\end{equation}
Note that this is nothing but Wigner's semicircle law centred at $\phi^{1/2} + \phi^{-1/2}$.
Thus,
\begin{equation} \label{def_w}
w_\phi(z) \;\deq\; \int \frac{\phi^{1/2} x \varrho_{\phi^{-1}}(\dd x)}{x - z} \;=\; \phi^{1/2}\pb{1 + z m_{\phi^{-1}}(z)} \;=\; \frac{\phi^{ 1/2}+\phi^{-1/2}-z+\ii \sqrt{(z-\gamma_-)(\gamma_+-z)}}{2}\,,
\end{equation}
where in the last step we used \eqref{S_MP}. Note that
\begin{equation*}
w_\phi \;=\; w_{\phi^{-1}}\,.
\end{equation*}
Using $w_\phi$ we can write \eqref{identity for m MP} as
\begin{equation} \label{self_w}
z \;=\; (1 - \phi^{-1/2} w_\phi^{-1}) (\phi^{1/2} - w_\phi)\,.
\end{equation}

\begin{lemma} \label{lemma: w}
For $z \in \f S$ and $\phi \geq 1$ we have
\begin{equation} \label{bounds on mg}
\abs{m_\phi(z)} \;\asymp\; \abs{w_\phi(z)} \;\asymp\; 1 \,, \qquad %\abs{1 - m_\phi(z)^2} \;\asymp\;
\abs{1 - w_\phi(z)^2} \;\asymp\; \sqrt{\kappa + \eta}\,,
\end{equation}
as well as
\begin{equation} \label{im m gamma}
\im m_\phi(z) \;\asymp\; \im w_\phi(z) \;\asymp\;
\begin{cases}
\sqrt{\kappa + \eta} & \text{if $E \in [\gamma_-, \gamma_+]$}
\\
\frac{\eta}{\sqrt{\kappa + \eta}} & \text{if $E \notin [\gamma_-, \gamma_+]$}\,.
\end{cases}
\end{equation}
Similarly,
\begin{equation} \label{re m gamma}
\re m_\phi(z) - I(z) \;\asymp\; \re w_\phi(z) - I(z) \;\asymp\;
\begin{cases}
\frac{\eta}{\sqrt{\kappa + \eta}} + \kappa & \text{if $E \in [\gamma_-, \gamma_+]$}
\\
\sqrt{\kappa + \eta} & \text{if $E \notin [\gamma_-, \gamma_+]$}\,,
\end{cases}
\end{equation}
where $I(z) \deq -1$ for $E \geq \phi^{1/2} + \phi^{-1/2}$ and $I(z) \deq +1$ for $E < \phi^{1/2} + \phi^{-1/2}$.
Finally, for $z \in \f S$ we have
\begin{equation} \label{im_m_swap}
\im m_{\phi^{-1}}(z) \;\asymp\; \frac{1}{\phi} \im m_\phi(z)\,.
\end{equation}
(All implicit constants depend on $\omega$ in the definition \eqref{def_S_theta} of $\f S$.)
\end{lemma}

\begin{proof}
The estimates \eqref{bounds on mg} and \eqref{im m gamma} follow from the explicit expressions in \eqref{S_MP} and \eqref{def_w}. In fact, these estimates have already appeared in previous works. Indeed, for $m_\phi$ the estimates \eqref{bounds on mg} and \eqref{im m gamma} were proved in \cite[Lemma 3.3]{BEKYY}. In order to prove them for $w_\phi$, we observe that the estimates \eqref{bounds on mg} and \eqref{im m gamma} follow from the corresponding ones for the semicircle law, which were proved in \cite[Lemma 4.3]{EKYY4}. The estimates \eqref{re m gamma} follow from \eqref{im m gamma} and the elementary identity
\begin{equation*}
\re w_\phi \;=\; - \frac{E - \phi^{1/2} - \phi^{-1/2}}{1 + \eta / \im w_\phi}\,,
\end{equation*}
which can be derived from \eqref{self_w}; the estimates for $m_\phi$ are derived similarly.  Finally, \eqref{im_m_swap} follows easily from
\begin{equation} \label{mphi_minvphi}
m_{\phi^{-1}}(z) \;=\; \frac{1}{\phi} \pbb{m_\phi(z) + \frac{1 - \phi}{z}}\,,
\end{equation}
which may itself be derived from \eqref{identity for m MP}.
\end{proof}

In analogy to $w_\phi$ (see \eqref{def_w}), we define the matrix-valued function
\begin{equation} \label{def_F}
F(z) \;\deq\; \phi^{1/2} (1 + z G(z))\,.
\end{equation}
Theorem \ref{thm: IMP gen} has the following analogue, which compares $F$ with $m_\phi$.
\begin{lemma} \label{lem: Qij iso}
Suppose that \eqref{cond on entries of X}, \eqref{NM gen}, and \eqref{moments of X-1} hold. 
Then
\begin{equation}\label{bound: Qij iso}
\absb{\scalar{\f v}{F(z) \f w} - w_\phi(z) \scalar{\f v}{\f w}} \;\prec\; \sqrt{\frac{\im w_\phi(z)}{K \eta}} + \frac{1}{K \eta}
\end{equation}
uniformly in $z \in \f S$ and any deterministic unit vectors $\f v, \f w \in \R^M$.
Moreover,
\begin{equation}\label{bound: Qij iso outside}
\absb{\scalar{\f v}{F(z) \f w} - w_\phi(z) \scalar{\f v}{\f w}} \;\prec\; \sqrt{\frac{\im w_\phi(z)}{K \eta}} \;\asymp\; (\kappa + \eta)^{-1/4} K^{-1/2}
\end{equation}
uniformly in $z \in \wt{\f S}$ and any deterministic unit vectors $\f v, \f w \in \R^M$.
\end{lemma}

\begin{proof}
The proof is an easy consequence of Theorem \ref{thm: IMP gen} and Lemma \ref{lemma: w}, combined with the fact that for $z \in \f S$ or $z \in \wt {\f S}$ we have $\abs{z} \asymp \phi^{1/2}$ for $\phi \geq 1$ and $\abs{z} \asymp \phi^{-1/2}$ for $\phi \leq 1$.
\end{proof}

\subsection{Extension of the spectral domain} \label{sec: ext_dom}
In this section we extend the spectral domain on which Theorem \ref{thm: IMP gen} and Lemma \ref{lem: Qij iso} hold. The argument relies on the Helffer-Sj\"ostrand functional calculus \cite{Davies}. Define the domains
\begin{equation*}
\wh {\f S} \;\equiv\; \wh{\f S}(\omega, K) \;\deq\; \hb{z \in \C \col E \notin [\gamma_-, \gamma_+]\,,\, \kappa \geq K^{-2/3 + \omega}\,,\, \eta > 0}\,, \qquad
\f B \;\equiv\; \f B(\omega) \;\deq\; \h{z \in \C \col \abs{z} < \omega}\,.
\end{equation*}
\begin{proposition} \label{prop:extension}
Fix $\omega, \tau \in (0,1)$.
\begin{enumerate}
\item
If $\phi < 1 - \tau$ then
\begin{equation}\label{extended_1}
\absb{\scalar{\f v}{G(z) \f w} - m_{\phi^{-1}}(z) \scalar{\f v}{\f w}} \;\prec\; \frac{1}{(\kappa + \eta)^2 + (\kappa + \eta)^{1/4}} K^{-1/2}
\end{equation}
uniformly for $z \in \wh{\f S}$ and any deterministic unit vectors $\f v, \f w \in \R^M$.
\item
If $\abs{\phi - 1} \leq \tau$ then \eqref{extended_1} holds uniformly for $z \in \wh{\f S} \setminus \f B$ and any deterministic unit vectors $\f v, \f w \in \R^M$.
\item
If $\phi > 1 + \tau$ then
\begin{equation}\label{extended_2}
\absb{\scalar{\f v}{G(z) \f w} - m_{\phi^{-1}}(z) \scalar{\f v}{\f w}} \;\prec\;  \frac{1}{\phi^{1/2} \abs{z} ((\kappa + \eta) + (\kappa + \eta)^{1/4})} K^{-1/2}
\end{equation}
uniformly for $z \in \wh{\f S} \setminus \{0\}$ and any deterministic unit vectors $\f v, \f w \in \R^M$.
\end{enumerate}
\end{proposition}

\begin{proof}
By polarization and linearity, we may assume that $\f w = \f v$. Define the signed measure
\begin{equation}
\rho^\Delta(\dd x) \;\deq\; \sum_{i = 1}^M \scalar{\f v}{\f \zeta_i} \scalar{\f \zeta_i}{\f v} \, \delta_{\lambda_i}(\dd x) -  \varrho_{\phi^{-1}}(\dd x) \,,
\end{equation}
so that
\begin{equation*}
m^\Delta(z) \;\deq\; \int \frac{\rho^\Delta(\dd x)}{x - z} \;=\; \scalar{\f v}{G(z) \f v} -  m_{\phi^{-1}}(z)\,.
\end{equation*}
The basic idea of the proof is to apply the Helffer-Sj\"ostrand formula to the function
\begin{equation*}
f_z(x) \;\deq\; \frac{1}{x - z} - \frac{1}{x_0 - z}\,,
\end{equation*}
where $x_0$ is chosen below. To that end, we need a smooth compactly supported cutoff function $\chi$ on the complex plane satisfying $\chi(w) \in [0,1]$ and $\abs{\partial_{\bar w} \chi(w)} \leq C(\omega, \tau)$. We distinguish the three cases $\phi < 1 - \tau$, $\abs{\phi - 1} \leq \tau$, and $\phi > 1 + \tau$.

Let us first focus on the case $\phi < 1 - \tau$. Set $x_0 \deq \phi^{1/2} + \phi^{-1/2}$ and choose a constant $\omega' = \omega'(\omega, \tau) \in (0, \omega)$ small enough that $\gamma_- \geq 4 \omega'$. We require that $\chi$ be equal to $1$ in the $\omega'$-neighbourhood of $[\gamma_-, \gamma_+]$ and $0$ outside of the $2\omega'$-neighbourhood of $[\gamma_-,\gamma_+]$. By Theorem \ref{thm: cov-rig} we have $\supp \rho^\Delta \subset \{\chi = 1\}$ with high probability. Now choose $z$ satisfying $\dist(z, [\gamma_-, \gamma_+]) \geq 3 \omega'$. Then the Helffer-Sj\"ostrand formula \cite{Davies} yields, for $x \in \supp \rho^\Delta$,
\begin{equation} \label{HS formula}
f_z(x) \;=\; \frac{1}{\pi} \int_{\C} \frac{\partial_{\bar w} (f_z(w) \chi(w))}{x - w} \, \dd w
\end{equation}
with high probability, where $\dd w$ denotes the two-dimensional Lebesgue measure in the complex plane. Noting that $\int \dd \rho^{\Delta} = 0$, we may therefore write
\begin{equation} \label{HS formula 2}
m^\Delta(z) \;=\; \int \rho^\Delta(\dd x) \, f_z(x)
\;=\;
\frac{1}{\pi} \int_{\C}  f_z(w) \, \partial_{\bar w} \chi(w) \, m^\Delta(w) \, \dd w
\end{equation}
with high probability, where in second step we used \eqref{HS formula} and the fact that $f_z$ is holomorphic away from $z$. The integral is supported on the set $\{\partial_{\bar w} \chi \neq 0\} \subset \h{w \col \dist(w,[\gamma_-, \gamma_+]) \in [\omega', 2 \omega']}$, on which we have the estimates $\abs{f_z(w)} \leq C (\kappa(z) + \eta(z))^{-2}$ and $\abs{m^\Delta(w)} \prec K^{-1/2}$, as follows from Theorem \ref{bound: Rij isotropic outside sc gen} applied to $\f S(\omega', K)$ and \eqref{im_m_swap}. Recalling Remark \ref{rem:all_z}, we may plug these estimates into the integral to get
\begin{equation*}
\abs{m^\Delta(z)} \;\prec\;  (\kappa + \eta)^{-2} K^{-1/2}\,,
\end{equation*}
which holds for $\dist(z, [\gamma_-, \gamma_+]) \geq 3 \omega'$. (Recall that $\abs{\partial_{\bar w} \chi(w)} \leq C$.) Combining this estimate with \eqref{bound: Rij isotropic outside sc gen}, the claim \eqref{extended_1} follows for $z \in \wh{\f S}$.

Next, we deal with the case $\abs{\phi - 1} \leq \tau$. The argument is similar. We again choose $x_0 \deq \phi^{1/2} + \phi^{-1/2}$. We require that $\chi$ be equal to $1$ in the $\omega$-neighbourhood of $[0, \gamma_+]$ and $0$ outside of the $2\omega$-neighbourhood of $[0,\gamma_+]$. We may now repeat the above argument almost verbatim. For $\dist\{z, [0, \gamma_+]\} \geq 3 \omega$ and $w \in \{\partial_{\bar w} \chi \neq 0\}$ we find that $\abs{f_z(w)} \leq C (\kappa(z) + \eta(z))^{-2}$ and $\abs{m^\Delta(w)} \prec K^{-1/2}$. Hence, recalling \eqref{bound: Rij isotropic outside sc gen}, we get \eqref{extended_1}  for $z \in \wh{\f S} \setminus \f B$.

Finally, suppose that $\phi > 1 + \tau$. Now we set $x_0 \deq 0$. We choose the same $\omega'$ and cutoff function $\chi$ as in the case $\phi < 1 - \tau$ above. Suppose that $\dist(z, [\gamma_-, \gamma_+]) \geq 3 \omega'$ and $z \neq 0$. Thus, \eqref{HS formula} holds with high probability for $x \in \supp \rho^\Delta \setminus \{0\}$. Since $f_w(0) = 0$, we therefore find that \eqref{HS formula 2} holds. As above, we find that for $w \in \{\partial_{\bar w} \chi \neq 0\}$ we have
\begin{equation*}
\abs{f_z(w)} \;\leq\; \frac{C \phi^{1/2}}{\abs{z} (\kappa(z) + \eta(z))}
\end{equation*}
and $\abs{m^\Delta(w)} \prec \phi^{-1} K^{-1/2}$. Recalling \eqref{bound: Rij isotropic outside sc gen}, we find that \eqref{extended_2} follows easily.
\end{proof}

Proposition \ref{prop:extension} yields the following result for $F$ defined in \eqref{def_F}.

\begin{corollary} \label{cor:extension}
Fix $\omega, \tau \in (0,1)$.
\begin{enumerate}
\item
If $\phi < 1 - \tau$ then
\begin{equation}\label{extended_1_Q}
\absb{\scalar{\f v}{F(z) \f w} - w_{\phi}(z) \scalar{\f v}{\f w}} \;\prec\; \frac{\phi^{1/2} \abs{z}}{(\kappa + \eta)^2 + (\kappa + \eta)^{1/4}} K^{-1/2}
\end{equation}
uniformly for $z \in \wh{\f S}$ and any deterministic unit vectors $\f v, \f w \in \R^M$.
\item
If $\abs{\phi - 1} \leq \tau$ then \eqref{extended_1_Q} holds uniformly for $z \in \wh{\f S} \setminus \f B$ and any deterministic unit vectors $\f v, \f w \in \R^M$.
\item
If $\phi > 1 + \tau$ then
\begin{equation}\label{extended_2_Q}
\absb{\scalar{\f v}{F(z) \f w} - w_{\phi}(z) \scalar{\f v}{\f w}} \;\prec\;  \frac{1}{(\kappa + \eta) + (\kappa + \eta)^{1/4}} K^{-1/2}
\end{equation}
uniformly for $z \in \wh{\f S}$ and any deterministic unit vectors $\f v, \f w \in \R^M$.
\end{enumerate}
\end{corollary}

\subsection{Identities for the resolvent and eigenvalues}\label{sec:linalg}
In this section we derive the identities on which our analysis of the eigenvalues and eigenvectors relies. Recall the definition of the set $\cal R$ from \eqref{def_calR}. We write the population covariance matrix $\Sigma$ from \eqref{def_Sigma} as
\begin{equation*}
\Sigma \;=\; 1 + \phi^{1/2} V D V^*\,,
\end{equation*}
where $D = \diag(d_i)_{i \in \cal R}$ is an invertible diagonal $\abs{\cal R} \times \abs{\cal R}$ matrix and
$V = [\f v_i]_{i \in \cal R}$ is the matrix of eigenvectors $\f v_i$ of $\Sigma$ indexed by the set $\cal R$.
Note that $V$ is an $N \times \abs{\cal R}$ isometry, i.e.\ $V$ satisfies $V^* V = I_{\abs{\cal R}}$.

We use the definitions
\begin{equation*}
G(z) \;\deq\; (H - z)^{-1} \,, \qquad \wt G(z) \;\deq\; (Q - z)^{-1}\,, \qquad F(z) \;\deq\; \phi^{1/2} (1 + z G(z))\,,
\end{equation*}
where $H = X X^*$ and $Q = \Sigma^{1/2} H \Sigma^{1/2}$.
We introduce the $\abs{\cal R} \times \abs{\cal R}$ matrix
\begin{equation*}
W(z) \;\deq\; V^* F(z) V\,.
\end{equation*}
We also denote by $\sigma(A)$ the spectrum of a square matrix $A$.

The following lemma collects the basic identities for analysing $\sigma(Q)$ and $\wt G$. We remark that versions of its part (i) have already appeared in several previous works on finite-rank deformations of random matrix ensembles \cite{BGGM1, BY1, KY2, SoshPert}.

\begin{lemma} \label{lem:linalg}
\begin{enumerate}
\item
Suppose that $\mu \notin \sigma(H)$. Then $\mu \in \sigma(Q)$ if and only if
\begin{equation} \label{pert1}
\det \pb{D^{-1} + W(\mu)} \;=\; 0\,.
\end{equation}
\item
We have
\begin{equation} \label{pert2}
\Sigma^{1/2} \wt G(z) \Sigma^{1/2} \;=\; G(z) - G(z) V \frac{\phi^{1/2} z}{D^{-1} + W(z)} V^* G(z)\,.
\end{equation}
\end{enumerate}
\end{lemma}

\begin{proof}
To prove (i), we write the condition $\mu \in \sigma(Q)$ as
\begin{equation*}
0 \;=\; \det \pb{\Sigma^{1/2} H \Sigma^{1/2} - \mu} \;=\; \det \pb{H - \mu \Sigma^{-1}} \det(\Sigma) \;=\;  \det\pb{1 + G(\mu) (1 - \Sigma^{-1}) \mu} \det (H - \mu) \det(\Sigma)\,,
\end{equation*}
where we used that $\mu \notin \sigma(H)$. Using
\begin{equation*}
1 - \Sigma^{-1} \;=\; V \frac{D}{\phi^{-1/2} + D} V^*\,,
\end{equation*}
the matrix identity $\det(1 + XY) = \det(1 + YX)$, and $\det(\Sigma) \neq 0$, we find
\begin{equation*}
0 \;=\; \det \pbb{1 + \frac{D}{\phi^{-1/2} + D} \mu V^* G(\mu) V}\,,
\end{equation*}
and the claim follows.

To prove (ii), we write
\begin{equation*}
\Sigma^{1/2} \wt G(z) \Sigma^{1/2} \;=\;  (H - \Sigma^{-1} z)^{-1} \;=\; \pb{H - z + (1 - \Sigma^{-1}) z}^{-1}\,.
\end{equation*}
The claim now follows from the identity
\begin{equation} \label{woodbury}
(A + S B T)^{-1} \;=\; A^{-1} - A^{-1} S \pb{B^{-1} + T A^{-1} S}^{-1} T  A^{-1}
\end{equation}
with $A = H - z$, $B = D (\phi^{-1/2} + D)^{-1}$, $S = V$, and $T = zV^*$.
\end{proof}

The result \eqref{pert2}, when restricted to the range of $V$, has an alternative form \eqref{pert3} which is often easier to work with, since it collects all of the randomness in the single quantity $W(z)$ on its right-hand side.

\begin{lemma}
We have
\begin{equation} \label{pert3}
V^* \wt G(z) V \;=\; \frac{1}{\phi^{1/2} z} \pbb{D^{-1} - \frac{\sqrt{1 + \phi^{1/2} D}}{D} \frac{1}{D^{-1} + W(z)} \frac{\sqrt{1 + \phi^{1/2} D}}{D}}\,.
\end{equation}
\end{lemma}

\begin{proof}
From \eqref{pert2} we get
\begin{equation*}
(1 + \phi^{1/2} D)^{1/2} \, V^* \wt G V \, (1 + \phi^{1/2} D)^{1/2} \;=\; V^*GV - V^*GV \frac{1}{(D^{-1} + \phi^{1/2})/(z \phi^{1/2}) + V^*GV} V^*GV\,.
\end{equation*}
Applying the identity
\begin{equation*}
A - A \, (A + B)^{-1} A \;=\; B - B \, (A + B)^{-1} B
\end{equation*}
to the right-hand side yields
\begin{equation*}
(1 + \phi^{1/2} D)^{1/2} \, V^* \wt G V \, (1 + \phi^{1/2} D)^{1/2} \;=\; \frac{1}{\phi^{1/2} z} \pbb{D^{-1} + \phi^{1/2} - (D^{-1} + \phi^{1/2}) \frac{1}{D^{-1} + W} (D^{-1} + \phi^{1/2})}\,,
\end{equation*}
from which the claim follows.
\end{proof}

\section{Eigenvalue locations} \label{sec:ev_locations}

In this section we prove Theorems \ref{thm: outlier locations} and \ref{thm:sticking}. The arguments are similar to those of \cite[Section 6]{KY2}, and we therefore only sketch the proofs. The proof of \cite[Section 6]{KY2} relies on three main steps: (i) establishing a forbidden region which contains with high probability no eigenvalues of $Q$; (ii) a counting estimate for the special case where $D$ does not depend on $N$, which ensures that each connected component of the allowed region (complement of the forbidden region) contains exactly the right number of eigenvalues of $Q$; and (iii) a continuity argument where the counting result of (ii) is extended to arbitrary $N$-dependent $D$ using the gaps established in (i) and the continuity of the eigenvalues as functions of the matrix entries. The steps (ii) and (iii) are exactly the same as in \cite{KY2}, and will not be repeated here. The step (i) differs slightly from that of \cite{KY2}, and in the proofs below we explain these differences.

We need the following eigenvalue interlacing result, which is purely deterministic. It holds for any nonnegative definite $M \times M$ matrix $H$ and any rank-one deformation of the form $Q = (1 + \tilde d \f v \f v^*)^{1/2} H (1 + \tilde d \f v \f v^*)^{1/2}$ with $\tilde d \geq -1$ and $\f v \in \R^M$.

\begin{lemma}[Eigenvalue interlacing] \label{lem: interlacing}
Let $\abs{\cal R} = 1$ and $D = d \in \cal D$. For $d > 0$ we have
\begin{equation*}
\mu_1 \;\geq\; \lambda_1 \;\geq\; \mu_2 \;\geq\; \cdots \;\geq\; \lambda_{M-1} \;\geq\; \mu_M \;\geq\; \lambda_M
\end{equation*}
and for $d < 0$ we have
\begin{equation*}
\lambda_1 \;\geq\; \mu_1 \;\geq\; \lambda_2 \;\geq\; \cdots \;\geq\; \mu_{M-1} \;\geq\; \lambda_M \;\geq\; \mu_M\,.
\end{equation*}
\end{lemma}

\begin{proof}
Using a simple perturbation argument (using that eigenvalues depend continuously on the matrix entries), we may assume without loss of generality that $\lambda_1, \dots, \lambda_M$ are all positive and distinct.
Writing $\Sigma = 1 + \phi^{1/2} d \f v \f v^*$, we get from \eqref{pert2} that
\begin{equation*}
\wt G_{\f v \f v}(z) \;=\; a^2 G_{\f v \f v}(z) - a^2 G_{\f v \f v}(z)^2 \frac{1}{b(z)^{-1} + G_{\f v \f v}(z)}\,, \qquad a \;\deq\; (\Sigma^{-1/2})_{\f v \f v} \,, \qquad b(z) \;\deq\; \frac{z}{1 + \phi^{-1/2} d^{-1}}\,.
\end{equation*}
Note that $a > 0$. Thus we get
\begin{equation*}
\frac{1}{G_{\f v \f v}(z)} + b(z) \;=\; \frac{a^2}{\wt G_{\f v \f v}(z)}\,.
\end{equation*}
Writing this in spectral decomposition yields
\begin{equation} \label{interlacing_step}
\pbb{\sum_i \frac{\scalar{\f v}{\f \zeta_i}^2}{\lambda_i - z}}^{-1} \;=\; a^2 \pbb{\sum_i \frac{\scalar{\f v}{\f \xi_i}^2}{\mu_i - z}}^{-1} - b(z)\,.
\end{equation}
As above, a simple perturbation argument implies that we may without loss of generality assume that all scalar products in \eqref{interlacing_step} are nonzero. Now take $z \in (0, \infty)$. Note that $b(z)$ and $d$ have the same sign.

To conclude the proof, we observe that the left-hand side of \eqref{interlacing_step} defines a function of $z \in (0,\infty)$ with $M - 1$ singularities 
and $M$ zeros, which is smooth and decreasing away from the singularities. Moreover, its zeros are the eigenvalues $\lambda_1, \dots, \lambda_M$.
The interlacing property now follows from the fact that $z$ is an eigenvalue of $Q$ if and only if the left-hand side of \eqref{interlacing_step} is equal to $-b(z)$.
\end{proof}

\begin{corollary} \label{cor: interlacing}
For the rank-$\abs{\cal R}$ model \eqref{wtH_cov} we have
\begin{equation*}
\mu_i \;\in\; [\lambda_{i + r}, \lambda_{i - r}] \qquad (i \in \qq{1,M})\,,
\end{equation*}
with the convention that $\lambda_{i} = 0$ for $i > K$ and $\lambda_i = \infty$ for $i < 1$.
\end{corollary}

We now move on to the proof of Theorem \ref{thm: outlier locations}.
Note that the function $\theta$ defined in \eqref{classical_loc} may be extended to a biholomorphic function
from $\{\zeta \in \C \col \abs{\zeta} > 1\}$ to $\{z \in \C \col z - (\phi^{1/2} + \phi^{-1/2}) \notin [-2,2])\}$.
Moreover, using \eqref{self_w} it is easy to check that for $\abs{\zeta} > 1$ we have
\begin{equation} \label{w_theta_relation}
w_\phi(z) \;=\; - \frac{1}{\zeta} \qquad \Longleftrightarrow \qquad z \;=\; \theta(\zeta)\,.
\end{equation}
Throughout the following we shall make use of the subsets of outliers
\begin{equation*}
\cal O_\tau^\pm \;\deq\; \hb{i \col \pm d_i \geq 1 + K^{-1/3 + \tau}}
\end{equation*}
for $\tau \geq 0$. 
Note that $\cal O \;=\; \cal O_0^+ \cup \cal O_0^-$.

\begin{proof}[Proof of Theorem \ref{thm: outlier locations}]
The proof of Proposition \ref{thm: outlier locations} is similar to that of \cite[Equation (2.20)]{KY2}. We focus first on the outliers to the right of the bulk spectrum. Let $\epsilon > 0$. We shall prove that there exists an event $\Xi$ of high probability (see Definition \ref{def: high probability}) such that for all $i \in \cal O_{4 \epsilon}^+$ we have
\begin{equation} \label{outl_bound_eps_1}
\ind{\Xi} \abs{\mu_i - \theta(d_i)} \;\leq\; C \Delta(d_i) K^{-1/2 + \epsilon}
\end{equation}
and for $i \in \qq{\abs{\cal O_{4 \epsilon}^+} + 1, \abs{\cal O_{4 \epsilon}^+} + r}$ we have
\begin{equation} \label{outl_bound_eps_2}
\ind{\Xi} \abs{\mu_i - \gamma_+} \;\leq\; C K^{-2/3 + 8 \epsilon}\,.
\end{equation}

Before proving \eqref{outl_bound_eps_1} and \eqref{outl_bound_eps_2}, we show how they imply \eqref{outlier locations} for $d_i > 0$ and \eqref{ex+ bound}. 
From \eqref{outl_bound_eps_2} we get for $i$ satisfying $K^{-1/3} \leq d_i - 1 \leq K^{-1/3 + 4 \epsilon}$
\begin{equation} \label{outl_smaller_eps}
\ind{\Xi} \abs{\mu_i - \theta(d_i)} \;\leq\; \ind{\Xi} \pb{\abs{\mu_i - \gamma_+} + \abs{\theta(d_i) - \gamma_+}} \;\leq\; C K^{-2/3 + 8 \epsilon} \;\leq\; C \Delta(d_i) K^{-1/2 + 8 \epsilon}\,.
\end{equation}
Since $\epsilon > 0$ was arbitrary, \eqref{outlier locations} for $d_i > 0$ and \eqref{ex+ bound} follow from \eqref{outl_bound_eps_1} and \eqref{outl_smaller_eps}.

What remains is the proof of \eqref{outl_bound_eps_1} and \eqref{outl_bound_eps_2}. As in \cite[Proposition 6.5]{KY2}, the first step is to prove that with high probability there are no eigenvalues outside a neighbourhood of the classical outlier locations $\theta(d_i)$. To that end, we define for each $i \in \cal O_\epsilon^+$ the interval
\begin{equation*}
I_i(D) \;\deq\; \qB{\theta(d_i) - \Delta(d_i) K^{-1/2 + \epsilon} \,,\, \theta(d_i) + \Delta(d_i) K^{-1/2 + \epsilon}}\,.
\end{equation*}
Moreover, we set $I_0 \deq [0, \theta(1 + K^{-1/3 + 2 \epsilon})]$.

We now claim that with high probability the complement of the set $I(D) \deq I_0 \cup \bigcup_{i \in \cal O_{\epsilon}^+} I_i(D)$ contains no eigenvalues of $Q$. Indeed, from Theorem \ref{thm: cov-rig} and Corollary \ref{cor:extension} combined with Remark \ref{rem:all_z} (with small enough $\omega \equiv \omega(\epsilon)$), we find that there exists an event $\Xi$ of high probability such that $\abs{\lambda_i - \gamma_+} \leq K^{-2/3 + \epsilon}$ for $i \in \qq{1,2r}$ and
\begin{equation*}
\ind{\Xi} \normb{W(x) - w_{\phi}(x)} \;\leq\; \cal E(x) \, K^{-1/2 + \epsilon/2}
\end{equation*}
for all $x \notin I_0$, where we defined
\begin{equation*}
\cal E(x) \;\deq\;
\begin{cases}
\kappa(x)^{-1/4} & \text{if } \kappa(x) \leq 1
\\
\frac{1}{\kappa(x)^2} \pb{1 + \frac{\kappa(x)}{1 + \phi^{-1/2}}} & \text{if } \kappa(x) > 1\,.
\end{cases}
\end{equation*}
In particular, we have $\ind{\Xi} \lambda_1 \leq \theta(1 + K^{-1/3 + \epsilon})$.
Hence we find from \eqref{pert1} that on the event $\Xi$ the value $x \notin I_0$ is an eigenvalue of $Q$ if and only if the matrix
\begin{equation*}
\ind{\Xi} \pb{D^{-1} + W(x)} \;=\; \ind{\Xi} \pb{D^{-1} + w_\phi(x) + O(\cal E(x) K^{-1/2 + \epsilon/2})}
\end{equation*}
is singular. Since $-d_i^{-1} = w_\phi(\theta(d_i))$ for $i \in \cal O_\epsilon^+$, we conclude from the definition of $I(D)$ that it suffices to show that if $x \notin I(D)$ then
\begin{equation} \label{min_i_gap}
\min_{i \in \cal O_\epsilon^+} \absb{w_\phi(x) - w_\phi(\theta(d_i))} \;\gg\; \cal E(x) K^{-1/2 + \epsilon/2}\,.
\end{equation}
We prove \eqref{min_i_gap} using the two following observations. First, $w_\phi$ is monotone increasing on $(\gamma_+, \infty)$ and
\begin{equation*}
w_\phi'(x) \;\asymp\; (d_i^2 - 1)^{-1} \qquad (x \in I_i(D))\,,
\end{equation*}
as follows from \eqref{w_theta_relation}. Second,
\begin{equation*}
\Delta(d_i) \;\asymp\; \frac{\cal E(\theta(d_i))}{\abs{w_\phi'(\theta(d_i))}} \;=\; (d_i^2 - 1) \cal E(\theta(d_i))\,.
\end{equation*}
We omit further details, which may be found e.g.\ in \cite[Section 6]{KY2}. Thus we conclude that on the event $\Xi$ the complement of $I(D)$ contains no eigenvalues of $Q$.

The next step of the proof consists in making sure that the allowed neighbourhoods $I_i(D)$ contain exactly the right number of outliers; the counting argument (sketched in the steps (ii) and (iii) at the beginning of this section) follows that of \cite[Section 6]{KY2}. First we consider the case $D = D(0)$ where for all $i \neq j \in \cal O_{\epsilon}^+$ we have $d_i(0),d_j(0) \geq 2$ and $\abs{d_i(0) - d_j(0)} \geq 1$, and show that each interval $\{I_i(D(0)) \col i \in \cal O_{\epsilon}^+\}$ contains exactly one eigenvalue of $Q$ (see \cite[Proposition 6.6]{KY2}). We then deduce the general case by a continuity argument, by choosing an appropriate continuous path $(D(t))_{t \in [0,1]}$ joining the initial configuration $D(0)$ to the desired final configuration $D = D(1)$. The continuity argument requires the existence of a gap in the set $I(D)$ to the left of $\bigcup_{i \in \cal O_{4 \epsilon}^+} I_i(D)$. The existence of such a gap follows easily from the definition of $I(D)$ and the fact that $\abs{\cal R}$ is bounded. The details are the same as in \cite[Section 6.5]{KY2}. Hence \eqref{outl_bound_eps_1} follows. Moreover, \eqref{outl_bound_eps_2} follows from the same argument combined with Corollary \ref{cor: interlacing} for a lower bound on $\mu_i$. This concludes the analysis of the outliers to the right of the bulk spectrum.

The case of outliers to the left of the bulk spectrum is analogous. Here we assume that $\phi < 1 - \tau$. The argument is exactly the same as for $d_i > 0$, except that we use the bound \eqref{extended_1_Q} to the left of the bulk spectrum as well as $\abs{\lambda_i - \gamma_-} \leq K^{-2/3 + \epsilon}$ for $i \in \qq{K-2r, K}$ with high probability.
\end{proof}

\begin{proof}[Proof of Theorem \ref{thm:sticking}]
We only give the proof of \eqref{sticking_1}; the proof of \eqref{sticking_2} is analogous. Fix $\epsilon > 0$.
By Theorem \ref{thm: outlier locations}, Theorem \ref{thm: cov-rig}, Theorem \ref{thm: IMP gen}, Lemma \ref{lem: Qij iso}, and Remark \ref{rem:all_z}, there exists a high-probability event $\Xi \equiv \Xi_N(\epsilon)$ satisfying the following conditions.
\begin{enumerate}
\item
We have
\begin{equation} \label{Xi_bulk_edge}
\ind{\Xi} \abs{\mu_{s_+ + 1} - \gamma_+} \;\leq\; K^{-2/3 + \epsilon}\,, \qquad 
\ind{\Xi} \abs{\lambda_i-\gamma_i} \;\leq\; i^{-1/3}K^{-2/3 + \epsilon} \qquad (i \leq (1 - \tau)K)\,.
\end{equation}
\item
For $z \in \f S(\epsilon, K)$ we have
\begin{equation}
\label{Xi_W}
\ind{\Xi} \normb{W(z) - w_{\phi}(z)} \;\leq\; K^\epsilon \pBB{\sqrt{\frac{\im w_{\phi}(z)}{K \eta}} + \frac{1}{K \eta}}
\end{equation}
and
\begin{equation} \label{Xi_G}
\max_{i,j} \absb{\scalar{\f v_i}{G(z) \f v_j} - m_{\phi^{-1}}(z) \delta_{ij}} \;\leq\; K^\epsilon \pBB{\sqrt{\frac{\im m_{\phi^{-1}}(z)}{M \eta}} + \frac{1}{M \eta}}\,.
\end{equation}
\end{enumerate}

For the following we fix a realization $H \in \Xi$.
We suppose first that
\begin{equation} \label{alpha_lower_bound}
\alpha_+ \;\geq\; K^{-1/3 + \epsilon}\,,
\end{equation}
and define $\eta \deq K^{-1 + 2 \epsilon} \alpha_+^{-1}$.
Now suppose that $x$ satisfies
\begin{equation} \label{conditions_x_sticking}
x \;\in\; \qb{\gamma_+ - 1, \gamma_+ + K^{-2/3 + 2 \epsilon}}\,, \qquad \dist(x, \sigma(H)) \;>\; \eta\,.
\end{equation}
We shall show, using \eqref{pert1}, that any $x$ satisfying \eqref{conditions_x_sticking} cannot be an eigenvalue of $Q$. First we deduce from \eqref{Xi_W} that
\begin{equation} \label{W_add_eta}
\normb{W(x) - W(x + \ii \eta)} \;\leq\; C (1 + \phi) \max_{i} \im G_{\f v_i \f v_i}(x + \ii \eta)\,.
\end{equation}
The estimate \eqref{W_add_eta} follows by spectral decomposition of $F(\cdot)$ together with the estimate $2 \abs{\lambda_i - x} \geq \sqrt{(\lambda_i - x)^2 + \eta^2}$ for all $i$. We get from \eqref{W_add_eta} and Lemma \ref{lemma: w} that
\begin{align*}
W(x) \;=\; w_\phi(x + \ii \eta) + O \pbb{\im w_\phi(x + \ii \eta) + \frac{K^{\epsilon}}{K \eta}}
\;=\; -1 + O \pB{\sqrt{\kappa(x)} + \sqrt{\eta} + K^{-\epsilon} \alpha_+^{-1}}\,,
\end{align*}
where we use the notation $A = B + O(t)$ to mean $\norm{A - B} \leq C t$.
Recalling \eqref{pert1}, we conclude that on the event $\Xi$ the value $x$ is not an eigenvalue of $Q$ provided
\begin{equation*}
\min_i\abs{1/d_i - 1} \;\geq\; K^{\epsilon/2} \pB{\sqrt{\kappa(x)} + \sqrt{\eta} + K^{-\epsilon} \alpha_+^{-1}}\,.
\end{equation*}
It is easy to check that this condition is satisfied if
\begin{equation*}
\kappa(x) + \eta \;\leq\; C K^{-\epsilon} \alpha_+^2\,,
\end{equation*}
which holds provided that
\begin{equation*}
\kappa(x) \;\leq\; C K^{-\epsilon} \alpha_+^2\,,
\end{equation*}
where we used \eqref{alpha_lower_bound}. Recalling \eqref{Xi_bulk_edge}, we therefore conclude that for $i \leq K^{1 - 2 \epsilon} \alpha_+^3$ the set 
\begin{equation*}
\hB{x \in \qb{\lambda_{i - r - 1}, \gamma_+ + K^{-2/3 + 2 \epsilon}} \col \dist(x, \sigma(H)) > K^{-1 + 2 \epsilon}\alpha_+^{-1}}
\end{equation*}
contains no eigenvalue of $Q$.

The next step of the proof is a counting argument (sketched in the steps (ii) and (iii) at the beginning of this section), which uses the eigenvalue interlacing from Lemma \ref{lem: interlacing}. They details are the same as in \cite[Section 6]{KY2}, and hence omitted here. The counting argument implies that for $i \leq K^{1 - 2 \epsilon} \alpha_+^3$ and assuming \eqref{alpha_lower_bound} we have
\begin{equation} \label{sticking estimate proof}
\abs{\mu_{i + s_+} - \lambda_i} \;\leq\; C K^{-1 + 2 \epsilon}\alpha_+^{-1}\,.
\end{equation}
What remains is to check \eqref{sticking estimate proof} for the cases $\alpha_+ < K^{-1/3 + \epsilon}$ and $i > K^{1 - 2 \epsilon} \alpha_+^3$.

Suppose first that $\alpha_+ < K^{-1/3 + \epsilon}$. Then using the rigidity from \eqref{Xi_bulk_edge} and interlacing from Corollary \ref{cor: interlacing} we find
\begin{equation*}
\abs{\mu_{i + s_+} - \lambda_i} \;\leq\; C \, i^{-1/3}K^{-2/3 + \epsilon} \;\leq\; C K^{-1 + 2 \epsilon}\alpha_+^{-1}\,,
\end{equation*}
where we used the trivial bound $i \geq 1$. Similarly, if $i > K^{1 - 2 \epsilon} \alpha_+^3$ satisfies $i \leq (1 - \tau)K$, we may repeat the same estimate.

We conclude that \eqref{sticking estimate proof} under the sole assumption that $i \leq (1 - \tau)K$. Since $\epsilon > 0$ was arbitrary, \eqref{sticking_1} follows.
\end{proof}

\section{Outlier eigenvectors} \label{sec:outliers}

In this section we focus on the outlier eigenvectors $\xi_a$, $a \in \cal O$. Here we in fact prove Theorem \ref{thm: outlier eigenvectors} under the stronger assumption
\begin{equation} \label{assumption on A}
1 + K^{-1/3 + \tau} \;\leq\; d_i \;\leq\; \tau^{-1} \qquad (i \in A)
\end{equation}
instead of $1 + K^{-1/3} \leq d_i \leq \tau^{-1}$. How to improve the lower bound from $1 + K^{-1/3 + \tau}$ to the claimed $K^{-1/3}$ requires a completely different approach, relying on eigenvector delocalization bounds, and is presented in Section \ref{sec:bulk} in conjunction with results for the non-outlier eigenvectors $\xi_a$, $a \notin \cal O$.

The proof of Theorem \ref{thm: outlier eigenvectors 2} is similar to that of Theorem \ref{thm: outlier eigenvectors}; one has to adapt the proof to cover the range $d_i \in [1 + \tau, \infty)$ instead of $d_i \in [1 + K^{-1/3}, \tau^{-1}]$. The key input is the extension of the spectral domain from Corollary \ref{cor:extension}. For the sake of brevity we omit the details of the proof of Theorem \ref{thm: outlier eigenvectors 2}, and focus solely on Theorem \ref{thm: outlier eigenvectors}.

The following proposition is the main result of this section.

\begin{proposition} \label{prop: outl2}
Fix $\tau > 0$. Suppose that $A$ satisfies \eqref{assumption on A}. Then for all $i,j = 1, \dots, M$ we have
\begin{multline} \label{vi_vj_result 2}
\scalar{\f v_i}{P_A \f v_j} \;=\; \delta_{ij} \ind{i \in A} u(d_i)
+ O_\prec \Biggl[ \frac{\ind{i,j \in A}}{(d_i - 1)^{1/4} (d_j - 1)^{1/4} M^{1/2}}
\\
+
\frac{\sqrt{\sigma_i \sigma_j}}{M} \pbb{\frac{1}{\nu_i} + \frac{\ind{i \in A}}{d_i - 1}} \pbb{\frac{1}{\nu_j} + \frac{\ind{j \in A}}{d_j - 1}}
+ \frac{\ind{i \in A} \ind{j \notin A} (d_i - 1)^{1/2} \sqrt{\sigma_j}}{(1 + \phi)^{1/4} \nu_j M^{1/2}}  + (i \leftrightarrow j)
\Biggr]\,,
\end{multline}
where the symbol $(i \leftrightarrow j)$ denotes the preceding terms with $i$ and $j$ interchanged.
\end{proposition}

Note that, under the assumption \eqref{assumption on A}, Theorem \ref{thm: outlier eigenvectors} is an easy consequence of Proposition \ref{prop: outl2}.
As explained above, the proof of Theorem \ref{thm: outlier eigenvectors} in full generality is given in Section \ref{sec:bulk}, where we give the additional argument required to relax \eqref{assumption on A}.

The rest of this section is devoted to the proof of Proposition \ref{prop: outl2}.
%We assume throughout that \eqref{assumption on A} holds. 

\subsection{Non-overlapping outliers}
We first prove a slightly stronger version of \eqref{vi_vj_result 2} under the additional \emph{non-overlapping condition}
\begin{equation} \label{non-overlapping}
\nu_i(A) \;\geq\; (d_i - 1)^{-1/2} K^{-1/2 + \delta}
\end{equation}
for all $i \in A$,
where $\delta > 0$ is a positive constant. This is a precise version of the second condition of \eqref{cone_condition2}, whose interpretation was given below \eqref{cone_condition2}: an outlier indexed by $A$ cannot overlap with an outlier indexed by $A^c$. Note, however, that there is no restriction on the outliers indexed by $A$ overlapping among themselves. The assumption \eqref{non-overlapping} will be removed in Section \ref{sec: overlapping}.
The main estimate for non-overlapping outliers is the following.
\begin{proposition} \label{prop: outl1}
Fix $\tau > 0$ and $\delta > 0$. Suppose that $A$ satisfies \eqref{assumption on A} and \eqref{non-overlapping} for all $i \in A$. Then for all $i,j = 1, \dots, M$ we have
\begin{multline} \label{vi_vj_result}
\scalar{\f v_i}{P_A \f v_j} \;=\; \delta_{ij} \ind{i \in A} u(d_i)
+ O_\prec \Biggl[ \frac{\ind{i,j \in A}}{(d_i - 1)^{1/4} (d_j - 1)^{1/4} M^{1/2}}
\\
+
\frac{\sqrt{\sigma_i \sigma_j}}{M} \pbb{\frac{1}{\nu_i} + \frac{\ind{i \in A}}{d_i - 1}} \pbb{\frac{1}{\nu_j} + \frac{\ind{j \in A}}{d_j - 1}}
+ \frac{\ind{i \in A} \ind{j \notin A} (d_i - 1)^{1/2} \sqrt{\sigma_j}}{(1 + \phi)^{1/4} \abs{d_i - d_j} M^{1/2}}  + (i \leftrightarrow j)
\Biggr]\,.
\end{multline}
\end{proposition}

\begin{remark}
The only difference between \eqref{vi_vj_result 2} and \eqref{vi_vj_result} is the term proportional to $\ind{j \notin A}$ on the last line. In order to prove \eqref{vi_vj_result 2} without the overlapping condition \eqref{non-overlapping}, it is necessary to start from the stronger bound \eqref{vi_vj_result}; see Section \ref{sec: overlapping} below.
\end{remark}

The rest of this subsection is devoted to the proof of Proposition \ref{prop: outl1}.
We begin by defining $\omega \deq \tau / 2$ and letting $\epsilon < \min\h{\tau/3, \delta}$ be a positive constant to be determined later. We choose a high-probability event $\Xi \equiv \Xi_N(\epsilon, \tau)$ (see Definition \ref{def: high probability}) satisfying the following conditions.
\begin{enumerate}
\item
We have
\begin{equation} \label{Q bound Xi}
\ind{\Xi} \absb{W_{ij}(z) - w_\phi(z) \delta_{ij}} \;\leq\; \abs{z - \gamma_+}^{-1/4} K^{-1/2 + \epsilon}
\end{equation}
for $i,j \in \cal R$,
large enough $K$, and all $z$ in the set
\begin{equation} \label{domain of contour}
\hb{z \in \C \col \re z \geq \gamma_+ + K^{-2/3 + \omega} \,,\, \abs{z} \leq \omega^{-1}}\,.
\end{equation}
\item
For all $i$ satisfying $1 + K^{-1/3} \leq d_i \leq \omega^{-1}$ we have
\begin{equation} \label{outlier loc on Xi}
\ind{\Xi} \abs{\mu_{i} - \theta(d_i)} \;\leq\; (d_i - 1)^{1/2} \, K^{-1/2 + \epsilon}\,.
\end{equation}
\item
We have
\begin{equation} \label{bound_mu_edge}
\ind{\Xi} \abs{\mu_{s_++1} - \gamma_+} \;\leq\; K^{-2/3+\epsilon}\,.
\end{equation}
\end{enumerate}
Note that such an event $\Xi$ exists. Indeed, \eqref{outlier loc on Xi} and \eqref{bound_mu_edge} may be satisfied using Theorem \ref{thm: outlier locations},
and \eqref{Q bound Xi} using Theorem \ref{lem: Qij iso} combined with Remark \ref{rem:all_z}.

For the sequel we fix a realization $H \in \Xi$ satisfying the conditions (i)--(iii) above. Hence, the rest of the proof of Proposition \ref{prop: outl1} is entirely deterministic, and the randomness only enters in ensuring that $\Xi$ has high probability. Our starting point is a contour integral representation of the projection $P_A$. In order to construct the contour, we define for each $i \in A$ the radius
\begin{equation} \label{def_rho_i}
\rho_i \;\deq\; \frac{\nu_i \wedge (d_i - 1)}{2}\,.
\end{equation}
We define the contour $\Gamma \deq \partial \Upsilon$ as the boundary of the union of discs $\Upsilon \deq \bigcup_{i \in A} B_{\rho_i}(d_i)$,
where $B_\rho(d)$ is the open disc of radius $\rho$ around $d$. We shall sometimes need the decomposition $\Gamma = \bigcup_{i \in A} \Gamma_i$, where $\Gamma_i \deq \Gamma \cap \partial B_{\rho_i}(d_i)$.
See Figure \ref{fig: contour} for an illustration of $\Gamma$.
\begin{figure}[ht!]
\begin{center}
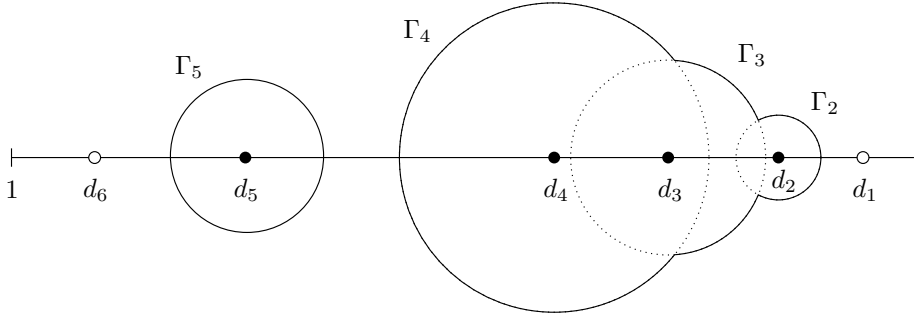
\end{center}
\caption{The integration contour $\Gamma = \bigcup_{i \in A} \Gamma_i$. In this example $\Gamma$ consists of two components, and we have $\abs{\cal R} = 6$ with $A = \{2,3,4,5\}$. We draw the locations of $d_i$ with $i \in A$ using black dots and the other $d_i$ using white dots. The contour is constructed by drawing circles of radius $\rho_i$ around each $d_i$ for $i \in A$ (depicted with dotted lines). The piece $\Gamma_i$ consists of the points on the circle centred at $d_i$ that lie outside all other circles. \label{fig: contour}} 
\end{figure}

We shall have to use the estimate \eqref{Q bound Xi} on the set $\ol{\theta(\Upsilon)}$. Its applicability is an immediate consequence of the following lemma.

\begin{lemma}
The set $\ol{\theta(\Upsilon)}$ lies in \eqref{domain of contour}.
\end{lemma}
\begin{proof}
It is easy to check that $\theta(\zeta) \leq \omega^{-1}$ for all $\zeta \in \Upsilon$. In order to check the lower bound on $\re \theta(\zeta)$, we note that for any $\alpha \in (0,1)$ there exists a constant $c \equiv c(\alpha, \tau)$ such that
\begin{equation*}
\re \theta(\zeta) \;\geq\; \gamma_+ + c (\re \zeta - 1)^2
\end{equation*}
for $\re \zeta \geq 1$, $\abs{\im \zeta} \leq \alpha (\re \zeta - 1)$, and $\abs{\zeta} \leq \tau^{-1}$. Now the claim follows easily from $\re \zeta \geq 1 + K^{-1/3 + \tau}/2$ for all $\zeta \in \Upsilon$, by choosing $\alpha = 1/\sqrt{3}$.
\end{proof}

\begin{lemma} \label{lem:contour}
Each outlier $\{\mu_{i}\}_{i \in A}$ lies in $\theta(\Upsilon)$, and all other eigenvalues of $Q$ lie in the complement of $\ol{\theta(\Upsilon)}$.
\end{lemma}
\begin{proof}
It suffices to prove that (a) for each $i \in A$ we have $\mu_{i} \in \theta(B_{\rho_i}(d_i))$ and (b) all the other eigenvalues $\mu_j$ satisfy $\mu_j \notin\theta(B_{\rho_i}(d_i))$ for all $i \in A$.

In order to prove (a), we note that
\begin{equation} \label{bound on rho}
\rho_i \;\geq\; \frac{1}{2} \, (d_i - 1)^{-1/2} K^{-1/2 + \delta}\,,
\end{equation}
for $i \in A$, as follows from \eqref{non-overlapping} and \eqref{assumption on A}. Using
\begin{equation} \label{theta prime small z}
\abs{\theta'(\zeta)} \;\asymp\; \abs{\zeta - 1}  \qquad (\re \zeta \geq 1 \,,\, \abs{\zeta} \leq \tau^{-1})\,,
\end{equation}
it is then not hard to get (a) from \eqref{bound on rho} and \eqref{outlier loc on Xi}.

In order to prove (b), we consider the two cases (i) $1 + K^{-1/3} \leq d_j \leq \omega^{-1}$ with $j \notin A$, and (ii) and $j \geq s_+ + 1$. In the case (i), the claim (b) follows using \eqref{outlier loc on Xi}, \eqref{theta prime small z}, and \eqref{non-overlapping}. In the case (ii), the claim (b) follows from \eqref{bound_mu_edge} and the estimate
\begin{equation} \label{theta small z}
\abs{\theta(\zeta) - \gamma_+} \;\asymp\; \abs{\zeta - 1}^2 \qquad (\re \zeta \geq 1 , \abs{\zeta} \leq \tau^{-1})\,.
\end{equation}
This concludes the proof.
\end{proof}
Using the spectral decomposition of $\wt G(z)$, Lemma \ref{lem:contour}, and the residue theorem, we may write the projection $P_A$ as
\begin{equation*}
P_A \;=\; - \frac{1}{2 \pi \ii} \oint_{\theta(\Gamma)} \wt G(z) \, \dd z \;=\; - \frac{1}{2 \pi \ii} \oint_{\Gamma} \wt G(\theta(\zeta)) \, \theta'(\zeta) \, \dd \zeta\,.
\end{equation*}
Hence we get from \eqref{pert3} that
\begin{equation} \label{outlier integral}
V^* P_A V \;=\; \phi^{-1/2} \, \frac{1}{2 \pi \ii} \oint_{\Gamma} \frac{\sqrt{1 + \phi^{1/2} D}}{D} \frac{1}{D^{-1} + W(\theta(\zeta))} \frac{\sqrt{1 + \phi^{1/2} D}}{D} \frac{\theta'(\zeta)}{\theta(\zeta)} \, \dd \zeta\,.
\end{equation}
This is the desired integral representation of $P_A$.

We first use \eqref{outlier integral} to compute $\scalar{\f v_i}{P_A \f v_j}$ in the case $i,j \in \cal R$, where $\f v_i$ and $\f v_j$ lie in the range of $V$. In that case we get from \eqref{outlier integral} that
\begin{equation*}
\scalar{\f v_i}{P_A \f v_j} \;=\; \frac{\sqrt{\sigma_i \sigma_j}}{\phi^{1/2}d_i d_j} \, \frac{1}{2 \pi \ii} \oint_{\Gamma} \pbb{\frac{1}{D^{-1} + W(\theta(\zeta))}}_{ij} \frac{\theta'(\zeta)}{\theta(\zeta)} \, \dd \zeta\,.
\end{equation*}
We now perform a resolvent expansion on the denominator
\begin{equation} \label{resolvent expansion}
D^{-1} + W(\theta) \;=\; (D^{-1} + w_\phi(\theta)) - \Delta(\theta) \,, \qquad \Delta(\theta) \;\deq\; w_\phi(\theta) - W(\theta)\,.
\end{equation}
Thus we get
\begin{equation} \label{123split}
\scalar{\f v_i}{P_A \f v_j} \;=\; \frac{\sqrt{\sigma_i \sigma_j}}{\phi^{1/2}d_i d_j} \pb{S^{(0)}_{ij} + S^{(1)}_{ij} + S^{(2)}_{ij}}\,,
\end{equation}
where we defined
\begin{align}
S^{(0)}_{ij} &\;\deq\; \frac{1}{2 \pi \ii} \oint_{\Gamma} \pbb{\frac{1}{D^{-1} + w_\phi(\theta(\zeta))}}_{ij} \frac{\theta'(\zeta)}{\theta(\zeta)} \, \dd \zeta\,,
\\
S^{(1)}_{ij} &\;\deq\; \frac{1}{2 \pi \ii} \oint_{\Gamma} \pbb{\frac{1}{D^{-1} + w_\phi(\theta(\zeta))} \Delta(\theta(\zeta)) \frac{1}{D^{-1} + w_\phi(\theta(\zeta))}}_{ij} \frac{\theta'(\zeta)}{\theta(\zeta)} \, \dd \zeta\,,
\\ \label{def_S^2}
S^{(2)}_{ij} &\;\deq\; \frac{1}{2 \pi \ii} \oint_{\Gamma} \pbb{\frac{1}{D^{-1} + w_\phi(\theta(\zeta))} \Delta(\theta(\zeta)) \frac{1}{D^{-1} + W(\theta(\zeta))} \Delta(\theta(\zeta)) \frac{1}{D^{-1} + w_\phi(\theta(\zeta))} }_{ij} \frac{\theta'(\zeta)}{\theta(\zeta)} \, \dd \zeta\,.
\end{align}

We begin by computing
\begin{equation} \label{S^0 result}
S^{(0)}_{ij} \;=\; \delta_{ij} \frac{1}{2 \pi \ii} \oint_{\Gamma} \pbb{\frac{1}{d_i^{-1} - \zeta^{-1}}}_{ij} \frac{\theta'(\zeta)}{\theta(\zeta)} \, \dd \zeta
\;=\; \delta_{ij} \ind{i \in A} \frac{d_i^2 - 1}{\theta(d_i)}\,,
\end{equation}
where we used Cauchy's theorem, \eqref{w_theta_relation}, and the fact that $d_i$ lies in $\Upsilon$ if and only if $i \in A$.

Next, we estimate
\begin{equation} \label{S_1 starting point}
S^{(1)}_{ij} \;=\; d_i d_j \, \frac{1}{2 \pi \ii} \oint_{\Gamma} \frac{f_{ij}(\zeta)}{(\zeta - d_i)(\zeta - d_j)} \, \dd \zeta \,, \qquad f_{ij}(\zeta) \;\deq\; \zeta^2 \Delta(\theta(\zeta)) \frac{\theta'(\zeta)}{\theta(\zeta)}\,,
\end{equation}
using the fact that $f_{ij}$ is holomorphic inside $\Gamma$ and satisfies the bounds
\begin{equation} \label{fij bounds}
\abs{f_{ij}(\zeta)} \;\leq\; C \phi^{1/2} (1 + \phi)^{-1} \abs{\zeta - 1}^{1/2} K^{-1/2 + \epsilon}\,, \qquad
\abs{f_{ij}'(\zeta)} \;\leq\; C \phi^{1/2} (1 + \phi)^{-1} \abs{\zeta - 1}^{-1/2} K^{-1/2 + \epsilon}\,.
\end{equation}
The first bound of \eqref{fij bounds} follows from \eqref{Q bound Xi}, \eqref{theta prime small z}, and \eqref{theta small z}. The second bound of \eqref{fij bounds} follows by plugging the first one into
\begin{equation*}
f_{ij}'(\zeta) \;=\; \frac{1}{2 \pi \ii} \oint_{\cal C} \frac{f_{ij}(\xi)}{(\xi - \zeta)^2} \, \dd \xi\,,
\end{equation*}
where the contour $\cal C$ is the circle of radius $\abs{\zeta - 1}/2$ centred at $\zeta$. (By assumptions on $\epsilon$ and $\omega$, the function $f_{ij}$ is holomorphic in a neighbourhood of the closed interior of $\cal C$.)

In order to estimate \eqref{S_1 starting point}, we consider the three cases (i) $i,j \in A$, (ii) $i \in A$, $j \notin A$, (iii) $i \notin A$, $j \in A$. Note that \eqref{S_1 starting point} vanishes if $i,j \notin A$.
We start with the case (i). Suppose first that $i \neq j$ and $d_i \neq d_j$. Then we find
\begin{equation*}
\abs{S^{(1)}_{ij}} \;=\; \abs{d_i d_j} \absbb{\frac{f_{ij}(d_i) - f_{ij}(d_j)}{d_i - d_j}} \;\leq\; \frac{\abs{d_i d_j}}{\abs{d_i - d_j}} \absbb{\int_{d_i}^{d_j} \abs{f'_{ij}(t)} \, \dd t} \;\leq\; \frac{C \abs{d_i d_j} \phi^{1/2}}{(1 + \phi) (d_i - 1)^{1/4} (d_j - 1)^{1/4}} K^{-1/2 + \epsilon}\,.
\end{equation*}
A simple limiting argument shows that this bound is also valid for $d_i = d_j$ and $i = j$. Next, in the case (ii) we get from \eqref{fij bounds}
\begin{equation*}
\abs{S^{(1)}_{ij}} \;=\; \frac{\abs{d_i d_j f_{ij}(d_i)}}{\abs{d_i - d_j}} \;\leq\; \frac{C \abs{d_i d_j} \phi^{1/2} (d_i - 1)^{1/2}}{(1 + \phi) \abs{d_i - d_j}} K^{-1/2 + \epsilon}\,.
\end{equation*}
A similar estimate holds for the case (iii). Putting all three cases together, we find
\begin{equation} \label{S^1 result}
\abs{S^{(1)}_{ij}} \;\leq\; \frac{C \ind{i,j \in A} \abs{d_i d_j} \phi^{1/2}}{(1 + \phi) (d_i - 1)^{1/4} (d_j - 1)^{1/4}} K^{-1/2 + \epsilon} + 
\frac{C \ind{i \in A} \ind{j \notin A}\abs{d_i d_j} \phi^{1/2} (d_i - 1)^{1/2}}{(1 + \phi) \abs{d_i - d_j}} K^{-1/2 + \epsilon} + (i \leftrightarrow j)\,.
\end{equation}

What remains is the estimate of $S_{ij}^{(2)}$. Here residue calculations are unavailable, and the precise choice of the contour $\Gamma$ is crucial. We use the following basic estimate to control the integral.
\begin{lemma} \label{lem:radius_bound}
For $k \in A$, $l \in \cal R$, and $\zeta \in \Gamma_k$ we have
\begin{equation*}
\abs{\zeta - d_l} \;\asymp\; \rho_k + \abs{d_k - d_l}\,.
\end{equation*}
\end{lemma}
\begin{proof}
The upper bound $\abs{\zeta - d_l} \leq \rho_k + \abs{d_k - d_l}$ is trivial, so that we only focus on the lower bound. Suppose first that $l \notin A$. Then we get $\abs{\zeta - d_l} \geq \abs{d_k - d_l} - \rho_k$, from which the claim follows since $\abs{d_k - d_l} \geq 2 \rho_k$ by \eqref{def_rho_i}.

For the remainder of the proof we may therefore suppose that $l \in A$.
Define $\delta \deq \abs{d_k - d_l} - \rho_k - \rho_l$, the distance between the discs $D_{\rho_k}(d_k)$ and $D_{\rho_l}(d_l)$ (see Figure \ref{fig: contour}). We consider the two cases $4 \delta \leq \abs{d_k - d_l}$ and $4 \delta > \abs{d_k - d_l}$ separately.

Suppose first that $4 \delta \leq \abs{d_k - d_l}$. Then by definition of $\delta$ we have $\abs{d_k - d_l} \leq \frac{4}{3} (\rho_k + \rho_l)$. Now a simple estimate using the definition of $\rho_i$ yields $\rho_k / 5 \leq \rho_l \leq 5 \rho_k$, from which we conclude $\abs{d_k - d_l} \leq 8 \rho_k$. The claim now follows from the bound $\abs{\zeta - d_l} \geq \rho_l$.

Suppose now that $4 \delta > \abs{d_k - d_l}$. Hence $\rho_k + \rho_l \leq \frac{3}{4} \abs{d_k - d_l}$, so that in particular $\rho_k \leq \abs{d_k - d_l}$. Thus we get
\begin{equation*}
\abs{\zeta - d_l} \;\geq\; \abs{d_k - d_l} - \rho_k - \rho_l \;\geq\; \frac{1}{4} \abs{d_k - d_l} \;\geq\; \frac{1}{8} \pb{\abs{d_k - d_l} + \rho_k}\,.
\end{equation*}
This concludes the proof.
\end{proof}

From \eqref{def_S^2}, \eqref{Q bound Xi}, \eqref{theta prime small z}, and \eqref{theta small z} we get
\begin{equation} \label{S^2 1}
\absb{S_{ij}^{(2)}} \;\leq\; C \oint_{\Gamma} \frac{\phi^{1/2} \abs{d_i d_j} K^{-1 + 2 \epsilon}}{(1 + \phi) \abs{\zeta - d_i} \abs{\zeta - d_j}} \normbb{\frac{1}{D^{-1} + W(\theta(\zeta))}} \, \abs{\dd \zeta}\,,
\end{equation}
where we also used the estimate $\abs{\theta(\zeta)} \asymp \phi^{-1/2} (1 + \phi)$ for $\zeta \in \Gamma$.

In order to estimate the matrix norm, we observe that for $\zeta \in \Gamma_k$ we have on the one hand
\begin{equation*}
\norm{W(\theta) - w_\phi(\theta)} \;\leq\; (d_k - 1)^{-1/2} K^{-1/2 + \epsilon}
\end{equation*}
from \eqref{Q bound Xi} and on the other hand
\begin{equation*}
\abs{w_\phi(\theta) - d_l^{-1}} \;\geq\; c (\abs{\zeta - d_l} \wedge 1) \;\geq\; c \abs{\zeta - d_k} \;=\; c \rho_k \;\geq\; c (d_k - 1)^{-1/2} K^{-1/2 + \delta}
\end{equation*}
for any $l \in \cal R$, where in the last step we used \eqref{bound on rho}. Since $\epsilon < \delta$, these estimates combined with a resolvent expansion give the bound
\begin{equation*}
\normbb{\frac{1}{D^{-1} + W(\theta(\zeta))}} \;\leq\; \frac{1}{\min_{b \in \cal R} \abs{w_\phi(\theta) - d_b^{-1}} - \norm{W(\theta) - w_\phi(\theta)}} \;\leq\; \frac{C}{\rho_k}
\end{equation*}
for $\zeta \in \Gamma_k$. Decomposing the integration contour in \eqref{S^2 1} as $\Gamma = \bigcup_{k \in A} \Gamma_k$, and recalling that $\Gamma_k$ has length bounded by $2 \pi \rho_k$, we get from Lemma \ref{lem:radius_bound}
\begin{equation} \label{S^2_step1}
\absb{S_{ij}^{(2)}} \;\leq\; C \sum_{k \in A} \sup_{\zeta \in \Gamma_k} \frac{\phi^{1/2} \abs{d_i d_j} K^{-1 + 2 \epsilon}}{(1 + \phi) \abs{\zeta - d_i} \abs{\zeta - d_j}}
\;\leq\; C \sum_{k \in A}  \frac{\phi^{1/2} \abs{d_i d_j} K^{-1 + 2 \epsilon}}{(1 + \phi) (\rho_k + \abs{d_k - d_i}) (\rho_k + \abs{d_k - d_j})}\,.
\end{equation}
We estimate the right-hand side using Cauchy-Schwarz. For $i \notin A$ we find, using \eqref{def_rho_i},
\begin{equation} \label{CS1}
\sum_{k \in A} \frac{1}{(\rho_k + \abs{d_k - d_i})^2} \;\leq\; \sum_{k \in A} \frac{1}{\abs{d_k - d_i}^2} \;\leq\; \frac{C}{\nu_i^2}\,.
\end{equation}
For $i \in A$ we use \eqref{def_rho_i} and the estimate $\rho_k + \abs{d_i - d_k} \geq \rho_i$ for all $k \in A$ to get
\begin{equation} \label{CS2}
\sum_{k \in A} \frac{1}{(\rho_k + \abs{d_k - d_i})^2} \;\leq\; \frac{C}{\rho_i^2} \;\leq\; \frac{C}{\nu_i^2} + \frac{C}{(d_i - 1)^2}\,.
\end{equation}
From \eqref{S^2_step1}, \eqref{CS1}, and \eqref{CS2}, we get
\begin{equation} \label{S^2 result}
\absb{S_{ij}^{(2)}} \;\leq\; \frac{C \phi^{1/2} \abs{d_i d_j} K^{-1 + 2 \epsilon}}{1 + \phi} \pbb{\frac{1}{\nu_i} + \frac{\ind{i \in A}}{d_i - 1}} \pbb{\frac{1}{\nu_j} + \frac{\ind{j \in A}}{d_j - 1}}\,.
\end{equation}

Recall that $M \asymp (1 + \phi) K$.
Hence, plugging \eqref{S^0 result}, \eqref{S^1 result}, and \eqref{S^2 result} into \eqref{123split}, we find
\begin{multline} \label{vPv for R}
\scalar{\f v_i}{P_A \f v_j} \;=\; \delta_{ij} \ind{i \in A} u(d_i)
+ O \Biggl[ \frac{\ind{i,j \in A} K^\epsilon}{(d_i - 1)^{1/4} (d_j - 1)^{1/4} M^{1/2}}
\\
+
\frac{\sqrt{\sigma_i \sigma_j} K^{2 \epsilon}}{M} \pbb{\frac{1}{\nu_i} + \frac{\ind{i \in A}}{d_i - 1}} \pbb{\frac{1}{\nu_j} + \frac{\ind{j \in A}}{d_j - 1}}
+ \frac{\ind{i \in A} \ind{j \notin A} (d_i - 1)^{1/2} \sqrt{\sigma_j} K^\epsilon}{(1 + \phi)^{1/4} \abs{d_i - d_j} M^{1/2}}  + (i \leftrightarrow j)
\Biggr]\,.
\end{multline}
We have proved \eqref{vPv for R} under the assumption that $i, j \in \cal R$. The general case is an easy corollary. For general $i,j \in \qq{1,M}$, we define $\wh {\cal R} \deq \cal R \cup \{i,j\}$ and consider
\begin{equation*}
\wh \Sigma \;\deq\; 1 + \phi^{1/2} \wh V \wh D \wh V^*\,, \qquad \wh V \;\deq\; [\f v_k]_{k \in \wh {\cal R}} \,, \qquad \wh D \;\deq\; \diag(\wh d_k)_{k \in \wh {\cal R}}\,,
\end{equation*}
where $\wh d_k \deq d_k$ for $k \in \cal R$ and $\wh d_k \in (0,1/2)$ for $k \in \wh {\cal R} \setminus \cal R$.
Since $\abs{\wh {\cal R}} \leq r + 2$ and $\wh D$ is invertible, we may apply the result \eqref{vPv for R} to this modified model. Now taking the limit $\wh d_k \to 0$ for $k \in \wh {\cal R} \setminus \cal R$ in \eqref{vPv for R} concludes the proof in the general case. Now Proposition \ref{prop: outl1} follows since $\epsilon$ may be chosen arbitrarily small. This concludes the proof of Proposition \ref{prop: outl1}.

\subsection{Removing the non-overlapping assumption} \label{sec: overlapping}
In this subsection we complete the proof of Proposition \ref{prop: outl2} by extending Proposition \ref{prop: outl1} to the case where \eqref{non-overlapping} does not hold.

\begin{proof}[Proof of Proposition \ref{prop: outl2}]
Let $\delta < \tau/4$. We say that $i,j \in \cal O_{\tau/2}^+$ \emph{overlap} if $\abs{d_i - d_j} \leq (d_i - 1)^{-1/2} K^{-1/2 + \delta}$ or $\abs{d_i - d_j} \leq (d_j - 1)^{-1/2} K^{-1/2 + \delta}$. For $A \subset \cal O_{\tau}^+$ we introduce sets $S(A), L(A) \subset \cal O_{\tau/2}^+$ satisfying $S(A) \subset A \subset L(A)$. Informally, $S(A) \subset A$ is the largest subset of indices of $A$ that do not overlap with its complement. It is by definition constructed by successively choosing $k \in A$, such that $k$ overlaps with an index of $A^c$, and removing $k$ from $A$; this process is repeated until no such $k$ exists. One can check that the result is independent of the choice of $k$ at each step. Note that $S(A)$ may be empty.

Informally, $L(A) \supset A$ is the smallest subset of indices in $\cal O_{\tau/2}^+$ that do not overlap with its complement. It is by definition constructed by successively choosing $k \in \cal O_{\tau/2}^+ \setminus A$, such that $k$ overlaps with an index of $A$, and adding $k$ to $A$; this process is repeated until no such $k$ exists. One can check that the result is independent of the choice of $k$ at each step. See Figure \ref{fig: S and L} for an illustration of $S(A)$ and $L(A)$. Throughout the following we shall repeatedly make use of the fact that, for any $A \subset \cal O_{\tau}^+$, Proposition \ref{prop: outl1} is applicable with $(\tau, A)$ replaced by $(\tau/2, S(A))$ or $(\tau/2, L(A))$.

\begin{figure}[ht!]
\begin{center}
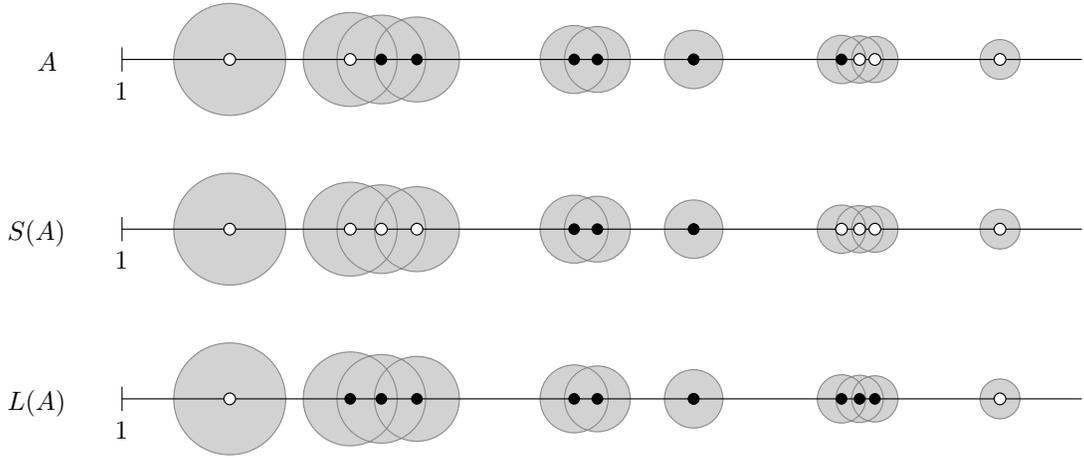
\end{center}
\caption{The construction of the sets $S(A)$ and $L(A)$. The black and white dots are the outlier indices $\h{d_i \col i \in \cal O_{\tau/2}^+}$, contained in the interval $[1, \infty)$. Around each outlier index $d_i$ we draw a grey circle of radius $(d_i - 1)^{-1/2} K^{-1/2 + \delta}$. By definition, two dots overlap if one is contained in the grey circle of the other. The three pictures depict (from top to bottom) the sets $A$, $S(A)$, and $L(A)$, respectively. In each case, the given set is drawn using black dots and its complement using white dots. \label{fig: S and L}} 
\end{figure}

After these preparations, we move on to the proof of \eqref{vi_vj_result 2}. We divide the argument into four steps.

\subsubsection*{(a) $i = j \notin A$}
We consider two cases, $i \notin L(A)$ and $i \in L(A)$. Suppose first that $i \notin L(A)$. Using that $\abs{\cal R}$ is bounded, it is not hard to see that $\nu_i(A) \asymp \nu_i(L(A))$. We now invoke Proposition \ref{prop: outl1} and get
\begin{equation} \label{i_R-A1}
\scalar{\f v_i}{P_A \f v_i} \;\leq\; \scalar{\f v_i}{P_{L(A)} \f v_i} \;\prec\; \frac{\sigma_i}{M \nu_i(L(A))^2} \;\leq\; C \frac{\sigma_i}{M \nu_i(A)^2}\,.
\end{equation}
In the complementary case, $i \in L(A)$, a simple argument yields
\begin{equation} \label{nuA_nuLA}
\nu_i(A) \;\leq\; C (d_i - 1)^{-1/2} K^{-1/2 + \delta} \;\leq\; C \nu_i(L(A))\,,
\end{equation}
as well as $\sigma_i \asymp 1 + \phi^{1/2}$.
From Proposition \ref{prop: outl1} we therefore get
\begin{multline*}
\scalar{\f v_i}{P_A \f v_i} \;\leq\; \scalar{\f v_i}{P_{L(A)} \f v_i} \;\prec\; \frac{d_i - 1}{1 + \phi^{1/2}} + \frac{1}{(d_i - 1)^{1/2} M^{1/2}} + \frac{1 + \phi^{1/2}}{M \nu_i(L(A))^2} + \frac{1 + \phi^{1/2}}{M (d_i - 1)^2}
\\
\leq\; C K^{2 \delta} \frac{d_i - 1}{1 + \phi^{1/2}} \;\leq\; C K^{2 \delta} \frac{1 + \phi^{1/2}}{M \nu_i(A)^2} \;\leq\; C K^{2 \delta} \frac{\sigma_i}{M \nu_i(A)^2}\,,
\end{multline*}
where we used that $M \asymp (1 + \phi) K$. Recalling \eqref{i_R-A1}, we conclude
\begin{equation} \label{tR-A gen}
\scalar{\f v_i}{P_A \f v_i} \;\prec\; K^{2 \delta} \frac{\sigma_i}{M \nu_i(A)^2} \qquad (i \notin A)\,.
\end{equation}

\subsubsection*{(b) $i = j \in A$}
We consider the two cases $i \in S(A)$ and $i \notin S(A)$. Suppose first that $i \in S(A)$. We write
\begin{equation} \label{vAv split}
\scalar{\f v_i}{P_A \f v_i} \;=\; \scalar{\f v_i}{P_{S(A)} \f v_i} + \scalar{\f v_i}{P_{A \setminus S(A)} \f v_i}\,.
\end{equation}
We compute the first term of \eqref{vAv split} using Proposition \ref{prop: outl1} and the observation that $\nu_i(A) \asymp \nu_i(S(A))$:
\begin{equation*}
\scalar{\f v_i}{P_{S(A)} \f v_i} \;=\; u(d_i) + O_\prec \qBB{\frac{1}{(d_i - 1)^{1/2} M^{1/2}} + \frac{\sigma_i}{M} \pbb{\frac{1}{\nu_i(A)^2} + \frac{1}{(d_i - 1)^2}}}\,.
\end{equation*}
In order to estimate the second term of \eqref{vAv split}, we note that $\nu_i(A) \asymp \nu_i(A \setminus S(A))$. We therefore apply \eqref{tR-A gen} with $A$ replaced by $A \setminus S(A)$ to get
\begin{equation*}
\scalar{\f v_i}{P_{A \setminus S(A)} \f v_i} \;\prec\; K^{2 \delta} \frac{\sigma_i}{M \nu_i(A)^2}\,.
\end{equation*}
Going back to \eqref{vAv split}, we have therefore proved that
\begin{equation} \label{tA gen}
\scalar{\f v_i}{P_A \f v_i} \;=\; u(d_i) + K^{2 \delta} O_\prec \qBB{\frac{1}{(d_i - 1)^{1/2} M^{1/2}} + \frac{\sigma_i}{M} \pbb{\frac{1}{\nu_i(A)^2} + \frac{1}{(d_i - 1)^2}}}
\end{equation}
for $i \in S(A)$.

Next, we consider the case $i \notin S(A)$. Now we have \eqref{nuA_nuLA}, so that Proposition \ref{prop: outl1} yields
\begin{equation*}
\scalar{\f v_i}{P_A \f v_i} \;\leq\; \scalar{\f v_i}{P_{L(A)} \f v_i} \;\prec\; u(d_i) + \frac{1}{(d_i - 1)^{1/2} M^{1/2}} + \frac{\sigma_i}{M} \pbb{\frac{1}{\nu_i(A)^2} + \frac{1}{(d_i - 1)^2}}\,.
\end{equation*}
By \eqref{nuA_nuLA} and $M \asymp (1 + \phi) K$, we have
\begin{equation*}
u(d_i) \;=\; \frac{\sigma_i}{\phi^{1/2} \theta(d_i)} (1 - d_i^{-2}) \;\leq\; C K^{2 \delta} \frac{1}{(d_i - 1)^{1/2} M^{1/2}}\,,
\end{equation*}
from which we deduce \eqref{tA gen} also in the case $i \notin S(A)$.

\subsubsection*{(c) $i \neq j$ and $i \notin A$ or $j \notin A$}
From cases (a) and (b) (i.e.\ \eqref{tR-A gen} and \eqref{tA gen}), combined with the estimate
\begin{equation*}
\absb{\scalar{\f v_i}{P_A \f v_j}}^2 \;\leq\; \scalar{\f v_i}{P_A \f v_i} \scalar{\f v_j}{P_A \f v_j}\,,
\end{equation*}
we find, assuming $i \notin A$ or $j \notin A$, that \eqref{vi_vj_result 2} holds with an additional factor $K^{2 \delta}$ multiplying the right-hand side.

\subsubsection*{(d) $i \neq j$ and $i,j \in A$}
We now deal with the last remaining case by using the splitting
\begin{equation} \label{vAv split 2}
\scalar{\f v_i}{P_A \f v_j} \;=\; \scalar{\f v_i}{P_{S(A)} \f v_j} + \scalar{\f v_i}{P_{A \setminus S(A)} \f v_j}\,.
\end{equation}
The goal is to show that
\begin{equation} \label{goal_case_e}
\absb{\scalar{\f v_i}{P_A \f v_j}} \;\prec\; \frac{K^{2 \delta}}{(d_i - 1)^{1/4} (d_j - 1)^{1/4} M^{1/2}} + 
K^{2 \delta} \frac{\sqrt{\sigma_i \sigma_j}}{M} \pbb{\frac{1}{\nu_i(A)} + \frac{1}{d_i - 1}} \pbb{\frac{1}{\nu_j(A)} + \frac{1}{d_j - 1}}\,.
\end{equation}
Note that here $\sigma_i \asymp \sigma_j \asymp 1 + \phi^{1/2}$.
We consider the four cases (i) $i,j \in S(A)$, (ii) $i \in S(A)$ and $j \notin S(A)$, (iii) $i \notin S(A)$ and $j \in S(A)$, and (iv) $i,j \notin S(A)$.

Consider first the case (i). The first term of \eqref{vAv split 2} is bounded using Proposition \ref{prop: outl1} combined with $\nu_i(A) \asymp \nu_i(S(A))$ and $\nu_j(A) \asymp \nu_j(S(A))$. The second term of \eqref{vAv split 2} is bounded using \eqref{vi_vj_result 2} from case (c) combined with $\nu_i(A) \leq C \nu_i(A \setminus S(A))$ and $\nu_j(A) \leq C \nu_j(A \setminus S(A))$. This yields \eqref{goal_case_e} for $\scalar{\f v_i}{P_A \f v_j}$ in the case (i).

Next, consider the case (ii). For the first term of \eqref{vAv split 2} we use the estimates
\begin{equation} \label{basic_case_ii}
\nu_i(S(A)) \;\asymp\; \nu_i(A)\,, \qquad 
\nu_j(A) \;\leq\; C (d_j - 1)^{-1/2} K^{-1/2 + \delta} \;\leq\; C \nu_j(S(A))\,, \qquad
\nu_i(A) \;\leq\; C \abs{d_i - d_j}\,.
\end{equation}
Thus we get from \eqref{vi_vj_result}
\begin{align}
\absb{\scalar{\f v_i}{P_{S(A)} \f v_j}} &\;\prec\; \frac{1 + \phi^{1/2}}{M \nu_i(S(A)) \nu_j(S(A))} + \frac{1 + \phi^{1/2}}{M \nu_j(S(A)) (d_i - 1)} + \frac{(d_i - 1)^{1/2}}{M^{1/2} \abs{d_i - d_j}} 
\notag \\ \label{case_ii_1}
&\;\leq\; \frac{1 + \phi^{1/2}}{M \nu_i(A) \nu_j(A)} + \frac{1 + \phi^{1/2}}{M \nu_j(A) (d_i - 1)} + \frac{(d_i - 1)^{1/2}}{M^{1/2} \abs{d_i - d_j}}\,.
\end{align}
In order to estimate the last term, we first assume that $d_j \leq d_i$ and $d_i - 1 \leq 2 \abs{d_i - d_j}$. Then we find
\begin{equation} \label{case_ii_2}
\frac{(d_i - 1)^{1/2}}{M^{1/2} \abs{d_i - d_j}} \;\leq\; \frac{2}{M^{1/2} (d_i - 1)^{1/2}} \;\leq\; \frac{2}{M^{1/2} (d_i - 1)^{1/4} (d_j - 1)^{1/4}}\,.
\end{equation}
Conversely, if $d_i \leq d_j$ or $d_i - 1 \geq 2 \abs{d_i - d_j}$, we have $d_i - 1 \leq 2 (d_j - 1)$. Therefore, using \eqref{basic_case_ii} and the estimate $M \asymp (1 + \phi) K$, we get
\begin{equation} \label{case_ii_3}
\frac{(d_i - 1)^{1/2}}{M^{1/2} \abs{d_i - d_j}} \;\leq\; \frac{C (d_j - 1)^{1/2}}{M^{1/2} \nu_i(A)} \;\leq\; K^\delta \frac{1 + \phi^{1/2}}{M \nu_i(A) \nu_j(A)}\,.
\end{equation}
Putting \eqref{case_ii_1}, \eqref{case_ii_2}, and \eqref{case_ii_3} together, we may estimate the first term of \eqref{vAv split 2} in the case (ii) as
\begin{equation} \label{vAvcaseii}
\absb{\scalar{\f v_i}{P_{S(A)} \f v_j}} \;\prec\; \frac{1 + \phi^{1/2}}{M \nu_i(A) \nu_j(A)} + K^\delta \frac{1 + \phi^{1/2}}{M \nu_j(A) (d_i - 1)} + \frac{1}{M^{1/2} (d_i - 1)^{1/4} (d_j - 1)^{1/4}}\,.
\end{equation}

For the second term of \eqref{vAv split 2} in the case (ii) we use the estimates
\begin{equation*}
\nu_i(A \setminus S(A)) \;\asymp\; \nu_i(A)\,, \qquad 
\nu_j(A) \;\leq\; C \nu_j(A \setminus S(A)) \;\leq\; C (d_j - 1)^{-1/2} K^{-1/2 + \delta} \,, \qquad
\nu_i(A) \;\leq\; C \abs{d_i - d_j}\,.
\end{equation*}
Thus we get from case (c) that
\begin{equation*}
\absb{\scalar{\f v_i}{P_{A \setminus S(A)} \f v_j}} \;\leq\; \frac{1+ \phi^{1/2}}{M \nu_i(A)} \pbb{\frac{1}{\nu_j(A)} + \frac{1}{d_j - 1}} + \frac{(d_j - 1)^{1/2}}{M^{1/2} \abs{d_i - d_j}}\,,
\end{equation*}
where the last term is bounded by $K^\delta \frac{1 + \phi^{1/2}}{M \nu_i(A) \nu_j(A)}$. Recalling \eqref{vAvcaseii}, we find \eqref{goal_case_e} in the case (ii). The case (iii) is dealt with in the same way.

What remains therefore is case (iv). For the first term of \eqref{vAv split 2} we use the estimates
\begin{equation*}
\nu_i(A) \;\leq\; C \nu_i(S(A))\,, \qquad \nu_j(A) \;\leq\; C \nu_j(S(A))\,.
\end{equation*}
Thus we get from \eqref{vi_vj_result} that
\begin{equation*}
\scalar{\f v_i}{P_{S(A)} \f v_j} \;\leq\; \frac{1 + \phi^{1/2}}{M \nu_i(A) \nu_j(A)}\,.
\end{equation*}
For the second term of \eqref{vAv split 2} we use the estimates
\begin{equation*}
\nu_i(A) \;\leq\; C \nu_i(A \setminus S(A)) \;\leq\; C (d_i - 1)^{-1/2} K^{-1/2 + \delta}\,, \qquad
\nu_j(A) \;\leq\; C \nu_j(A \setminus S(A)) \;\leq\; C (d_j - 1)^{-1/2} K^{-1/2 + \delta}\,.
\end{equation*}
Therefore we get from case (c) that
\begin{equation*}
\scalar{\f v_i}{P_{A \setminus S(A)} \f v_i}
\;\prec\; \frac{d_i - 1}{1 + \phi^{1/2}} + \frac{1}{(d_i - 1)^{1/2} M^{1/2}} + \frac{1 + \phi^{1/2}}{M} \pbb{\frac{1}{\nu_i(A \setminus S(A))^2} + \frac{1}{(d_i - 1)^2}} \;\leq\; C K^{2 \delta} \frac{1 + \phi^{1/2}}{M \nu_i(A)^2}\,,
\end{equation*}
and a similar estimate holds for $\scalar{\f v_j}{P_{A \setminus S(A)} \f v_j}$. Thus we conclude that
\begin{equation*}
\absb{\scalar{\f v_i}{P_{A \setminus S(A)} \f v_j}} \;\leq\; \scalar{\f v_i}{P_{A \setminus S(A)} \f v_i}^{1/2} \, \scalar{\f v_j}{P_{A \setminus S(A)} \f v_j}^{1/2} \;\leq\; C K^{2 \delta} \frac{1 + \phi^{1/2}}{M \nu_i(A) \nu_j(A)}\,,
\end{equation*}
which is \eqref{goal_case_e}. This concludes the analysis of case (iv), and hence of case (d).

\subsubsection*{Conclusion of the proof}
Putting the cases (a)--(d) together, we have proved that \eqref{vi_vj_result 2} holds for arbitrary $i,j$ with an additional factor $K^{2 \delta}$ multiplying the error term on the right-hand side. Since $\delta > 0$ can be chosen arbitrarily small, \eqref{vi_vj_result 2} follows. This concludes the proof of Proposition \ref{prop: outl2}.
\end{proof}

\section{Non-outlier eigenvectors} \label{sec:bulk}

In this section we focus on the non-outlier eigenvectors $\f \xi_a$, $a \notin \cal O$, as well as outlier eigenvectors close to the bulk spectrum. We derive isotropic delocalization bounds for $\f \xi_a$ and establish the asymptotic law of the generalized components of $\f \xi_a$. We also use the former result to complete the proof of Theorem \ref{thm: outlier eigenvectors} on the outlier eigenvectors.

In Section \ref{sec: deloc_bounds} we derive isotropic delocalization bounds on $\f \xi_a$ for $\dist\{d_a, [-1,1]\} \leq 1 + K^{-1/3 + \tau}$. In Section \ref{sec72} we use these bounds to prove Theorem \ref{thm: outlier eigenvectors} and to complete the proof of Theorem \ref{thm: non-outlier bound} started in Section \ref{sec:outliers}. Next, in Section \ref{sec73} we derive the law of the generalized components of $\f \xi_a$ for $a \notin \cal O$. This argument requires two tools as input: level repulsion (Proposition \ref{prop: level repulsion}) and quantum unique ergodicity (see Section \ref{sec: uncorr}) of the eigenvectors $\f\zeta_b$ of $H$ (Proposition \ref{prop:xie}). Both are explained in detail and proved below.

\subsection{Bound on the spectral projections in the neighbourhood of the bulk spectrum} \label{sec: deloc_bounds}

We first consider eigenvectors near the right edge of the bulk spectrum. Recall the typical distance from $\mu_a$ to the spectral edges, denoted by $\kappa_a$ and defined in \eqref{def_kappa_a}.
\begin{proposition}[Eigenvectors near the right edge] \label{prop: right bulk}
Fix $\tau \in (0,1/3)$.
For $a \in \qq{s_+ + 1, (1 - \tau)K}$ we have
\begin{equation} \label{bulk_right_1}
\scalar{\f v_i}{\f \xi_a}^2 \;\prec\; \frac{1}{M} + \frac{\sigma_i}{M (\abs{d_i - 1}^2 + \kappa_a)}\,.
\end{equation}
Moreover, if $a \in \qq{1, s_+}$ satisfies $d_a \leq 1 + K^{-1/3 + \tau}$ then
\begin{equation} \label{bulk_right_2}
\scalar{\f v_i}{\f \xi_a}^2 \;\prec\; K^{3 \tau} \pbb{\frac{1}{M} + \frac{\sigma_i}{M (\abs{d_i - 1}^2 + \kappa_a)}}\,.
\end{equation}
\end{proposition}
Proposition \ref{prop: right bulk} has a close analogue for the left edge of the bulk spectrum, which holds under the additional condition $\abs{\phi - 1} \geq \tau$; we omit its detailed statement.

%Proposition \ref{prop: right bulk} has the following analogue for the left edge of the bulk spectrum.
%\begin{proposition}[Eigenvectors near the left edge] \label{prop: left bulk}
%Fix $\tau \in (0,1/3)$ and suppose that $\abs{\phi - 1} \geq \tau$.
%For $a \in \qq{\tau K, K - s_-}$ we have
%\begin{equation} \label{bulk_left_1}
%\scalar{\f v_i}{\f \xi_a}^2 \;\prec\; \frac{1}{M} + \frac{\sigma_i}{M (\abs{d_i + 1}^2 + \kappa_a)}\,.
%\end{equation}
%Moreover, if $a \in \qq{K - s_- + 1, K}$ satisfies $d_a \geq -1 - K^{-1/3 + \tau}$ then
%\begin{equation} \label{bulk_left_2}
%\scalar{\f v_i}{\f \xi_a}^2 \;\prec\; K^{3 \tau} \pbb{\frac{1}{M} + \frac{\sigma_i}{M (\abs{d_i + 1}^2 + \kappa_a)}}\,.
%\end{equation}
%\end{proposition}
%The proofs of Propositions \ref{prop: right bulk} and \ref{prop: left bulk} are analogous, and we only prove the former.

\begin{proof}[Proof of Proposition \ref{prop: right bulk}]
Suppose first that $i \in \cal R$. Let $\epsilon > 0$ and set $\omega \deq \epsilon / 2$. Using \eqref{bound: Qij iso}, Remark \ref{rem:all_z}, Theorem \ref{thm: outlier locations}, and  \eqref{rigidity1}, we choose a high-probability event $\Xi$ satisfying \eqref{Xi_bulk_edge}, \eqref{Xi_W}, and
\begin{equation} \label{Xi_outl_edge}
\ind{\Xi} \abs{\mu_i - \theta(d_i)} \;\leq\; (d_i - 1)^{1/2} K^{-1/2 + \epsilon} \qquad (1 + K^{-1/3} \leq d_i \leq 1 + K^{-1/3 + \tau})\,.
\end{equation}
For the following we fix a realization $H \in \Xi$.
We choose the spectral parameter $z = \mu_a + \eta$, where $\eta > 0$ is the smallest (in fact unique) solution of the equation $\im w_\phi(\mu_a + \ii \eta) = K^{-1 + 6 \epsilon} \eta^{-1}$. Hence \eqref{Xi_W} reads
\begin{equation} \label{W-w_bulk}
\norm{W(z) - w_\phi(z)} \;\leq\; \frac{K^{2 \epsilon}}{K \eta}\,.
\end{equation}
Abbreviating $\kappa \equiv \kappa(\mu_a)$, we find from \eqref{im m gamma} that
\begin{equation} \label{eta_case1}
\eta \;\asymp\; \frac{K^{6\epsilon}}{K \sqrt{\kappa} + K^{2/3 + 2\epsilon}} \qquad (\mu_a \leq \gamma_+ + K^{-2/3 + 4 \epsilon})
\end{equation}
and
\begin{equation} \label{eta_case2}
\eta \;\asymp\; K^{-1/2 + 3 \epsilon} \kappa^{1/4} \qquad (\mu_a \geq \gamma_+ + K^{-2/3 + 4 \epsilon})\,.
\end{equation}

Armed with these definitions, we may begin the estimate of $\scalar{\f v_i}{\f \xi_a}^2$.
The starting point is the bound
\begin{equation} \label{y_bound_G}
\scalar{\f v_i}{\f \xi_a}^2 \;\leq\; \eta \im \wt G_{\f v_i \f v_i}(z)\,,
\end{equation}
which follows easily by spectral decomposition. Since $i \in \cal R$, we get from \eqref{pert3}, omitting the arguments $z$ for brevity,
\begin{align}
\phi^{1/2} z \, \wt G_{\f v_i \f v_i} &\;=\;  \frac{1}{d_i} - \frac{\sigma_i}{d_i^2} \pbb{\frac{1}{D^{-1} + W}}_{ii}
\notag \\ \label{wtG_expanded}
&\;=\;\frac{1}{d_i} - \frac{\sigma_i}{d_i^2}\qBB{\frac{1}{d_i^{-1} + w_\phi} + \frac{1}{(d_i^{-1} + w_\phi)^2} \pbb{\p{w_\phi - W} + \p{w_\phi - W} \frac{1}{D^{-1} + W} \p{w_\phi - W}}_{ii}}\,,
\end{align}
where the last step follows from a resolvent expansion as in \eqref{resolvent expansion}. We estimate the error terms using
\begin{equation*}
\min_j \abs{d_j^{-1} + w_\phi} \;\geq\; \im w_\phi \;=\; \frac{K^{6 \epsilon}}{K \eta} \;\gg\; \frac{K^{2 \epsilon}}{K \eta} \;\geq\; \norm{W - w_\phi}\,,
\end{equation*}
where we used the definition of $\eta$ and \eqref{W-w_bulk}. Hence a resolvent expansion yields
\begin{equation*}
\normbb{\frac{1}{D^{-1} + W}} \;\leq\; \frac{2}{\im w_\phi} \;=\; 2 K^{1 - 6 \epsilon} \eta\,.
\end{equation*}
We therefore get from \eqref{wtG_expanded} that
\begin{equation} \label{wtG_expanded 2}
\phi^{1/2} z \, \wt G_{\f v_i \f v_i} \;=\; \frac{w_\phi - \phi^{1/2}}{1 + d_i w_\phi} + O \pbb{\frac{\sigma_i}{\abs{1 + d_i w_\phi}^2} \frac{K^{2 \epsilon}}{K \eta}}\,.
\end{equation}

Next, we claim that for any fixed $\delta \in [0, 1/3-\epsilon)$ we have the lower bound
\begin{equation} \label{1+dw_bound}
\abs{1 + d w_\phi} \;\geq\; c \pb{K^{-2 \delta} \abs{d - 1} + \im w_\phi}
\end{equation}
whenever $\mu_a \in [\theta(0), \theta(1 + K^{-1/3 + \delta + \epsilon})]$.
To prove \eqref{1+dw_bound}, suppose first that $\abs{d - 1} \geq 1/2$. By \eqref{re m gamma}, there exists a constant $c_0 > 0$ such that for $\kappa \leq c_0$ we have $\abs{\re w_\phi + 1} \leq 1/4$. Thus we get, for $\kappa \leq c_0$,
\begin{equation*}
\abs{1 + d w_\phi} \;\geq\; \abs{1 + d \re w_\phi} \;\asymp\; \abs{d-1} + \im w_\phi\,,
\end{equation*}
where we used that $\im w_\phi \leq C$ by \eqref{bounds on mg}. Moreover, if $\kappa \geq c_0$ we find from \eqref{im m gamma} that $\im w_\phi \geq c$, from which we get
\begin{equation*}
\abs{1 + d w_\phi} \;\geq\; \abs{1 + d \re w_\phi} + \abs{d} \im w_\phi \;\geq\; c (1 + \abs{d}) \;\asymp\; \abs{d-1} + \im w_\phi\,,
\end{equation*}
where in the second step we used $\abs{\re w_\phi} \leq C$ as follows from \eqref{bounds on mg}. This concludes the proof of \eqref{1+dw_bound} for the case $\abs{d - 1} \geq 1/2$.

Suppose now that $\abs{d - 1} \leq 1/2$. Then we get
\begin{equation*}
\abs{1 + d w_\phi} \;\asymp\; \abs{1 + d \re w_\phi} + \abs{d} \im w_\phi \;\geq\; \pb{\abs{d - 1} - \abs{\re w_\phi + 1}}_+ + \im w_\phi\,.
\end{equation*}
We shall estimate this using the elementary bound
\begin{equation} \label{x-y-z}
(x-y)_+ + z \;\geq\; \frac{x}{3M} + \frac{z}{3} \qquad \text{if } y \leq M z \text{ for some } M \geq 1\,.
\end{equation}
For $\mu_a \in [\theta(0), \theta(1)]$ we get from \eqref{x-y-z} with $M = C$, recalling \eqref{im m gamma} and \eqref{re m gamma}, that $\abs{1 + d w_\phi} \geq c(\abs{d - 1} + \im w_\phi)$. By a similar argument, for $\mu_a \in [\theta(1), \theta(1 + K^{-1/3 + \delta + \epsilon})]$ we set $M = K^{2 \delta}$ and get \eqref{1+dw_bound} using \eqref{eta_case1} and \eqref{eta_case2}. This concludes the proof of \eqref{1+dw_bound}.

Going back to \eqref{y_bound_G}, we find using \eqref{wtG_expanded 2}
\begin{align}
\scalar{\f v_i}{\f \xi_a}^2 &\;\leq\; \eta \phi^{-1/2} \im \pbb{\frac{w_\phi - \phi^{1/2}}{z (1 + d_i w_\phi)}} + \frac{C \sigma_i}{\abs{1 + d_i w_\phi}^2} \frac{K^{2 \epsilon}}{K \eta}
\notag \\ \label{ya_estimate1}
&\;=\; \frac{\eta^2}{\phi^{1/2} \abs{z}^2} \re \frac{\phi^{1/2} - w_\phi}{1 + d_i w_\phi}
+ \frac{\eta \mu_a}{\phi^{1/2} \abs{z}^2} \im \frac{w_\phi - \phi^{1/2}}{1 + d_i w_\phi}
+ \frac{C \sigma_i}{\phi^{1/2} \abs{z} \abs{1 + d_i w_\phi}^2} \frac{K^{2 \epsilon}}{K}\,.
\end{align}
Using $\abs{z} \asymp \mu_a \asymp \phi^{-1/2} + \phi^{1/2}$ and \eqref{1+dw_bound}, we estimate the first term on the right-hand side of \eqref{ya_estimate1} as
\begin{equation*}
\frac{\eta^2}{\phi^{1/2} \abs{z}^2} \re \frac{\phi^{1/2} - w_\phi}{1 + d_i w_\phi} \;\leq\; \frac{\eta^2}{(1 + \phi) \abs{1 + d_i w_\phi}} \;\leq\; 
\frac{\eta^2}{(1 + \phi) \im w_\phi} \;\leq\; C \frac{\eta^3 K}{1 + \phi} \;\leq\; C \frac{K^{12 \epsilon + 3 \delta}}{M}\,,
\end{equation*}
where in the last step we used that $\eta \leq K^{-2/3 + 4 \epsilon + \delta}$, as follows from \eqref{eta_case1} and \eqref{eta_case2}.

Next, we estimate the second term of \eqref{ya_estimate1} as
\begin{equation*}
\frac{\eta \mu_a}{\phi^{1/2} \abs{z}^2} \im \frac{w_\phi - \phi^{1/2}}{1 + d_i w_\phi} \;=\; \frac{\eta \mu_a}{\phi^{1/2} \abs{z}^2} \frac{\sigma_i \im w_\phi}{\abs{1 + d_i w_\phi}^2} \;\asymp\; \frac{\sigma_i \eta \im w_\phi}{(1 + \phi) \abs{1 + d_i w_\phi}^2} \;\leq\; \frac{C \sigma_i K^{6 \epsilon}}{M \abs{1 + d_i w_\phi}^2}\,.
\end{equation*}
We estimate the last term of \eqref{ya_estimate1} as
\begin{equation*}
\frac{C \sigma_i}{\phi^{1/2} \abs{z} \abs{1 + d_i w_\phi}^2} \frac{K^{2 \epsilon}}{K} \;\leq\; \frac{C \sigma_i K^{2 \epsilon}}{M \abs{1 + d_i w_\phi}^2}\,.
\end{equation*}
Putting all three estimates together, we conclude that
\begin{equation} \label{ya_estimate2}
\scalar{\f v_i}{\f \xi_a}^2 \;\leq\; \frac{K^{12 \epsilon + 3 \delta}}{M} + \frac{C \sigma_i K^{6 \epsilon}}{M \abs{1 + d_i w_\phi}^2}\,.
\end{equation}

In order to estimate the denominator of \eqref{ya_estimate2} from below using \eqref{1+dw_bound}, we need a suitable lower bound on $\im w_\phi(\mu_a + \ii \eta)$. First, if $a \geq s_+ + 1$ then we get from \eqref{Xi_bulk_edge}, Corollary \ref{cor: interlacing}, \eqref{eta_case1}, and \eqref{im m gamma} that
\begin{equation*}
\im w_\phi(\mu_a + \ii \eta) \;\geq\; c \sqrt{\kappa_a}\,,
\end{equation*}
in which case we get by choosing $\delta = 0$ in \eqref{1+dw_bound} that
\begin{equation} \label{bulk_right_1'}
\scalar{\f v_i}{\f \xi_a}^2 \;\leq\; \frac{K^{12 \epsilon}}{M} + \frac{C \sigma_i K^{6 \epsilon}}{M (\abs{d_i - 1}^2 + \kappa_a)}\,.
\end{equation}

Next, if $a \leq s_+$ satisfies $d_a \leq K^{-1/3 + \tau}$ we get from \eqref{eta_case1}, \eqref{eta_case2}, and \eqref{im m gamma} that
\begin{equation*}
\im w_\phi(\mu_a + \ii \eta) \;\geq\; c \sqrt{\eta} \;\geq\; c K^{-1/3 + 2 \epsilon}\,.
\end{equation*}
In this case we have $\mu_a \leq \theta(1 + K^{-1/3 + \tau + \epsilon})$ by \eqref{Xi_outl_edge}, so that setting $\delta = \tau$ in \eqref{1+dw_bound} yields
\begin{equation} \label{bulk_right_2'}
\scalar{\f v_i}{\f \xi_a}^2 \;\leq\; \frac{K^{12 \epsilon + 3 \tau}}{M} + \frac{C \sigma_i K^{6 \epsilon + 2 \tau}}{M (\abs{d_i - 1}^2 + \kappa_a)}\,.
\end{equation}
Since $\epsilon > 0$ was arbitrary, \eqref{bulk_right_1} and \eqref{bulk_right_2} follow from \eqref{bulk_right_1'} and \eqref{bulk_right_2'} respectively. 
This concludes the proof of Proposition \ref{prop: right bulk} in the case $i \in \cal R$.

Finally, the case $i \notin \cal R$ is handled by replacing $\cal R$ with $\wh {\cal R} \deq \cal R \cup \{i\}$ and using a limiting argument, exactly as after \eqref{vPv for R}.
\end{proof}

\subsection{Proof of Theorems \ref{thm: outlier eigenvectors} and \ref{thm: non-outlier bound}} \label{sec72}

We now have all the ingredients needed to prove Theorems \ref{thm: outlier eigenvectors} and \ref{thm: non-outlier bound}.

\begin{proof}[Proof of Theorem \ref{thm: non-outlier bound}]
The estimate \eqref{bulk_right} is an immediate corollary of \eqref{bulk_right_1} from Proposition \ref{prop: right bulk}. The estimate \eqref{bulk_left} is proved similarly (see also the remark following Proposition \ref{prop: right bulk}).
\end{proof}

\begin{proof}[Proof of Theorem \ref{thm: outlier eigenvectors}]
We prove Theorem \ref{thm: outlier eigenvectors} using Propositions \ref{prop: outl2}, \ref{prop: outl1}, and \ref{prop: right bulk}. First we remark that it suffices to prove that \eqref{vi_vj_result 2} holds for $A \subset \cal O$ satisfying $1 + K^{-1/3} \leq d_k \leq \tau^{-1}$ for all $k \in A$. Indeed, supposing this is done, we get the estimate
\begin{multline*}
\scalar{\f w}{P_A \f w} \;=\; \scalar{\f w}{Z_A \f w}
+ O_\prec \Biggl[
\sum_{i \in A} \frac{w_i^2}{M^{1/2}(d_i - 1)^{1/2}}
+ \sum_{i \in A} \frac{\sigma_i w_i^2}{M (d_i - 1)^2}
\\
+ \sum_{i = 1}^M \frac{\sigma_i w_i^2}{M \nu_i^2}
+
\scalar{\f w}{Z_A \f w}^{1/2} \pBB{\sum_{i \notin A} \frac{\sigma_i w_i^2}{M \nu_i^2}}^{1/2} \Biggr]\,,
\end{multline*}
from which Theorem \ref{thm: outlier eigenvectors} follows by noting that the second error term may be absorbed into the first, recalling that $\sigma_i \asymp 1 + \phi^{1/2}$ for $i \in A$, that $M \asymp (1 + \phi) K$, and that $d_i - 1 \geq K^{-1/3}$.

Fix $\epsilon > 0$. Note that there exists some $s \in [1, \abs{\cal R}]$ satisfying the following gap condition: for all $k$ such that $d_k > 1 + s K^{-1/3 + \epsilon}$ we have $d_k \geq 1 + (s+1) K^{-1/3 + \epsilon}$. The idea of the proof is to split $A = A_0 \sqcup A_1$, such that $d_k \leq 1 + s K^{-1/3 + \epsilon}$ for $k \in A_0$ and $d_k \geq 1 + (s + 1) K^{-1/3 + \epsilon}$ for $k \in A_1$. Note that such a splitting exists by the above gap property. Without loss of generality, we assume that $A_0 \neq \emptyset$ (for otherwise the claim follows from Proposition \ref{prop: outl2}).

It suffices to consider the six cases (a) $i,j \in A_0$, (b) $i \in A_0$ and $j \in A_1$, (c) $i \in A_0$ and $j \notin A$, (d) $i,j \in A_1$, (e) $i \in A_1$ and $j \notin A$, (f) $i,j \notin A$.

\subsubsection*{(a) $i,j \in A_0$}
We split
\begin{equation} \label{split near bulk}
\scalar{\f v_i}{P_A \f v_j} \;=\; \scalar{\f v_i}{P_{A_0} \f v_j} + \scalar{\f v_i}{P_{A_1} \f v_j}\,.
\end{equation}
We apply Cauchy-Schwarz and Proposition \ref{prop: right bulk} to the first term, and Proposition \ref{prop: outl2} to the second term. Using the above gap condition, we find
\begin{align*}
\absb{\scalar{\f v_i}{P_A \f v_j}} &\;\prec\; \frac{K^{3 \epsilon} \sqrt{\sigma_i \sigma_j}}{M \abs{d_i - 1} \abs{d_j - 1}} + \frac{\sqrt{\sigma_i \sigma_j}}{M} \pbb{\frac{1}{\nu_i} + \frac{1}{d_i - 1}} \pbb{\frac{1}{\nu_j} + \frac{1}{d_j - 1}}
\\
&\;=\; \delta_{ij}u(d_i)
+ K^{3 \epsilon} \, O \qBB{\frac{1}{(d_i - 1)^{1/4} (d_j - 1)^{1/4} M^{1/2}} + \frac{\sqrt{\sigma_i \sigma_j}}{M} \pbb{\frac{1}{\nu_i} + \frac{1}{d_i - 1}} \pbb{\frac{1}{\nu_j} + \frac{1}{d_j - 1}}}\,,
\end{align*}
where the last step follows from $d_i - 1 \leq C K^{-1/3 + \epsilon}$.

\subsubsection*{(b) $i \in A_0$ and $j \in A_1$}
For this case it is crucial to use the stronger bound \eqref{vi_vj_result} and not \eqref{vi_vj_result 2}. Hence, we need the non-overlapping condition \eqref{non-overlapping}. To that end, we assume first that \eqref{non-overlapping} holds with $\delta \deq \epsilon$. Thus, by the above gap assumption \eqref{non-overlapping} also holds for $A_1$. In this case we get from \eqref{split near bulk} and Propositions \ref{prop: outl1} and \ref{prop: right bulk} that
\begin{equation*}
\absb{\scalar{\f v_i}{P_A \f v_j}} \;\prec\; \frac{K^{3 \epsilon} \sqrt{\sigma_i \sigma_j}}{M \abs{d_i - 1} \abs{d_j - 1}} + \frac{\sqrt{\sigma_i \sigma_j}}{M} \pbb{\frac{1}{\nu_i} + \frac{1}{d_i - 1}} \pbb{\frac{1}{\nu_j} + \frac{1}{d_j - 1}} + \frac{(d_j - 1)^{1/2} \sqrt{\sigma_i}}{(1 + \phi^{1/4}) \abs{d_i - d_j} M^{1/2}}\,.
\end{equation*}
Clearly, the first two terms are bounded by the right-hand side of \eqref{vi_vj_result 2} times $K^{3 \epsilon}$. The last term is estimated as
\begin{equation*}
\frac{(d_j - 1)^{1/2} \sqrt{\sigma_i}}{(1 + \phi^{1/4}) \abs{d_i - d_j} M^{1/2}} \;\asymp\;
\frac{(d_j - 1)^{1/2}}{\abs{d_i - d_j} M^{1/2}} \;\leq\; \frac{1}{(d_i - 1)^{1/4} (d_j - 1)^{1/4} M^{1/2}}\,,
\end{equation*}
where we used that $d_i - 1 \leq d_j - 1 \leq C \abs{d_i - d_j}$ be the above gap condition. This concludes the proof in the case where the non-overlapping condition \eqref{non-overlapping} holds.

If \eqref{non-overlapping} does not hold, we replace $A_1$ with the smaller set $S(A_1)$ defined in Section \ref{sec: overlapping}. Then we proceed as above, except that we have to deal in addition with the term $\scalar{\f v_i}{P_{A_1 \setminus S(A_1)}}$. The details are analogous to those of Section \ref{sec: overlapping}, and we omit them here.

\subsubsection*{(c), (e), (f) $j \notin A$}
We use the splitting \eqref{split near bulk} and apply Cauchy-Schwarz and Proposition \ref{prop: right bulk} to the first term, and Proposition \ref{prop: outl2} to the second term. Since $\nu_j(A_1) \leq \abs{d_j - 1}$ in all cases, it is easy to prove that \eqref{split near bulk} is bounded by $K^{3 \epsilon}$ times the right-hand side of \eqref{vi_vj_result 2}.

\subsubsection*{(d) $i,j \in A_1$} From \eqref{split near bulk} and Propositions \ref{prop: right bulk} and \ref{prop: outl2} we get
\begin{multline*}
\scalar{\f v_i}{P_A \f v_j} \;=\; \delta_{ij}u(d_i)
\\
+ O_\prec \qBB{\frac{K^{3 \epsilon} \sqrt{\sigma_i \sigma_j}}{M \abs{d_i - 1} \abs{d_j - 1}} + \frac{1}{(d_i - 1)^{1/4} (d_j - 1)^{1/4} M^{1/2}} + \frac{\sqrt{\sigma_i \sigma_j}}{M} \pbb{\frac{1}{\nu_i} + \frac{1}{d_i - 1}} \pbb{\frac{1}{\nu_j} + \frac{1}{d_j - 1}}}\,,
\end{multline*}
From which we get \eqref{vi_vj_result 2} with the error term multiplied by $K^{3 \epsilon}$.

\subsubsection*{Conclusion of the proof}
We have proved that, for all $i,j \in \qq{1,M}$  and $A$ satisfying the assumptions of Theorem \ref{thm: outlier eigenvectors}, the estimate \eqref{vi_vj_result 2} holds with an additional factor $K^{3 \epsilon}$ multiplying the error term. Since $\epsilon$ was arbitrary, we get \eqref{vi_vj_result 2}. This concludes the proof.
\end{proof}

\subsection{The law of the non-outlier eigenvectors}\label{sec73}
For $a \leq K/2$ define
\begin{equation} \label{def_Delta}
\Delta_a \;\deq\; K^{-2/3} a^{-1/3}\,,
\end{equation}
the typical distance between $\lambda_{a+1}$ and $\lambda_a$. More precisely, the classical locations $\gamma_a$ defined in \eqref{def:gamma_alpha} satisfy $\gamma_a - \gamma_{a+1} \asymp \Delta_a$ for $a \leq K/2$.

We may now state the main result behind the proof of Theorem \ref{thm: eigenvector law}. Recall the definitions \eqref{def_alpha} of $\alpha_+$ and \eqref{def_r_pm} of $s_+$, the number of outliers to the right of the bulk spectrum. Recall also from \eqref{def_lambda} and \eqref{def_zeta} that $\{\lambda_a\}$ and $\{\f \zeta_a\}$ denote the eigenvalues and eigenvectors of $H = X X^*$.
\begin{proposition} \label{prop: bulk_law}
Let $s_+ + 1 \leq a \leq K^{1 - \tau} \alpha_+^3$ and define $b \deq a - s_+$.
Define the event
\begin{equation*}
\Omega \;\equiv\; \Omega_{a, b, \tau} \;\deq\; \hB{\abs{\mu_{b' + s_+} - \lambda_{b'}} \leq K^{-\tau/4} \Delta_a \txt{ for } \abs{b' - b} \leq 1} \cap \hB{\abs{\lambda_{b'} - \lambda_b} \geq K^{-\tau/6} \Delta_a \txt{ for } \abs{b' - b} = 1}\,.
\end{equation*}
Then
\begin{equation} \label{ind_Omega_component}
\ind{\Omega} \scalar{\f w}{\f \xi_a}^2 \;=\; \ind{\Omega} \absbb{\scalarbb{\sum_i \frac{\sqrt{\sigma_i} \, w_i}{d_i - 1} \, \f v_i}{\f \zeta_b}}^2 + O_\prec \qBB{\frac{K^{-1/3 + \tau/5} a^{1/3}}{\alpha_+} \sum_i \frac{\sigma_i w_i^2}{M (d_i - 1)^2}}\,. 
\end{equation}
\end{proposition}

Informally, Proposition \ref{prop: bulk_law} expresses generalized components of the eigenvectors of $Q$ in terms of generalized components of eigenvectors of $H$, under the assumption that $\Omega$ has high probability. We first show how Proposition \ref{prop: bulk_law} implies Theorem \ref{thm: eigenvector law}. This argument requires two key tools. The first one is level repulsion, which, together with the eigenvalue sticking from Theorem \ref{thm:sticking}, will imply that $\Omega$ indeed has high probability. The second tool is quantum unique ergodicity (See Section \ref{sec: uncorr}) of the eigenvectors of $H$, which establishes the law of the generalized components of the eigenvectors of $H$.

The precise statement of level repulsion sufficient for our needs is as follows.

\begin{proposition}[Level repulsion] \label{prop: level repulsion}
Fix $\tau \in (0, 1)$. For any $\epsilon > 0$ there exists a $\delta > 0$ such that for all $a \leq K^{1 - \tau}$ we have
\begin{equation} \label{level repulsion}
\P \pb{\abs{\lambda_a - \lambda_{a+1}} \leq \Delta_a K^{-\epsilon}} \;\leq\; K^{-\delta}\,.
\end{equation}
\end{proposition}
The proof of Proposition \ref{prop: level repulsion} consists of two steps: (i) establishing \eqref{level repulsion} for the case of Gaussian $X$ and (ii) a comparison argument showing that if $X^{(1)}$ and $X^{(2)}$ are two matrix ensembles satisfying \eqref{cond on entries of X} and \eqref{moments of X-1}, and if \eqref{level repulsion} holds for $X^{(1)}$, then \eqref{level repulsion} also holds for $X^{(2)}$. Both steps have already appeared, in a somewhat different form, in the literature. Step (i) is performed in Lemma \ref{lemma:lr1} below, and step (ii) in Lemma \ref{lemma:lr2} below. Together, Lemmas \ref{lemma:lr1} and \ref{lemma:lr2} immediately yield Proposition \ref{prop: level repulsion}.

\begin{lemma}[Level repulsion for the Gaussian case] \label{lemma:lr1}
Proposition \ref{prop: level repulsion} holds if $X$ is Gaussian.
\end{lemma}
\begin{proof}
We mimic the proof of Theorem 3.2 in \cite{BEY3}. Indeed, the proof from \cite[Appendix D]{BEY3} carries over almost verbatim.  The key input is the eigenvalue rigidity from Theorem \ref{thm: cov-rig}, which for the model of \cite{BEY3} was established using a different method than Theorem \ref{thm: cov-rig}. As in \cite{BEY3}, we condition on the eigenvalues $\{\lambda_i \col i > K^{1 - \tau}\}$. On the conditioned measure, level repulsion follows as in \cite{BEY3}. Finally, thanks to Theorem \ref{thm: cov-rig} we know that the frozen eigenvalues $\{\lambda_i \col i > K^{1 - \tau}\}$ are with high probability near their classical locations. Note that for $\phi \approx 1$, the rigidity estimate \eqref{rigidity1} only holds for indices $i \leq (1 - \tau) K$; however, this is enough for the argument of \cite[Appendix D]{BEY3}, which is insensitive to the locations of eigenvalues at a distance of order one from the right edge $\gamma_+$. We omit the full details.
\end{proof}

\begin{lemma}[Stability of level repulsion] \label{lemma:lr2}
Let $X^{(1)}$ and $X^{(2)}$ be two matrix ensembles satisfying \eqref{cond on entries of X} and \eqref{moments of X-1}. Suppose that Proposition \ref{prop: level repulsion} holds for $X^{(1)}$. Then Proposition \ref{prop: level repulsion} also holds for $X^{(2)}$.
\end{lemma}
The proof of Lemma \ref{lemma:lr2} relies on Green function comparison, and is given in Section \ref{sec:pf_lrcomparison}.

The second tool behind the proof of Theorem \ref{thm: eigenvector law} is the quantum unique ergodicity of the eigenvectors $\f \zeta_a$ of the matrix $H = X X^*$, stated in Proposition \ref{prop:xie} below. As noted in Section \ref{sec: uncorr}, quantum unique ergodicity is a term borrowed from quantum chaos that describes the complete ``flatness'' of the eigenvectors of $H$. Here ``flatness'' means that the eigenvectors are asymptotically uniformly distributed on the unit sphere of $\R^M$.
The first result on quantum unique ergodicity of Wigner matrices is \cite{KY1}, where the quantum unique ergodicity of eigenvectors near the spectral edge was established. Under an additional four-moment matching condition, this result was extended to the bulk. Subsequently, this second result was derived using a different method in \cite{TV3}. Recently, a new approach to the quantum unique ergodicity was developed in \cite{BoY1}, where quantum unique ergodicity is established for all eigenvectors of generalized Wigner matrices. In this paper, we adopt the approach of \cite{KY1}, based on Green function comparison. As compared to the method of \cite{BoY1}, its first advantage is that it is completely local in the spectrum, and in particular when applied near the right-hand edge of the spectrum it is insensitive to the presence of a hard edge at the origin. The second advantage of the current method is that it is very robust and may be used to establish the asymptotic joint distribution of an arbitrary family of generalized components of eigenvectors, as in Remark \ref{rem: gen QUE} below; we remark that such joint laws cannot currently be analysed using the method of \cite{BoY1}. On the other hand, our results only hold for eigenvector indices $a$ satisfying $a \leq K^{1 - \tau}$ for some $\tau > 0$, while those of \cite{BoY1} admit $\tau = 0$.

Our proof of quantum unique ergodicity generalizes that of \cite{KY1} in three directions. First, we extend the method of \cite{KY1} to sample covariance matrices (in fact to general sample covariance matrices of the form \eqref{wtH_cov} with $\Sigma = T T^* = I_M$; see Section \ref{sec:general_model}). Second, we consider generalized components $\scalar{\f w}{\zeta_a}$ of the eigenvectors instead of the cartesian components $\zeta_a(i)$. The third and deepest generalization is that we establish quantum unique ergodicity much further into the bulk, requiring only that $a \leq K^{1 - \tau}$ instead of the assumption $a \leq (\log K)^{C \log \log K}$ from \cite{KY1}.

\begin{proposition}[Quantum unique ergodicity]\label{prop:xie}
Fix $\tau \in (0,1)$. Then for any $a \le K^{1 - \tau}$ and deterministic unit vector $\f w \in \R^M$ we have
\begin{equation} \label{fluct of vect comp}
M \scalar{\f w}{\f \zeta_a}^2 \; \longrightarrow \; \chi_1^2\,,
\end{equation}
in the sense of moments, uniformly in $a$ and $\f w$.
\end{proposition}

\begin{remark} \label{rem: gen QUE}
For simplicity, and bearing the application to Theorem \ref{thm: eigenvector law} in mind, in Proposition \ref{prop:xie} we establish the convergence of a single generalized component of a single eigenvector. However, our method may be easily extended to yield
\begin{equation*}
\pB{M \scalar{\f v_1}{\f \zeta_{a_1}} \scalar{\f \zeta_{a_1}}{\f w_1}, \dots, M \scalar{\f v_k}{\f \zeta_{a_k}} \scalar{\f \zeta_{a_k}}{\f w_k}} \;\simdist\; (Z_1, \dots, Z_k)
\end{equation*}
for any deterministic unit vectors $\f v_1, \dots, \f v_l, \f w_1, \dots, \f w_k \in \R^M$ and $a_1 < \dots < a_k \leq K^{1 - \tau}$, whereby we use the notation $A_N \simdist B_N$ to mean that $A_N$ and $B_N$ are tight, and $\lim_{N \to \infty}\E (f(A_N) - f(B_N)) = 0$ for all polynomially bounded and continuous $f$.
Here $(Z_1, \dots, Z_k)$ is a family of independent random variables defined by $Z_i = A_i B_i$, where $A_i$ and $B_i$ are jointly Gaussian with covariance matrix
\begin{equation*}
\begin{pmatrix}
1 & \scalar{\f v_i}{\f w_i}
\\
\scalar{\f v_i}{\f w_i} & 1
\end{pmatrix}
\,.
\end{equation*}
The proof of this generalization of Proposition \ref{prop:xie} follows that of Proposition \ref{prop:xie} presented in Section \ref{sec: QUE}, requiring only heavier notation. In fact, our method may also be used to prove the universality of the joint eigenvalue-eigenvector distribution for any matrix $Q$ of the form \eqref{wtH_cov} with $\Sigma = T T^* = I_M$; see Theorem \ref{thm:univ_H} below for a precise statement.
\end{remark}

The proof of Proposition \ref{prop:xie} is postponed to Section \ref{sec: QUE}.

Supposing Proposition \ref{prop: bulk_law} holds, together with Propositions \ref{prop: level repulsion} and \ref{prop:xie}, we may complete the proof of Theorem \ref{thm: eigenvector law}.

\begin{proof}[Proof of Theorem \ref{thm: eigenvector law}]
Abbreviating $b \deq a - s_+$ and
\begin{equation*}
\f u \;\deq\; \frac{1}{\sqrt{M}} \sum_i \frac{\sqrt{\sigma_i}\, w_i}{d_i - 1} \, \f v_i\,,
\end{equation*}
we define
\begin{equation*}
\wh \Theta(a, \f w) \;\deq\; M \frac{\scalar{\f u}{\f \zeta_b}^2}{\abs{\f u}^2}\,.
\end{equation*}
Then, by assumption on $a$, we may rewrite \eqref{ind_Omega_component} as
\begin{equation*}
\ind{\Omega} \scalar{\f w}{\f \xi_a}^2 \;=\; \ind{\Omega} \abs{\f u}^2 \, \wh \Theta(a, \f w) + O_\prec (K^{-2\tau/15} \abs{\f u}^2)\,. 
\end{equation*}
Moreover, by Theorem \ref{thm:sticking} and Proposition \ref{prop: level repulsion}, we have $\P(\Omega) \geq 1 - K^{-c}$ for some constant $c > 0$. Finally, by Proposition \ref{prop:xie} we have $\wh \Theta(a, \f w) \to \chi_1^2$ in distribution (even in the sense of moments). The claim now follows easily.
\end{proof}

The remainder of this section is devoted to the proof of Proposition \ref{prop: bulk_law}.

\begin{proof}[Proof of Proposition \ref{prop: bulk_law}]
We define the contour $\Gamma_a$ as the positively oriented circle of radius $K^{-\tau/5} \Delta_a$ with centre $\lambda_b$.
Let $\epsilon > 0$ and $\tau \deq 1/2$, and choose a high-probability event $\Xi$ such that \eqref{Xi_bulk_edge}, \eqref{Xi_W}, and \eqref{Xi_G} hold. For the following we fix a realization $H \in \Omega \cap \Xi$.
Define
\begin{equation*}
Y_a(\f w) \;\deq\; - \frac{1}{2 \pi \ii} \oint_{\Gamma_a} \phi^{1/2} z \, \wt G_{\f w \f w}(z) \, \dd z\,.
\end{equation*}
By the residue theorem and the definition of $\Omega$, we find
\begin{equation} \label{Y-rep}
Y_a(\f w) \;=\; \phi^{1/2} \mu_a \scalar{\f w}{\f \xi_a}^2\,.
\end{equation}
To simplify notation, suppose now that $i \in \cal R$ and consider $\f w = \f v_i$. From \eqref{pert3} we find that
\begin{equation} \label{Y_a_contour}
Y_a(\f v_i) \;=\; \frac{\sigma_i}{d_i^2} \, \frac{1}{2 \pi \ii} \oint_{\Gamma_a} \pbb{\frac{1}{D^{-1} + W(z)}}_{ii} \, \dd z\,.
\end{equation}

In order to compute \eqref{Y_a_contour}, we need precise estimates for $W$ on $\Gamma_a$. Because the contour $\Gamma_a$ crosses the branch cut of $w_\phi$, we should not compare $W(z)$ to $w_\phi(z)$ for $z \in \Gamma_a$. Instead, we compare $W(z)$ to $w_\phi(z_0)$, where
\begin{equation*}
z_0 \;\deq\; \lambda_b + \ii \eta \,, \qquad \eta \;\deq\; K^{-\tau/5} \Delta_a\,.
\end{equation*}
We claim that
\begin{equation} \label{W_w_estimate}
\norm{W(z) - w_\phi(z_0)} \;\leq\; C K^{-1 + \epsilon} \eta^{-1}\,.
\end{equation}
for all $z \in \Gamma_a$. To see this, we split
\begin{equation} \label{W_w_split}
\norm{W(z) - w_\phi(z_0)} \;\leq\; \norm{W(z) - W(z_0)} + \norm{W(z_0) - w_\phi(z_0)}\,.
\end{equation}
We estimate the first term of \eqref{W_w_split} by spectral decomposition, using that $\dist(z, \sigma(H)) \geq c \eta$, similarly to \eqref{W_add_eta}. The result is
\begin{align*}
\norm{W(z) - W(z_0)} &\;=\; C (1 + \phi) \max_{i} \im G_{\f v_i \f v_i}(z_0)
\\
&\;\leq\; C (1 + \phi) \pbb{\im m_{\phi^{-1}}(z_0) + \frac{K^\epsilon}{1 + \phi} \frac{1}{K \eta}}
\\
&\;\leq\; C \im w_\phi(z_0) + C K^{-1 + \epsilon} \eta^{-1}
\\
&\;\leq\; C K^{-1 + \epsilon} \eta^{-1}\,,
\end{align*}
where we used \eqref{Xi_bulk_edge}, \eqref{Xi_G}, and Lemma \ref{lemma: w}. Moreover, we estimate the second term of \eqref{W_w_split} using \eqref{Xi_W} as
\begin{equation*}
\norm{W(z_0) - w_\phi(z_0)} \;\leq\; K^{-1 + \epsilon} \eta^{-1}\,.
\end{equation*}
This concludes the proof of \eqref{W_w_estimate}.

Next, we claim that
\begin{equation} \label{di_z0_est}
\abs{1 + d_i w_\phi(z_0)} \;\geq\; c \abs{d_i - 1}\,.
\end{equation}
The proof of \eqref{di_z0_est} is analogous to that of \eqref{1+dw_bound}, using \eqref{Xi_bulk_edge} and the assumption on $a$; we omit the details.

Armed with \eqref{W_w_estimate} and \eqref{di_z0_est}, we may analyse \eqref{Y_a_contour}. A resolvent expansion in the matrix $w_\phi(z_0) - W(z)$ yields
\begin{multline} \label{Y_res_exp}
Y_a(\f v_i) 
\\
=\; \frac{\sigma_i}{d_i^2} \, \frac{1}{2 \pi \ii} \oint_{\Gamma_a} \pBB{\frac{1}{d_i^{-1} + w_\phi(z_0)} + \frac{w_\phi(z_0) - W_{ii}(z)}{(d_i^{-1} + w_\phi(z_0))^2}  + \pbb{\frac{w_\phi(z_0) - W(z)}{d_i^{-1} + w_\phi(z_0)} \frac{1}{D^{-1} + W(z)} \frac{w_\phi(z_0) - W(z)}{d_i^{-1} + w_\phi(z_0)}}_{ii}} \, \dd z\,.
\end{multline}
We estimate the third term using the bound
\begin{equation} \label{DW_inv_bound 2}
\normbb{\frac{1}{D^{-1} + W(z)}} \;\leq\; \frac{C}{\alpha_+}\,.
\end{equation}
To prove \eqref{DW_inv_bound 2}, we note first that by \eqref{di_z0_est} we have
\begin{equation*}
\min_i \absb{d_i^{-1} + w_\phi(z_0)} \;\geq\; \min_i \frac{\abs{1 + d_i w_\phi(z_0)}}{\abs{d_i}} \;\geq\; c \, \min_i \frac{\abs{d_i - 1}}{\abs{d_i}} \;\geq\; c \, \alpha_+\,.
\end{equation*}
By \eqref{W_w_estimate} and assumption on $a$, it is easy to check that
\begin{equation*}
\min_i \absb{d_i^{-1} + w_\phi(z_0)} \;\geq\; K^{\tau/5} \norm{W(z) - w_\phi(z_0)}\,,
\end{equation*}
from which \eqref{DW_inv_bound 2} follows.

We may now return to \eqref{Y_res_exp}. The first term vanishes, the second is computed by spectral decomposition of $W$, and the third is estimated using \eqref{DW_inv_bound 2}. This gives
\begin{equation*}
Y_a(\f v_i) \;=\; \frac{\phi^{1/2} \sigma_i \lambda_b}{(1 + d_i w_\phi(z_0))^2} \scalar{\f v_i}{\f \zeta_b}^2 + O \pbb{\frac{\sigma_i}{\abs{d_i - 1}^2 K} \, \frac{K^{2 \epsilon}}{K \eta \alpha_+}}\,,
\end{equation*}
where we also used \eqref{di_z0_est}.

Recalling \eqref{Y-rep} and \eqref{Xi_bulk_edge}, we therefore get
\begin{equation*}
\scalar{\f v_i}{\f \xi_a}^2 \;=\; \frac{\sigma_i \lambda_b / \mu_a}{(1 + d_i w_\phi(z_0))^2} \scalar{\f v_i}{\f \zeta_b}^2 + O \pbb{\frac{\sigma_i}{\abs{d_i - 1}^2 M} \, \frac{K^{-1/3 + 2 \epsilon + \tau/5} a^{1/3}}{\alpha_+}}\,,
\end{equation*}
where we used $\phi^{1/2}\mu_a \asymp 1 + \phi$.

In order to simplify the leading term, we use
\begin{equation*}
\frac{-1}{1 + d_i w_\phi(z_0)} \;=\; \frac{1}{d_i - 1} + O\pbb{\frac{K^{-1/3 + \epsilon} a^{1/3}}{\abs{d_i - 1} \alpha_+}}\,,
\end{equation*}
as follows from
\begin{equation*}
\abs{w_\phi(z_0) + 1} \;\leq\; C \sqrt{\kappa(z_0) + \eta} \;\leq\; K^{- 1/3 + \epsilon} a^{1/3}\,,
\end{equation*}
where we used Lemma \ref{lemma: w}. Moreover, we use that
\begin{equation*}
\lambda_b / \mu_a \;=\; 1 + O(K^{-2/3})\,.
\end{equation*}
Using that $\Xi$ has high probability for all $\epsilon > 0$ and recalling the isotropic delocalization bound \eqref{smfy gen}, we therefore get for any random $H$ that
\begin{equation} \label{Y_v_i_final}
\ind{\Omega} \scalar{\f v_i}{\f \xi_a}^2 \;=\; \ind{\Omega} \frac{\sigma_i}{(d_i - 1)^2} \scalar{\f v_i}{\f \zeta_b}^2 + O_\prec \pbb{\frac{\sigma_i}{\abs{d_i - 1}^2 M} \, \frac{K^{-1/3 + \tau/5} a^{1/3}}{\alpha_+}}\,.
\end{equation}

We proved \eqref{Y_v_i_final} under the assumption that $i \in \cal R$, but a continuity argument analogous to that given after \eqref{vPv for R} implies that \eqref{Y_v_i_final} holds for all $i \in \qq{1,M}$. The above argument may be repeated verbatim to yield
\begin{equation*}
\ind{\Omega} \scalar{\f v_i}{\f \xi_a} \scalar{\f \xi_a}{\f v_j} \;=\; \ind{\Omega} \frac{\sqrt{\sigma_i\sigma_j}}{(d_i - 1)(d_j - 1)} \scalar{\f v_i}{\f \zeta_b}\scalar{\f \zeta_b}{\f v_j} + O_\prec \pbb{\frac{\sqrt{\sigma_i \sigma_j}}{\abs{d_i - 1}  \abs{d_j - 1} M} \, \frac{K^{-1/3  + \tau/5} a^{1/3}}{\alpha_+}}\,.
\end{equation*}
Since we may always choose the basis $\{\f v_i\}_{i = 1}^M$ so that at most $\abs{\cal R} + 1$ components of $(w_1, \dots, w_M)$ are nonzero, the claim now follows easily.
\end{proof}

\section{Quantum unique ergodicity near the soft edge of $H$} \label{sec: QUE}

This section is devoted to the proof of Proposition \ref{prop:xie}.

\begin{lemma}\label{CTG}
Fix $\tau \in (0,1)$. Let $h$ be a smooth function satisfying 
\begin{equation} \label{theta_assumption}
\abs{h'(x)} \;\leq\; C (1 + \abs{x})^{C}
\end{equation}
for some positive constant $C$.  Let $a \leq K^{1 - \tau}$  and suppose that $\lambda_a$ satisfies \eqref{level repulsion} with some constants $\e$ and $\delta$. Then for small enough $\delta_1 = \delta_1(\epsilon, \delta)$ and $\delta_2 = \delta_2(\epsilon, \delta, \delta_1)$ the following holds. Defining
\be\label{defEI}
\eta \;\deq\; \Delta_a K^{ -2 \epsilon}, \qquad  E^\pm \;\deq\; E\pm K^{ \delta_1}\eta\,, \qquad  I\;\deq\; \qb{\gamma_a - K^{ \delta_2} \Delta_a  \,,\,  \gamma_a +K^{ \delta_2} \Delta_a}\,,
\ee
we have
\be\label{114}
\E  \, h \pb{M \scalar{\f w} {\f \zeta_a}^2}
 -
\E  \,h \pbb{\frac M\pi \int_{I}  \im G_{\f w \f w }(E + \ii \eta) \, \chi(E) \, \dd E}
\;=\; O(K^{-\delta_2/2})\,,
\ee
where we defined $\chi(E)  \deq \ind{\lambda_{a+1} \leq E^- \leq \lambda_a}$. % and introduced the convention $\lambda_0 \deq + \infty$. 
\end{lemma}

\begin{proof}
By the assumption \eqref{theta_assumption} on $h$, rigidity \eqref{rigidity1}, and delocalization \eqref{smfy gen}, we can write
\be\label{19}
\E  \,h \pb{M  \scalar{\f w}{\f  {\f \zeta}_a}^2}
\;=\;
\E \,h \pbb{\frac{M\eta}\pi \,  \int_{I\cap [\alpha,\beta]} \frac{ \scalar{\f w}{\f  {\f \zeta}_a}^2}{(E-\lambda_a)^2+\eta^2} \dd 
E} + O(K^{-\delta_1/2})
\ee
provided that
\begin{equation*}
\alpha \;\leq\; \lambda_a^-\,, \qquad \beta \;\geq\; \lambda_a^+\,,
\end{equation*}
where we defined  $\lambda_a^\pm \deq \lambda_a \pm K^{\delta_1}\eta$.
For the following we choose
\begin{align*}
\alpha \;\deq\; \lambda_a^- \wedge \lambda_{a + 1}^+\,, \qquad
\beta \;\deq\; \lambda_a^+\,.
\end{align*}
Now from \eqref{level repulsion} we get $\P(\lambda_{a+1}^+\ge \lambda_a^-)\le K^{-\delta}$ for $\delta_1 < \epsilon$. For $\delta_1 < \delta \wedge \epsilon$, we therefore get
\begin{equation} \label{uu cutoff}
\E  \,h \pb{M \absb{\scalar{\f w}{\f  {\f \zeta}_a }}^2}
\;=\;
\E \,h\left(\frac{M\eta}\pi \int_{I} \frac{\absb{\scalar{\f w}{\f  {\f \zeta}_a }}^2}{(E-\lambda_a)^2+\eta^2} \, 
\chi(E) \, \dd E
\right) + O(K^{-\delta_1/2})\,.
\end{equation}
In order to obtain \eqref{114}, we have to rewrite the integrand on the right-hand side of \eqref{uu cutoff} in terms of
%Next, we replace the integrand in \eqref{uu cutoff} by $  \im G_{\f w\f w}(E+ \ii \eta)$. By spectral decomposition, we have
\be \label{decomposition of tilde G}
   \im G_{\f w\f w}(E+ \ii \eta)
\;=\;
\sum_{b \neq a} \frac{\eta \scalar{\f w} {\f \zeta_b}^2}{(E-\lambda_b)^2+\eta^2}
+\frac{\eta\scalar{\f w}{\f  {\f \zeta}_a}^2}{(E-\lambda_a)^2+\eta^2}\,.
\ee
Hence \eqref{uu cutoff} and \eqref{decomposition of tilde G} combined with the mean value theorem imply that the left-hand side of \eqref{114} is bounded by
\be 
 \label{1119} K^{C \delta_2}
 \E \, \sum_{b \neq a} \frac{M\eta}\pi \, \int_{I} \frac{\scalar{\f w} {\f \zeta _b}^2}{(E-\lambda_b)^2+\eta^2}\, \chi(E) \, \dd E + C K^{-\delta_1 / 2}
\ee
for any fixed $\delta_2 \in (0,\delta_1)$. When applying the mean value theorem, we estimated the value of $\theta'(\cdot)$ using \eqref{theta_assumption},  the fact that all terms on the right-hand side of \eqref{decomposition of tilde G} are nonnegative, and the estimate
\begin{equation} \label{bound_for_mvt}
M \int_I \im G_{\f w \f w}(E + \ii \eta) \, \dd E \;\prec\; K^{\delta_2}\,.
\end{equation}
The proof of \eqref{bound_for_mvt} follows by using the spectral decomposition from \eqref{decomposition of tilde G} with the delocalization bound \eqref{smfy gen}; for $\abs{b - a} \geq K^{\delta_2}$ we use the rigidity bound \eqref{rigidity1}, and for $\abs{b - a} \leq K^{\delta_2}$ we estimate the integral using $\int \frac{\eta}{e^2 + \eta^2} \, \dd e = \pi$. We omit the full details.

Next, using the eigenvalue rigidity from \eqref{rigidity1}, it is not hard to see that there exists a constant $C_1$ such that the contribution of $|b - a| \geq K^{C_1\delta_2}$ to \eqref{1119} is bounded by $K^{-\delta_2}$. In order to prove \eqref{114}, therefore, it suffices to prove
\be\label{1119-3}
\E \, \sum_{b \neq a \col |b -a| \leq K^{C_1\delta_2} }\frac{M \eta}\pi \, \int_{I} \frac{\scalar{\f w} {\f \zeta _b}^2}{(E-\lambda_b)^2+\eta^2}\, \chi(E) \, \dd E \;=\; O(K^{-\delta_2/2})\,.
\ee
For $b > a$, we get using \eqref{smfy gen} that
\begin{align*}
\sum_{b > a : |b - a| \leq K^{C_1\delta_2}} \frac{M\eta}\pi \, \E  \int_{I} \frac{\scalar{\f w} {\f \zeta _b}^2}{(E-\lambda_b)^2+\eta^2}\, \chi(E) \, \dd E
&\;\leq\; K^{C_1\delta_2}\, \E \int_{\lambda_{a +1}^+}^{\infty} \frac{\eta}{(E-\lambda_{a + 1})^2+\eta^2} \, 
\dd E
\\
&\;\leq\; C K^{ C_1\delta_2 -\delta_1/2}\,,
\end{align*}
which is the right-hand side of \eqref{1119-3} provided $\delta_2$ is chosen small enough.
Here in the first step we replaced $\lambda_b$ with $\lambda_{a+1}$ using the estimates $\lambda_b \leq \lambda_{a+1} \leq E - K^{\delta_1} \eta$ valid for $b > a$ and $E$ in the support of $\chi$.
 
For $b < a$, we partition $I = I_1 \cup I_2$ with $I_1 \cap I_2 = \emptyset$ and
\begin{equation*}
I_1 \;\deq\; \Big\{E\in I \st \exists \, b < a, \;  |b -a| \leq K^{C_1\delta_2} \,,\, 
|E-\lambda_b|\leq \eta K^{\delta_1}
	 \Big\}\,.
\end{equation*}
As above, we find
\begin{align*}
\sum_{b < a: |b - a| \leq K^{C_1\delta_2}} \frac{M\eta}\pi \, \E  \int_{I_2} \frac{\scalar{\f w} {\f \zeta _b}^2}{(E-\lambda_b)^2+\eta^2}\, \chi(E) \, \dd E \;\leq\; K^{ C_2\delta_2 -\delta_1/2}\,.
\end{align*}
Let us therefore consider the integral over $I_{1}$. One readily finds, for $\lambda_a \leq 
\lambda_{a -1} \leq \lambda_b$, that
\begin{equation*}
\frac{  1\, }{(E-\lambda_b)^2+\eta^2} \, \ind{E^- \leq \lambda_a}
\;\leq\;
 \frac{K^{2\delta_1}}{(\lambda_b-\lambda_a)^2+\eta^2}
 \;\leq\; \frac{K^{2\delta_1}}{(\lambda_{a + 1}-\lambda_a)^2+\eta^2}\,.
\end{equation*}
Using delocalization \eqref{smfy gen} we therefore find that 
\begin{align} \label{118}
\sum_{b < a: |b -a| \leq K^{C_1\delta_2}} \frac{M\eta}\pi \, \E  \int_{I_1} \frac{\scalar{\f w} {\f \zeta _b}^2}{(E-\lambda_b)^2+\eta^2}\, \chi(E) \, \dd E \;\leq\; K^{C_1\delta_2+2\delta_1} \E  \,
\frac{\eta^2}{(\lambda_{a - 1}-\lambda_a)^2+\eta^2}\,.
\end{align}
The expectation $\E  \,
\frac{\eta^2}{(\lambda_{a - 1}-\lambda_a)^2+\eta^2}$ in \eqref{118} is bounded by $\P(|\lambda_{a - 1}-\lambda_a|\leq \Delta_a K^{-\e}) +O(K^{-\e})$.  Using \eqref{level repulsion}, we therefore obtain 
\begin{equation*}
\sum_{b< a : \abs{b -a} \leq K^{C_1\delta_2}}  \frac{M\eta}\pi \, \E  \int_{I_1} \frac{\scalar{\f w} {\f \zeta _b}^2}{(E-\lambda_b)^2+\eta^2}\, \chi(E) \, \dd E \;\leq\; K^{C_1\delta_2+2\delta_1-\delta}\,.
\end{equation*}
This concludes the proof.
\end{proof}

In the next step, stated in Lemma \ref{GCC} below, we replace the sharp cutoff function $\chi$ in \eqref{114} with a smooth function of $H$. Note first that from Lemma \ref{CTG} and the rigidity \eqref{rigidity1}, we get
\be\label{aaz}
\E  \, h \pb{M \scalar{\f w}{\f  {\f \zeta}_a}^2}
 -
\E  \,h \pbb{\frac M\pi \int_{I} \im G_{\f w \f w }(E + \ii \eta) \, 
\indb{\cal N( E^-,\tilde E)= a} \, \dd E}
\;=\; O(K^{-\delta_2/2})\,,
\ee
where $\tilde E \deq \gamma_+ + 1$ and $\cal N( E^-,\tilde E) \deq \abs{\{i \col E^- < \lambda_i < \tilde E\}}$ is an eigenvalue counting function. 

Next, for any $E_1, E_2 \in [\gamma_- - 1,  \gamma_+ + 1]$ and $\tilde \eta >0$ we define $f(\lambda) \equiv f_{E_1,E_2,\tilde \eta }(\lambda)$
to be the characteristic function of $[E_1, E_2]$ smoothed on scale $\tilde \eta$:
$f = 1$ on $[E_1, E_2]$, $f = 0$ on $\R\setminus [E_1-\tilde \eta , E_2+\tilde \eta ]$
and $\abs{f'}\le C\,\tilde \eta^{-1}$, $|f''|\le C\,\tilde \eta^{-2}$. Moreover, let $q \equiv q_a:\R \to\R_+$ be a smooth cutoff function 
concentrated around $a$, satisfying
\begin{equation} \label{def_q_cutoff}
q(x) = q_a(x) = 1 \quad  \text{if} \quad |x - a| \le 1/3\,,   \qquad q(x) = 0   \quad  \text{if} \quad |x - a| 
\ge 2/3\,, \qquad \abs{q'} \;\leq\; 6\,.
\end{equation}
The following result is the appropriate smoothed version of \eqref{aaz}. It is a simple extension of Lemma 3.2 and Equation (5.8) from \cite{KY1}, and its proof is omitted.
\begin{lemma}\label{GCC}
Let $\tilde E \deq \gamma_+ + 1$ and
\begin{equation*}
\tilde \eta \;\deq\; \eta  K^{-\epsilon} \;=\; \Delta_a K^{-3\e}\,,
\end{equation*}
and abbreviate $q \equiv q_a$ and $f_E  \equiv f_{ E^-, \tilde E,\tilde \eta }$.
Then under the assumptions of Lemma \ref{CTG} we have
\be\label{115}
\E \, h \pb{M \scalar{\f w}{\f  {\f \zeta}_a}^2}
-
\E  \,h \pbb{\frac M\pi \int_{I} \im G_{\f w \f w }(E + \ii \eta) \,
q\pb{\tr f_E (H)}\, \dd E}
\;=\; O(K^{-\delta_2/2})\,.
\ee
\end{lemma}

We may now conclude the proof of Proposition \ref{prop:xie}.

\begin{proof}[Proof of Proposition \ref{prop:xie}]
The basic strategy of the proof is to compare the distribution of $\scalar{\f w}{\f  {\f \zeta}_a}^2$ under a general $X$ to that under a Gaussian $X$. In the latter case, by unitary invariance of $H = X X^*$, we know that $\f \zeta_a$ is uniformly distributed on the unit sphere of $\R^M$, so that $M\scalar{\f w}{\f  {\f \zeta}_a}^2 \to \chi_1^2$ in distribution.

For the comparison argument, we use the Green function comparison method applied to the Helffer-Sj\"ostrand representation of $f(H)$. Using Lemma \ref{GCC} it suffices to estimate
\be\label{ghsjd}(\E -\E^{\txt{Gauss}}) \,h \pbb{\frac M\pi \int_{I} \im G_{\f w \f w }(E + \ii \eta) \,
q(\tr f_E(H))\, \dd E}\,,
\ee
where $\E^{\txt{Gauss}}$ denotes the expectation with respect to Gaussian $X$.
Now we express $f (H)$ in terms of Green functions using Helffer-Sj\"ostrand functional calculus. Recall the definition of $\kappa_a$ from \eqref{def_kappa_a}. Let $g(y)$ be a 
smooth cutoff function with support in $[-\kappa_a, \kappa_a]$, with $g(y)=1$ for
$|y| \leq \kappa_a/2$ and $\|g^{(n)}\|_\infty \le C\kappa_a^{-n}$, where $g^{(n)}$ denotes the $n$-th derivative of $g$. Then, similarly to \eqref{HS formula}, we have (see e.g.\ Equation (B.12) of \cite{ESY5})
\begin{equation*}
 f_E (\lambda) \;=\; \frac1{2\pi}\int_{\R^2}\frac{\ii \sigma  f_E ''(e)g (\sigma)+ \ii  f_E (e) g'(\sigma)-\sigma 
 f_E '(e)g'(\sigma)}{\lambda-e-\ii \sigma} \, \dd e \, \dd \sigma\,.
\end{equation*}
Thus we get the functional calculus, with $G(z)=(H-z)^{-1}$,
\begin{align}\label{HS split 1}
\tr  f_E (H) &\;=\; \frac{1}{2\pi}\int_{\R^2}\pB{\ii \sigma  f_E ''(e)g (\sigma)+ \ii  f_E (e) g'(\sigma)-\sigma 
 f_E '(e)g'(\sigma)} \tr G(e + \ii \sigma) \, \dd e \, \dd \sigma
\notag \\
&\;=\; \frac{1}{2\pi}\int_{\R^2}\pB{\ii  f_E (e) g'(\sigma)-\sigma  f_E '(e)g'(\sigma)} \tr G(e + \ii \sigma)\, \dd e \, 
\dd \sigma
\notag \\ 
&\qquad {}+{} \frac{\ii }{2 \pi} \int_{\abs{\sigma} > \tilde \eta K^{-d \epsilon}} \dd \sigma \, g(\sigma) \int \dd e \; f''_E(e) 
\, \sigma   \tr G(e + \ii \sigma)
\notag \\ 
&\qquad {}+{}
\frac{\ii}{2 \pi} \int_{-\tilde \eta K^{-d \epsilon}}^{\tilde \eta K^{-d \epsilon}} \dd \sigma \int \dd e \; f''_E(e) \, \sigma   \tr G(e + \ii \sigma)
\,.
\end{align}
As in Lemma 5.1 of \cite{KY1}, one can easily extend \eqref{bound: Rij isotropic gen} to $\eta$ satisfying the lower bound $\eta > 0$ instead of $\eta \geq K^{-1 + \omega}$ in \eqref{def_S_theta}; the proof is identical to that of \cite[Lemma 5.1]{KY1}. Thus we have, for $e \in [\gamma_+ - 1, \gamma_+ + 1]$ and $\sigma \in (0, 1)$,
\be 
\sigma   \tr G(e + \ii \sigma) \;=\; O_\prec(1)\,.
\ee
Therefore, by the trivial symmetry $\sigma \mapsto -\sigma$ combined with complex conjugation, the third term on the right-hand side of \eqref{HS split 1} is bounded  by
\begin{equation} \label{small sigma error}
\frac{\ii  }{2 \pi} \int_{-\tilde \eta K^{-d \epsilon}}^{\tilde \eta K^{-d \epsilon}} \dd \sigma \int \dd e \; f''_E(e) \, \sigma   \tr G(e + \ii \sigma)  \;=\; O_\prec( K^{-d \epsilon})\,,
\end{equation}
where  we used that $\int \abs{f''_E(e)} \, \dd e = O(\tilde \eta^{-1})$. Next, we note that \eqref{bound: Rij isotropic gen} and Lemma \ref{lemma: w} imply
\begin{equation} \label{bound on integral}
\frac M\pi \int_{I} \im G_{\f w \f w }(E + \ii \eta) \, \dd E \;=\; O_\prec(K^{3 \epsilon})\,.
\end{equation}
Recalling \eqref{theta_assumption} and using the mean value theorem, we find from \eqref{115}, \eqref{small sigma error}, and \eqref{bound on integral} that for large enough $d$, in order to estimate \eqref{ghsjd}, and hence prove \eqref{fluct of vect comp}, it suffices to prove the following lemma. Note that in it we choose $X^{(1)}$ to be the original ensemble and $X^{(2)}$ to be the Gaussian ensemble.
\end{proof}

\begin{lemma} \label{CPZX}
Suppose that the two $M \times N$ matrix ensembles $X^{(1)}$ and $X^{(2)}$ satisfy \eqref{cond on entries of X} and \eqref{moments of X-1}. Suppose that the assumptions of Lemma \ref{CTG} hold, and recall the notations
\be\label{defEI-2}
\eta \;\deq\; \Delta_a K^{-2 \epsilon}\,,
\qquad
\tilde \eta \;\deq\; \Delta_a K^{-3\e}\,,
\qquad  
E^\pm \;\deq\; E\pm K^{ \delta_1}\eta\,, \qquad 
\tilde E \;\deq\;  \gamma_+ + 1\,,
\ee
as well as
\be\label{defEI-3}
f_E \;\equiv\; f_{E^-, \tilde E, \tilde \eta}\,, \qquad I \;\deq\; \qb{\gamma_a - K^{ \delta_2} \Delta_a  \,,\,  \gamma_a +K^{ \delta_2} \Delta_a}\,,
\ee
where $f_{E^-, \tilde E, \tilde \eta}$ was defined above \eqref{def_q_cutoff}. Recall $q \equiv q_a$ from \eqref{def_q_cutoff}. Finally, suppose that $\e>4\delta_1$ and $\delta_1>4 \delta_2$.

Then for any $d > 1$ and for small enough $\e \equiv \epsilon(\tau,d) > 0$ and $\delta_2 \equiv \delta_2(\epsilon, \delta, \delta_1)$ we have
\be \label{ee427}
\qb{\E^{(1)} - \E^{(2)}} \,h \qbb{ \int_{I} x(E) \, q \pb{y(E) }   \, \dd E } \;=\; O(K^{-\delta_2})\,,
\ee
where we defined
\be\label{defxE}
x(E) \;\deq\; \frac {M }\pi   \im G_{\f w \f w}(E + \ii \eta)
\ee
and 
\begin{align}
y(E) &\;\deq\; \frac 1{2\pi}\int_{\R^2}
\ii \sigma f_E''(e) g (\sigma)  \tr G(e + \ii \sigma)\,  \indb{\abs{\sigma} > \tilde \eta K^{-d \e}} \,\dd e \, \dd \sigma
\nonumber \\ \label{defy428}
&\qquad \;+\; 
 \frac 1{2\pi}\int_{\R^2}
\pb{ \ii f_E(e) g'(\sigma)- \sigma f_E'(e)g'(\sigma)} \tr G(e + \ii \sigma) \, \dd e \, \dd \sigma\,.
\end{align}
\end{lemma}

The rest of this section is devoted to the proof of Lemma \ref{CPZX}.
 
\subsection{Proof of Lemma \ref{CPZX} I: preparations} \label{sec: 81}
We shall use the Green function comparison method \cite{EYY1,EYY3,KY1} to prove Lemma \ref{CPZX}. For definiteness, we assume throughout the remainder of Section \ref{sec: QUE} that $\phi \geq 1$. The case $\phi < 1$ is dealt with similarly, and we omit the details.

We first collect some basic identities and estimates that serve as a starting point for the Green function comparison argument. We work on the product probability space of the ensembles $X^{(1)}$ and $X^{(2)}$. We fix a bijective 
ordering map $\Phi$ on the index set of the matrix entries,
\[
\Phi \col \{(i, \mu) \col 1\le i\le M, \; 1\le  \mu \le N \} \;\longrightarrow\; \qq{1, MN}\,,
\]
and define the interpolating matrix $X_\gamma$, $\gamma \in \qq{1, MN}$, through
\begin{equation*}
(X_\gamma)_{i \mu} \;\deq\;
\begin{cases}
X_{i \mu}^{(1)} & \text{if } \Phi(i,\mu) > \gamma
\\
X_{i \mu}^{(2)} & \text{if } \Phi(i,\mu) \leq \gamma\,.
\end{cases}
\end{equation*}
In particular, $X_0 = X^{(1)}$ and $ X_{MN} = X^{(2)}$. Hence we have the telescopic sum
\begin{align}\label{tel}
\big [ \E^{(1)} - \E^{(2)} \big ]  \,h \left[\int_{I } x(E) \, q (y(E) ) \, \dd E \right] \;=\; \sum_{\gamma = 
1}^{MN} \Big [ \E^{X_{\gamma-1}}  - \E^{X_{\gamma}} \Big ]  \,h \left[  \int_{I } x(E) \, q 
(y(E) ) \, \dd E\right]
\end{align}
(in self-explanatory notation).

Let us now fix a $\gamma$ and let $(b,\beta)$ be determined by $\Phi (b, \beta) = \gamma$. Throughout the following we consider 
$b, \beta$ to be arbitrary but fixed and often omit dependence on them from the notation. Our strategy is to 
compare $X_{\gamma-1}$ with $X_\gamma$ for each $\gamma$. In the end we shall sum up the differences in the telescopic 
sum \eqref{tel}.

Note that $X_{\gamma - 1}$ and $X_\gamma$ differ only in the matrix entry indexed by $(b,\beta)$. 
Thus we may write
\begin{align}
X_{\gamma-1} &\;=\; \bar X + \wt U\,, \qquad \wt U_{i \mu} \;\deq\; \delta_{i b} \delta_{\mu \beta} X^{(1)}_{b\beta}
 \,,
\notag \\ \label{defHg1}
X_\gamma &\;=\; \bar X + U\,, \qquad U_{i \mu} \;\deq\; \delta_{i b} \delta_{\mu \beta} X^{(2)}_{b\beta}  \,.
\end{align}
Here $\bar X$ is the matrix obtained from $X_{\gamma}$ (or, equivalently, from $X_{\gamma - 1}$) by setting the entry indexed by $(b, \beta)$ to zero.
Next, we define the resolvents
\be\label{defG}
T(z) \;\deq\; (\bar X\bar X^*-z)^{-1}\,,\qquad  \qquad   S(z) \;\deq\; (X_\gamma X_\gamma^*-z)^{-1}\,. 
\ee
We shall show that the difference between the expectations $\E^{X_{\gamma}}$ and $\E^{\bar X}$ depends only on the first two moments of $X^{(2)}_{b \beta}$, up to an error term that is negligible even after summation 
over $\gamma$. Together with same argument applied to $\E^{X_{\gamma -1}}$, and the fact that the second moments of 
$X^{(1)}_{b \beta}$ and $X^{(2)}_{b \beta}$ coincide, this will prove Lemma \ref{CPZX}.

We define $x_T(E)$ and $y_T(E)$ as in \eqref{defxE} and \eqref{defy428} with $G$ replaced by $T$, and similarly $x_S(E)$ and $y_S(E)$ with $G$ replaced by $S$. Throughout the following we use the notation $\f w = (w(i))_{i = 1}^M$ for the components of $\f w$. In order to prove \eqref{ee427} using \eqref{tel}, it is enough to prove that for some constant $c > 0$ we have
\be\label{Nva-1}
\E h \left[  \int_{I} x_S(E) \, q \pb{y_S(E) }   \, \dd E  \right]
-
 \E h \left[  \int_{I} x_T(E) \, q \pb{y_T(E) }   \, \dd E  \right]
 =\E\cal A+O (K^{-c}) \pb{\phi^{-1}K^{-2}
 +K^{-1} \abs{w(b)}^2}\,,
\ee
where $\cal A$ is polynomial of degree two in $U_{b \beta}$ whose coefficients are $\bar X$-measurable.

The rest of this section is therefore devoted to the proof of \eqref{Nva-1}. Recall that we assume throughout that $\phi \geq 1$ for definiteness; in particular, $K = N$.

We begin by collecting some basic identities from linear algebra. In addition to $G(z) \deq (X X^* - z)^{-1}$ we introduce the auxiliary resolvent $R(z) \deq (X^* X - z)^{-1}$. Moreover, for $\mu \in \qq{1,M}$ we split
\begin{equation*}
X \;=\; X_{[\mu]} + X^{[\mu]} \,, \qquad (X_{[\mu]})_{i \nu} \;\deq\; \ind{\nu = \mu} X_{i \nu}\,, \qquad (X^{[\mu]})_{i \nu} \;\deq\; \ind{\nu \neq \mu} X_{i \nu}\,.
\end{equation*}
We also define the resolvent $G^{[\mu]} \deq (X^{[\mu]} (X^{[\mu]})^* - z)^{-1}$. A simple Neumann series yields the identity
\begin{equation} \label{Neumann}
G \;=\; G^{[\mu]} - \frac{G^{[\mu]} X_{[\mu]} X_{[\mu]}^* G^{[\mu]}}{1 + (X^* G^{[\mu]} X)_{\mu \mu}}\,.
\end{equation}
Moreover, from \cite[Equation (3.11)]{BEKYY}, we find
\begin{equation} \label{R_mumu}
z R_{\mu \mu} \;=\; - \frac{1}{1 + (X^* G^{[\mu]} X)_{\mu \mu}}\,.
\end{equation}
From \eqref{Neumann} and \eqref{R_mumu} we easily get
\begin{equation} \label{G_X}
GX_{[\mu]} \;=\; -z \,R_{\mu\mu} \, G^{[\mu]}X_{[\mu]}\,, \qquad
X^*_\mu G \;=\; -z R_{\mu \mu} X_{[\mu]}^* G^{[\mu]}\,.
\end{equation}

Throughout the following we shall make use of the fundamental error parameter
\begin{equation} \label{def_Psi}
\Psi(z) \;\deq\; \sqrt{\frac{\im m_\phi(z)}{N \eta}} + \frac{1}{N \eta}\,,
\end{equation}
which is analogous to the right-hand side of \eqref{bound: Rij isotropic gen} and will play a similar role.
We record the following estimate, which is analogous to Theorem \ref{thm: IMP gen}.

\begin{lemma}\label{thm: bound on RX}
Under the assumptions of Theorem \ref{thm: IMP gen} we have, for $z \in \f S$,
\begin{equation}\label{bound: XRv}
\absb{(GX )_{\f w \mu} }  \;\prec\;  \phi^{-1/4} \Psi
\end{equation}
and 
\begin{equation}\label{bound: XRX}
 \absb{(X^*GX)_{\mu\nu} - \delta_{\mu\nu}(1+zm_{\phi}) } \;\prec\; \phi^{1/2}\Psi \,.
 \end{equation}
 \end{lemma}

\begin{proof}
This result is a generalization of (5.22) in \cite{PY}. The key identity is \eqref{G_X}.
Since $G^{[\mu]}$ is independent of $(X_{i \mu})_{i = 1}^N$, we may apply the large deviation estimate \cite[Lemma 3.1]{BEKYY} to $G^{[\mu]}X_{[\mu]}$. Moreover, $\abs{R_{\mu \mu}} \prec 1$, as follows from Theorem \ref{thm: IMP gen} applied to $X^*$, and Lemma \ref{lemma: w}. Thus we get
\begin{multline*}
\abs{(GX )_{\f w \mu}} \;\prec\; \phi^{1/2} \pBB{(M N)^{-1/2} \sum_{i = 1}^M \abs{G_{\f w i}^{[\mu]}}^2}^{1/2} \;=\; \phi^{1/4} \pBB{\frac{1}{N \eta} \im G_{\f w \f w}^{[\mu]}}^{1/2}
\\
\prec\; \phi^{1/4} \pBB{\frac{1}{N \eta} \pbb{\im m_{\phi^{-1}} + \phi^{-1} \Psi}}^{1/2}
\;\leq\; C \phi^{-1/4} \Psi\,,
\end{multline*}
where the second step follows by spectral decomposition, the third step from Theorem \ref{thm: IMP gen} applied to $X^{[\mu]}$ as well as \eqref{im_m_swap}, and the last step by definition of $\Psi$. This concludes the proof of \eqref{bound: XRv}.
 
Finally, \eqref{bound: XRX} follows easily from Theorem \ref{thm: IMP gen} applied to the identity
$X^*G X = 1 + zR$. 
\end{proof}

After these preparations, we continue the proof of \eqref{Nva-1}. We first expand the difference between $S$ and $T$ in terms of $V$ (see \eqref{defHg1}).
We use the resolvent expansion: for any $m \in \N$ we have
\be\label{SR-N}
	 S \;=\; T + \sum_{k=1}^m (-1)^{k}[T( \bar XU^*+U \bar X^* +  U U^* )]^kT
	 + (-1)^{m+1} [T(  \bar XU^*+U \bar X^*+  U U^* )]^{m+1} S
\ee
and
\be\label{RS-N}
	 T \;=\;  S + \sum_{k=1}^m [S( XU^*+UX^* +  U U^* )]^kS
	 +  [S( XU^*+UX^*+  U U^* )]^{m+1} T\,.
\ee
Note that Theorem \ref{thm: IMP gen} and Lemma \ref{thm: bound on RX} immediately yield for $z \in \f S$
\begin{equation*}
\absb{S_{\bv \bw}-\scalar{\f v}{\f w} m_{\phi^{-1}}} \;\prec\; \phi^{-1}\Psi\,, 
\qquad
 \abs{(SX_\gamma)_{\bv i}} \;\prec\; \phi^{-1/4}\Psi\,, 
  \qquad 
\absb{(X_\gamma^* SX_\gamma)_{\mu\nu} - \delta_{\mu\nu}(1+zm_{\phi })}  \;\prec\;  \phi^{1/2}\Psi
\end{equation*}
Using \eqref{RS-N}, we may extend these estimates to analogous ones on $\bar X$ and $T$ instead of $X_\gamma$ and $S$. Indeed, using the facts $\norm{R} \leq \eta^{-1}$, $\Psi \geq N^{-1/2}$, and $\abs{U_{b \beta}} \prec \phi^{-1/4}N^{-1/2}$ (which are easily derived from the definitions of the objects on the left-hand sides) combined with \eqref{RS-N}, we get the following result.
\begin{lemma}\label{princ} For $A\in\{S,T\}$ and $B\in\{X_\gamma, \bar X\}$ we have
\begin{equation*}
 \absb{A _{\bv \bw}-\scalar{\bv}{\bw} m_{\phi^{-1}}} \;\prec\; \phi^{-1}\Psi\,, 
\qquad
 \abs{( A B)_{\bv i}} \;\prec\; \phi^{-1/4}\Psi \,, 
  \qquad 
  \absb{(B^*  A   B )_{\mu\nu} - \delta_{\mu\nu}(1+zm_{\phi }(z))}  \;\prec\;  \phi^{1/2}\Psi\,.
\end{equation*}
\end{lemma}

The final tool that we shall need is the following lemma, which collects basic algebraic properties of stochastic domination $\prec$. We shall use them tacitly throughout the following. Their proof is an elementary exercise using union bounds and Cauchy-Schwarz. See \cite[Lemma 3.2]{BEKYY} for a more general statement.
\begin{lemma} \label{lemma: basic properties of prec}
\begin{enumerate}
\item
Suppose that $A(v) \prec  B(v)$ uniformly in $v \in V$. If $\abs{V} \leq N^C$ for some constant $C$ then $\sum_{v \in V} A(v) \prec \sum_{v \in V} B(v)$.
\item
Suppose that $A_{1} \prec B_{1}$ and $A_{2} \prec B_{2}$. Then $A_{1} A_{2} \prec B_{1} B_{2}$.
\item
Suppose that $\Psi \geq N^{-C}$ is deterministic and $A$ is a nonnegative random variable satisfying $\E A^2 \leq N^{C}$. Then $A \prec \Psi$ implies that $\E A \prec \Psi$.
\end{enumerate}
If the above random variables depend on an additional parameter $u$ and all hypotheses are uniform in $u$ then so are the conclusions.
\end{lemma}

\subsection{Proof of Lemma \ref{CPZX} II: the main expansion} \label{sec: 82}
Lemma \ref{princ} contains the a-priori estimates needed to control the resolvent expansion \eqref{SR-N}. The precise form that we shall need is contained in the following lemma, which is our main expansion. Define the control parameter
\begin{equation*}
\Psi_b(z) \;\deq\; \phi^{1/2} \abs{w(b)} + \Psi(z)\,,
\end{equation*}
where we recall the notation $\f w = (w(i))_{i = 1}^M$ for the components of $\f w$.

\begin{lemma}[Resolvent expansion of $x(E)$ and $y(E)$] \label{zddz}
The following results hold for $E \in I$. (Recall the definition \eqref{defEI-3}. For brevity, we omit $E$ from our notation.)
\begin{enumerate}
\item
We have the expansion
\begin{equation} \label{lghj} 
x_S - x_T \;=\; \sum_{l=1}^3 x_l \, U_{b\beta}^l
+ O_\prec\pB{\phi^{-1} N^{-1} (\phi^{1/4} \abs{w(b)} + \Psi)^2}\,,
\end{equation}
where $x_l$ is a polynomial, with constant number of terms, in the variables
\begin{equation*}
\hb{T_{bb}, T_{\f w  b}, T_{b \f w}, (T\bar X)_{\f w \beta}, (\bar X^*T)_{\beta \f w}, (\bar X^*T\bar X)_{\beta\beta}}\,.
\end{equation*}
In each term of $x_l$, the index $\f w$ appears exactly twice, while the indices $b$ and $\beta$ each appear exactly $l$ times. 

Moreover, we have the estimates
\begin{equation} \label{bzdbb}
\abs{x_1} + \abs{x_3} \;\prec\; \phi^{-1/4}N \Psi\Psi_b\,,
\qquad
\abs{x_2} \;\prec\; \phi^{-1/2}N  \Psi_b^2+N\Psi^2\,,
\end{equation}
where the spectral parameter on the right-hand side is $z = E + \ii \eta$.
\item
We have the expansion
\begin{equation} 
\label{lghj2}
\tr S - \tr T \;=\; \sum_{l=1}^3 J_l U_{b\beta}^l + O_\prec\pb{\phi^{-1} N^{-1} \Psi^2}\,,
\end{equation}
where $J_l$ is a polynomial, with constant number of terms, in the variables
\begin{equation*}
\hb{T_{bb}, (T^2)_{bb}, (T^2 \bar X)_{b\beta}, (\bar X^*T^2)_{\beta b}, (\bar X^*T\bar X)_{\beta\beta}, (\bar X^*T^2\bar X)_{\beta\beta}}\,.
\end{equation*}
In each term of $J_l$, $T^2$ appears exactly once, while the indices $b$ and $\beta$ each appear exactly $l$ times. 

Moreover, for $z \in \f S$ we have the estimates
\begin{equation} \label{lghj2_est}
\abs{J_1} + \abs{J_3} \;\prec\; \phi^{-1/4}N\Psi^2\,, \qquad \abs{J_2} \;\prec\; N  \Psi^2\,.
\end{equation}

\item
Defining
\begin{equation*}
y_l \;\deq\; \frac{1}{2\pi}\int_{\R^2} J_{l}     \, 
\Big( \ii \sigma f_E''(e) g(\sigma)  \,   \indb{\abs{\sigma} > \tilde \eta N^{-d \e}}
+\ii f_E(e) g'(\sigma)- \sigma f_E'(e)g'(\sigma)\Big)\,\dd e \, \dd 
\sigma \,,
\end{equation*}
we have the expansion
\begin{equation} \label{lghj3}
y_S - y_T \;=\; \sum_{l=1}^3 y_{l} U_{a\beta}^l+ O_\prec\pb{N^{C \epsilon} \phi^{-1} N^{-2} \kappa_a^{1/2}}
\end{equation}
together with the bounds
\begin{equation} \label{bzdbby}
\abs{y_1} + \abs{y_3} \;\prec\; \phi^{-1/4}N^{C\e}\kappa_a^{1/2}\,, \qquad \abs{y_2}  \;\prec\;   N^{C\e}\kappa_a^{1/2}\,.
\end{equation}
Here all constants $C$ depend on the fixed parameter $d$.
\end{enumerate}
\end{lemma}

\begin{proof}
The proof is an application of the resolvent expansion \eqref{SR-N} with $m = 3$ to the definitions of $x$ and $y$.

We begin with part (i). The expansion \eqref{lghj} is obtained by expanding the resolvent $S_{\f w \f w}$ in the definition of $x_S$ using \eqref{SR-N} with $m = 3$. The terms are regrouped according the power, $l$, of $U_{b \beta}$. The error term of \eqref{lghj} contains all terms with $l \geq 4$. It is a simple matter to check that the polynomials $x_l$, for $l = 1,2,3$, have the claimed algebraic properties. In order to establish the claimed bounds on the terms of the expansion, we use Lemma \ref{princ} to derive the estimates
\be\label{uamza}
\abs{T_{\f w b}} \;\prec\; \phi^{-1} \Psi_b,\, 
 \qquad \abs{(T\bar X )_{\f w \beta}} \;\prec\; \phi^{-1 /4}\Psi \,,  
 \qquad \abs{T_{bb}} \;\prec\; C\phi^{-1/2}\,,
 \qquad \abs{(\bar X^*T\bar X)_{\beta\beta}} \;\prec\; \phi^{1/2}\,,
\ee
and the same estimates hold if $T$ is replaced by $S$. Note that in \eqref{uamza} we used the bound $\abs{m_{\phi^{-1}}} \asymp \phi^{-1/2}$, which follows from the identity
\begin{equation*} 
m_{\phi^{-1}}(z) \;=\; \frac{1}{\phi} \pbb{m_\phi(z) + \frac{1 - \phi}{z}}
\end{equation*}
and Lemma \ref{lemma: w}. Using \eqref{uamza}, it is not hard to conclude the proof of part (i).

Part (ii) is proved in the same way as part (i), simply by setting $\f w = \f e_i$ and summing over $i = 1, \dots, M$.

What remains is to prove the bounds in part (iii).  
To that end, we integrate by parts, first in $e$ and then in $\sigma$, 
in the term containing $ f_E''(e)$, and obtain
\begin{multline*}
\int_{\R^2} \ii \sigma f_E''(e) g (\sigma)  J_l(e + \ii \sigma)\,  \ind{\abs{\sigma} > \tilde \eta_d} \,\dd e \, \dd \sigma \;=\; 
\sum_{\pm} \mp \int \tilde \eta_d f_E'(e) g(\pm\tilde \eta_d)  J_l(e \pm \ii \tilde \eta_d)\,\dd e
\\
+ \int_{\R^2} \pb{\sigma g' (\sigma) + g(\sigma)} f_E'(e)  J_l(e + \ii \sigma)\, \ind{\abs{\sigma} > \tilde \eta_d} \,\dd e \, \dd \sigma\,,
\end{multline*}
where we abbreviated $\tilde \eta_d \deq  \tilde \eta N^{-d \epsilon}$. Thus we get the bound
\begin{multline} \label{866}
\abs{y_l(E)} \;\leq\; \int \dd e \, \tilde \eta_d \abs{ f_E'(e)} \absb{J_l(e + \ii \tilde \eta_d)}
+
\int \dd e \, \dd \sigma \, \pb{\abs{f_E(e) g'(\sigma)} + \abs{\sigma f_E'(e)g'(\sigma)}} \absb{J_l(e + \ii \sigma)}
\\ +
\int \dd e \int_{\tilde \eta_d}^\infty \dd \sigma \,  \pb{\sigma \abs{g' (\sigma)} + g(\sigma)} \abs{f_E'(e)} \absb{J_l(e + \ii \sigma)}\,.
\end{multline}
Using \eqref{866}, the conclusion of the proof of part (iii) follows by a careful estimate of each term on the right-hand side, using part (ii) as input. The ingredients are the definitions of $\tilde \eta_d$ and $\kappa_a$, as well as the estimate
\begin{equation*}
\Psi^2(z) \;\leq\; C \frac{\sqrt{\kappa} + \sqrt{\eta}}{N \eta} + \frac{C}{N^2 \eta^2}\,.
\end{equation*}
The same argument yields the error bound in \eqref{lghj3}. This concludes the proof.
\end{proof}

Armed with the expansion from Lemma \ref{zddz}, we may do a Taylor expansion of $q$. To that end, we record the estimate $ \int_I \abs{x_T(E)} \, \dd E \prec N^{C \epsilon}$, as follows from Lemma \ref{princ}. Hence using Lemma \ref{zddz} and expanding $q(y_S(E))$ around $q(y_T(E))$ with a fourth order rest term, we get
\begin{multline} \label{whew}
\int_{I} x_S(E)  \, q \pb{y_S(E) } \, \dd E - \int_{I} x_T(E)  \, q \pb{y_T(E) }   \, \dd E
\\
=\;
\sum_{\f l \in \cal L} A_{\f l} \, U_{b\beta}^{\abs{\f l}}
+O_\prec \pB{\phi^{-1} N^{-2+C\e }\kappa_a^{1/2} + \phi^{-1/2}N^{-1+C\e } \Delta_a \abs{w(b)}^2}\,,
\end{multline}
where we defined
\begin{equation*}
\cal L \;\deq\; \hb{\f l = (l_0, \dots, l_m) \in \qq{0,3} \times \qq{1,3}^m \col m \in \N\,,\, 1 \leq \abs{\f l} \leq 3}\,,
\qquad
\abs{\f l} \;\deq\; \sum_{i = 0}^m l_i\,,
\end{equation*}
as well as the polynomial
\begin{equation} \label{def_Al}
A_{\f l} \;\deq\; \int_I \frac{q^{(m)}(y_T)}{m!} x_{l_0} y_{l_1} \cdots y_{l_m} \, \dd E\,,
\end{equation}
where we abbreviated $m \equiv m(\f l)$. Here we use the convention that $x_0 \deq x_T$.
Note that $\cal L$ is a finite set (it has 14 elements), and for each $\f l \in \cal L$ the polynomial $A_{\f l}$ is independent of $U_{b \beta}$.
In the estimate of the error term on the right-hand side of \eqref{whew} we also used the fact that for $E\in I$ we have $\Psi(E+i\eta) \leq N^{C\e}\kappa_a^{1/2}$

Next, using lemma \ref{zddz} and $N \Delta_a \asymp \kappa_a^{-1/2}$, we find
\begin{equation} \label{mganq0}
\abs{A_{\f l}} \;\prec\;
N^{C \epsilon}
\begin{cases}
\phi^{-1/4} \pb{\kappa_a^{1/2} + \phi^{1/2} \abs{w(b)}} & \text{if } \abs{\f l} = 1
\\
\kappa_a^{1/2} + \kappa_a^{-1/2} \phi^{1/2} \abs{w(b)}^2 & \text{if } \abs{\f l} = 2
\\
\phi^{-1/4} \pb{\kappa_a^{1/2} + \phi^{1/2} \abs{w(b)} + \phi \abs{w(b)}^2} & \text{if } \abs{\f l} = 3\,.
\end{cases}
\end{equation}
Using \eqref{whew} and \eqref{mganq0}, we may do a Taylor expansion of $h$ on the left-hand side of \eqref{Nva-1}. This yields
\begin{multline} \label{lgcld}
h \left[  \int_{I} x_S \, q \pb{y_S }   \, \dd E  \right]
-
h \left[  \int_{I} x_T \, q \pb{y_T }   \, \dd E  \right]
\;=\; \sum_{k=1}^3 \frac{1}{k!}h^{(k)}(A_{\f 0}) \pBB{\sum_{\f l \in \cal L} A_{\f l} U_{b\beta}^{\abs{\f l}}}^{k}
\\
+O_\prec \pB{\phi^{-1} N^{-2+C\e }\kappa_a^{1/2} + N^{ -2+C\e} \kappa_a^{-1/2} \abs{w(b)}^2 }\,,
\end{multline}
where we abbreviated $A_{\f 0} \deq \int_I x_0 \, \dd E$. Since $a \leq N^{1 - \tau}$, it is easy to see that, by choosing $\epsilon$ small enough depending on $\tau$, the error term in \eqref{lgcld} is bounded by $N^{-c} (\phi^{-1} N^{-2} + N^{-1} \abs{w(b)}^2)$ for some positive constant $c$. Taking the expectation and recalling that $\abs{U_{b \beta}} \prec \phi^{-1/4} N^{-1/2}$ , we therefore get
\begin{multline} \label{app}
\E h \left[  \int_{I} x_S \, q \pb{y_S }   \, \dd E  \right]
-
\E h \left[  \int_{I} x_T \, q \pb{y_T }   \, \dd E  \right]
\;=\;\E \cal A +O_\prec \pB{N^{-c} \pb{\phi^{-1} N^{-1} + N^{-1} \abs{w(b)}^2}}
\\
+ \E U_{b \beta}^3 \, \E
\sum_{k=1}^3 \frac{1}{k!}h^{(k)}(A_{\f 0}) \sum_{\f l_1, \dots, \f l_{k} \in \cal L} \indBB{\sum_{i = 1}^{k} \abs{\f l_i} = 3} \prod_{i = 1}^{k} A_{\f l_i}
\,,
\end{multline}
where $\E \cal A$ is as described after \eqref{Nva-1}, i.e.\ it depends on the random variable $U_{b \beta}$ only through its first two moments.

At this point we note that if we make the stronger assumption that the first \emph{three} moments of $X^{(1)}$ and $X^{(2)}$ match (which, in the ultimate application to the proof of Proposition \ref{prop:xie}, means that $\E X_{i \mu}^3 = 0$), the proof is now complete. Indeed, in that case we may allow $\cal A$ to be a polynomial of degree three in $U_{b \beta}$ with $\bar X$-measurable coefficients, and we may absorb the last line of \eqref{app} into $\E \cal A$. This completes the proof of \eqref{Nva-1}, and hence of Lemma \ref{CPZX}, for the special case that the third moments of $X^{(1)}$ and $X^{(2)}$ match.

For the general case, we still have to estimate the last line of \eqref{app}.
The terms that we need to analyse are
\begin{subequations}
\label{terms of degree three}
\begin{gather}
 h^{(3)}(A_{  {\bf 0} })\,
 A_{ {{(1 )}}} ^m A_{  {(0,1)} }^n \qquad (m+n=3)\,,
\\
 h^{(2)}(A_{\f 0} )\,
 \pb{ A_{  (2 ) }+   A_{(0,2) }
 +    A_{ {(1,1)} }
 +     A_{  {(0,1,1)} }
 }    \pb{A_{(1)} +A_{(0,1)}}\,,
\\
 h^{(1)}(A_{\bf 0})\,
 \pb{ A_{(3)}+A_{(0,3)}+  A_{(1,2)}
 + A_{(2,1)}
 + A_{(0,1,2)}
 +A_{(1,1,1)}+A_{(0,1,1,1)}}\,.
\end{gather}
\end{subequations}
These terms are dealt with in the following lemma.
\begin{lemma}\label{lem: 3m}
Let $Y$ denote any term of \eqref{terms of degree three}.
Then there is a constant $c > 0$ such that
\begin{equation} \label{Y_bound_claim}
\abs{\E Y} \;\leq\; N^{-c}\pb{\phi^{-1/4}N^{-1/2}
 +\phi^{3/4} \abs{w(b)}^2}\,.
\end{equation}
\end{lemma}

Plugging the estimate of Lemma \ref{lem: 3m} into \eqref{app}, and recalling that $\E U_{b \beta}^3 \leq C \phi^{-3/4} N^{-3/2}$, it is easy to complete the proof of \eqref{Nva-1}, and hence of Lemma \ref{CPZX}.  Lemma \ref{lem: 3m} is proved in the next subsection.

\subsection{Proof of Lemma \ref{CPZX} III: the terms of order three and proof of Lemma \ref{lem: 3m}} \label{sec: 83}
Recall that we assume $\phi \geq 1$, i.e.\ $K = N$; the case $\phi \leq 1$ is dealt with analogously, and we omit the details.

We first remark that using the bounds \eqref{mganq0} we find
\begin{equation} \label{E_bound_naive}
\abs{\E Y} \;\leq\; N^{C \epsilon} \pb{\kappa_a^{-1/2} \phi^{3/4} \abs{w(b)}^2 + \phi^{1/4} \abs{w(b)} + \phi^{-1/4} \kappa_a^{1/2}}\,.
\end{equation}
Comparing this to \eqref{Y_bound_claim}, we see that we need to gain an additional factor $N^{-1/2}$. How to do so is the content of this subsection.

The basic idea behind the additional factor $N^{-1/2}$ is that the expectation $\E Y$ is smaller than the typical size $\sqrt{\E \abs{Y}^2}$ of $Y$ by a factor $N^{-1/2}$. This is a rather general property of random variables which can be written, up to a negligible error term, as a polynomial of odd degree in the entries $\{\bar X_{i \beta}\}_{i = 1}^m$. A systematic representation of a large family of random variables in terms of polynomials was first given in \cite{EKY2}, and was combined with a parity argument in \cite{BEKYY}. Subsequently, an analogous parity argument for more singular functions was developed in \cite{Y1}. Following \cite{Y1}, we refer to the process of representing a random variable $Y$ as a polynomial in $\{\bar X_{i \beta}\}_{i = 1}^M$ up to a negligible error term as the \emph{polynomialization} of $Y$.

We shall develop a new approach to the polynomialization of the variables \eqref{terms of degree three}. The main reason is that these variables have a complicated algebraic structure, which needs to be combined with the Helffer-Sj\"ostrand representation \eqref{defy428}. These difficulties lead us to define a family of graded polynomials (given in Definitions \ref{def:O_prec1}--\ref{def:O_prec3}), which is general enough to cover the polynomialization of all terms from \eqref{terms of degree three} and imposes conditions on the coefficients that ensure the gain of $N^{-1/2}$. The basic structure behind these polynomials is a classification based on the $\ell^2$- and $\ell^3$-norms of their coefficients.

Let us outline the rough idea of the parity argument. We use the notations $\bar X = \bar X_{[\beta]} + \bar X^{[\beta]}$ and $T^{[\beta]}(z) \deq (\bar X^{[\beta]} (\bar X^{[\beta]})^* - z)^{-1}$, in analogy to those introduced before \eqref{Neumann}. A simple example of a polynomial is
\begin{equation*}
\cal P_2 \;=\; (\bar X^* T^{[\beta]} \bar X)_{\beta \beta} \;=\; \sum_{i,j} T^{[\beta]}_{ij} \bar X_{i \beta} \bar X_{j \beta}\,.
\end{equation*}
This is a polynomial of degree two. Note that the coefficients $T^{[\beta]}_{ij}$ are $\bar X^{[\beta]}$-measurable, i.e.\ independent of $\bar X_{[\beta]}$. It is not hard to see that $\E \cal P_2$ is of the same order as $\sqrt{\E \abs{\cal P_2}^2}$, so that taking the expectation of $\cal P_2$ does not yield better bounds. The situation changes drastically if the polynomial is \emph{odd} degree. Consider for instance the polynomial
\begin{equation*}
\cal P_3 \;\deq\; (\bar X^* T^{[\beta]} \bar X)_{\beta \beta} (T^{[\beta]} \bar X)_{\f w \beta} \;=\; \sum_{i,j,k} T^{[\beta]}_{ij} T^{[\beta]}_{\f w k} \bar X_{i \beta} \bar X_{j \beta} \bar X_{k \beta}\,.
\end{equation*}
Now we have $\abs{\E \cal P_3} \lesssim N^{-1/2} \sqrt{\E \abs{\cal P_3}^2}$. The reason for this gain of a factor $N^{-1/2}$ is clear: taking the expectation forces all three summation indices $i,j,k$ to coincide.

In the following we define a large family of $\Z_2$-graded polynomials that is sufficiently general to cover the polynomializations of the terms \eqref{terms of degree three}. We shall introduce a notation $O_{\prec, *}(A)$, which generalizes the notation $O_{\prec}(A)$ from \eqref{def:stocdom}; here $* \in \h{\txt{even}, \txt{odd}}$ denotes the parity of the polynomial, and $A$ its size. We always have the trivial bound $O_{\prec, *}(A) = O_{\prec}(A)$. In addition, we roughly have the estimates
\begin{equation*}
\E O_{\prec, \txt{even}}(A) \;\lesssim\; A\,, \qquad \E O_{\prec, \txt{odd}}(A) \;\lesssim\; N^{-1/2} A\,.
\end{equation*}
The need to gain an additional factor $N^{-1/2}$ from odd polynomials imposes nontrivial constraints on the polynomial coefficients, which are carefully stated in Definitions \ref{def:O_prec1}--\ref{def:O_prec3}; they have been tailored to the class of polynomials generated by the terms \eqref{terms of degree three}.

We now move on to the proof of Lemma \ref{lem: 3m}. We recall that we assume throughout that $\phi \geq 1$.
We first introduce a family of graded polynomials suitable for our purposes. It depends on a constant $C_0$, which we shall fix during the proof to be some large but fixed number.

\begin{definition}[Admissible weights] \label{def:weight}
Let $\varrho = (\varrho_i \col i \in \qq{1, \phi N})$ be a family of deterministic nonnegative weights. We say that $\varrho$ is an \emph{admissible weight} if
\begin{equation} \label{l2-l3_bounds}
\frac{1}{N^{1/2} \phi^{1/4}} \pbb{\sum_i \varrho_i^2}^{1/2} \;\leq\; 1\,, \qquad \frac{1}{N^{1/2} \phi^{1/4}} \pbb{\sum_i \varrho_i^3}^{1/3} \;\leq\; N^{-1/6}\,.
\end{equation}
\end{definition}
\begin{definition}[$O_{\prec, d}(\cdot)$] \label{def:O_prec1}
For a given degree $d \in \N$ let
\begin{equation} \label{cal_P_rep}
\cal P \;=\; \sum_{i_1, \dots, i_d = 1}^{\phi N} V_{i_1 \cdots i_d} \bar X_{i_1 \beta} \cdots \bar X_{i_d \beta}
\end{equation}
be a polynomial in $\bar X$. Analogously to the notation $O_\prec(\cdot)$ introduced in Definition \ref{def:stocdom}, we write $\cal P = O_{\prec, d}(A)$
if the following conditions are satisfied.
\begin{enumerate}
\item
$A$ is deterministic and $V_{i_1 \cdots i_d}$ is $\bar X^{[\beta]}$-measurable.
\item
There exist admissible weights $\varrho^{(1)}, \dots, \varrho^{(d)}$ such that
\begin{equation} \label{V_leq_rho}
\abs{V_{i_1 \cdots i_d}} \;\prec\; A \, \varrho^{(1)}_{i_1} \cdots \varrho^{(d)}_{i_d}\,.
\end{equation}
\item
We have the deterministic bound $\abs{V_{i_1 \cdots i_d}} \leq N^{C_0}$.
\end{enumerate}
The above definition extends trivially to the case $d = 0$, where $\cal P = V$ is $\bar X^{[\beta]}$-measurable.
\end{definition}

\begin{definition}[$O_{\prec, \diamond}(\cdot)$] \label{def:O_prec2}
Let $\cal P$ be a polynomial of the form
\begin{equation} \label{P_diamond}
\cal P \;=\; \sum_{i = 1}^{\phi N} V_i \pbb{\bar X_{i \beta}^2 - \frac{1}{N \phi^{1/2}}}\,.
\end{equation}
We write $\cal P \;=\; O_{\prec, \diamond}(A)$ if $V_i$ is $\bar X^{[\beta]}$-measurable, $\abs{V_i} \leq N^{C_0}$, and $\abs{V_i} \prec A$ for some deterministic $A$.
\end{definition}
\begin{definition}[Graded polynomials] \label{def:O_prec3}
We write $\cal P = O_{\prec, \txt{even}}(A)$ if $\cal P$ is a sum of at most $C_0$ terms of the form
\begin{equation*}
A \cal P_0 \prod_{s = 1}^m \cal P_i\,, \qquad \cal P_0 \;=\; O_{\prec, 2n}(1) \,, \qquad \cal P_i \;=\; O_{\prec, \diamond}(1)\,,
\end{equation*}
where $n,m \leq C_0$ and $A$ is deterministic.

Moreover, we write $\cal P = O_{\prec, \txt{odd}}(A)$ if $\cal P = \wh{\cal P}\, \cal P_{\txt{even}}$, where $\wh{\cal P} = O_{\prec, 1}(1)$ and $\cal P_{\txt{even}} = O_{\prec, \txt{even}}(A)$.
\end{definition}

Definitions \ref{def:O_prec1}--\ref{def:O_prec3} refine Definition \ref{def:stocdom} in the sense that
\begin{equation} \label{LDE for graded}
\cal P \;=\; O_{\prec, d}(A) \quad \text{or} \quad \cal P \;=\; O_{\prec, \diamond}(A) \qquad \Longrightarrow \qquad \cal P \;=\; P_\prec(A)\,.
\end{equation}
Indeed, let $\cal P = O_{\prec, d}(A)$ be of the form \eqref{cal_P_rep}. Then a simple large deviation estimate (e.g.\ a trivial extension of \cite[Theorem B.1 (iii)]{EKYY3}) yields
\begin{equation*}
\abs{\cal P} \;\prec\; \pbb{\pb{N \phi^{1/2}}^{-d} \sum_{i_1, \dots, i_d} \abs{V_{i_1 \cdots i_d}}^2}^{1/2} \;\prec\; A\,,
\end{equation*}
where the last step follows from the definition of admissible weights. Similarly, if $\cal P = O_{\prec, \diamond}(A)$ is of the form \eqref{P_diamond}, a large deviation estimate (e.g.\cite[Theorem B.1 (i)]{EKYY3}) yields
\begin{equation*}
\abs{\cal P} \;\prec\; \pbb{N^{-2} \phi^{-1} \sum_i \abs{V_i}^2}^{1/2} \;\prec\; N^{-1/2} A \;\leq\; A\,.
\end{equation*}

Note that terms of the form $O_{\prec, \f \cdot}(A)$ satisfy simple algebraic rules. For instance, we have
\begin{equation*}
O_{\prec, \txt{even}}(A_1) + O_{\prec, \txt{even}}(A_2) \;=\; O_{\prec, \txt{even}}(A_1 + A_2)\,,
\end{equation*}
and
\begin{equation*}
O_{\prec, \txt{odd}}(A_1) \, O_{\prec, \txt{even}}(A_2) \;=\; O_{\prec, \txt{odd}}(A_1 A_2)
\end{equation*}
after possibly increasing $C_0$.
(As with the standard big O notation, such expressions are to be read from left to right.)  We stress that such operations may be performed an arbitrary, but bounded, number of times. It is a triviality that all of the following arguments will involve at most $C_0$ such algebraic operations on graded polynomials, for large enough $C_0$.

The point of the graded polynomials is that bounds of the form \eqref{LDE for graded} are improved if $d$ is odd and we take the expectation. The precise statement is the following.

\begin{lemma} \label{lem: gain of N1/2}
Let $\cal P = O_{\prec, \txt{odd}}(A)$ for some deterministic $A \leq N^C$. Then for any fixed $D > 0$ we have
\begin{equation*}
\abs{\E \cal P} \;\prec\; N^{-1/2} A + N^{-D}\,.
\end{equation*}
\end{lemma}
\begin{proof}
It suffices to set $A = 1$ and consider $\cal P = \wh {\cal P} \cal P_0 \prod_{s = 1}^{m}\cal P_i$, where $\wh {\cal P}$, $\cal P_0$, and $\cal P_i$ are as in Definition \ref{def:O_prec3}. By linearity, it suffices to consider
\begin{equation*}
\cal P \;=\; \sum_{i_0} W_{i_0} \bar X_{i_0 \beta} \sum_{i_1, \dots, i_d} V_{i_1 \cdots i_d} \bar X_{i_1 \beta} \cdots \bar X_{i_d \beta} \prod_{l = d+1}^{d+m} \pBB{\sum_{i_l} V^{(l)}_{i_l} \pbb{\bar X_{i_l \beta}^2 - \frac{1}{N \phi^{1/2}}}}\,,
\end{equation*}
where $d = 2n$ is even. We suppose that $\abs{W_{i_0}} \prec \varrho^{(0)}_{i_0}$, $\abs{V_{i_1 \cdots i_d}} \prec \varrho^{(1)}_{i_1} \cdots \varrho^{(d)}_{i_d}$, and $\abs{V^{(d+l)}_{i_{d+l}}} \prec 1$ for $l = d+1, \dots, d+m$. Here $\varrho^{(k)}_{i_k}$ denotes an admissible weight (see Definition \ref{def:weight}).
Thus we have
\begin{equation*}
\abs{\E \cal P} \;\prec\; \sum_{i_0, \dots, i_{d+m}} \varrho_{i_0}^{(0)} \cdots \varrho_{i_d}^{(d)} \absBB{\E \pBB{\bar X_{i_0 \beta} \cdots \bar X_{i_d \beta} \prod_{l = d+1}^{d+m} \pbb{\bar X_{i_l \beta}^2 - \frac{1}{N \phi^{1/2}}}}} + N^{-D}\,,
\end{equation*}
where the term $N^{-D}$ comes from the trivial deterministic bound $\abs{V_{i_1 \cdots i_d}} \leq N^C$ on the low-probability event of $\prec$ (i.e.\ the event inside $\P [\,\cdot\,]$ in \eqref{prec_def}) in \eqref{V_leq_rho}, and analogous bounds for the other $\bar X^{[\beta]}$-measurable coefficients.

The expectation imposes that each summation index $i_0, \dots, i_{d+m}$ coincide with at least one other summation index. Thus we get
\begin{equation} \label{E PP estimate}
\abs{\E \cal P} \;\prec\; \sum_{i_0, \dots, i_{d+m}} \tilde \varrho^{(0)}_{i_0} \cdots \tilde \varrho^{(d)}_{i_d} I(i_0, \dots, i_{d+m}) \frac{1}{(N \phi^{1/2})^m} + N^{-D}\,,
\end{equation}
where the indicator function $I(\cdot)$ imposes the condition that each summation index must coincide with at least another one, and we introduced the weight $\tilde \varrho_i^{(k)} \deq N^{-1/2} \phi^{-1/4} \varrho_i$. Note that
\begin{equation} \label{tho tilde bounds}
\sum_i \tilde \varrho^{(k)}_i \;\leq\; N^{1/2} \phi^{1/2} \,, \qquad \sum_i \pb{\tilde \varrho^{(k)}_i}^2 \;\leq\; 1 \,, \qquad \sum_i \pb{\tilde \varrho_i^{(k)}}^q \;\leq\; N^{-q/6} \quad (q \geq 3)\,.
\end{equation}
Here for $q > 3$ we used the inequality $\norm{\tilde \varrho^{(k)}}_{\ell^q} \leq \norm{\tilde \varrho^{(k)}}_{\ell^p}$ for $q \geq p$.
The indicator function $I$ on the right-hand side of \eqref{E PP estimate} imposes a reduction in the number of independent summation indices. We may write $I = \sum_{P} I_P$ as a sum over all partitions $P$ of the set $\qq{0, d+m}$ with blocks of size at least two, whereby
\begin{equation*}
I_P(i_1, \dots, i_{d+m}) \;=\; \prod_{p \in P} \ind{i_k = i_l \txt{ for all } k,l \in p}\,.
\end{equation*}
Hence the summation over $i_1, \dots, i_{d+m}$ factors into a product over the blocks of $P$. We shall show that the contribution of each block is at most one, and that there is a block whose contribution is at most $N^{-1/2}$.

Fix $p \in P$ and denote by $S_p$ the contribution of the block $p$ to the summation in the main term of \eqref{E PP estimate}. Define $s \deq \abs{p \cap \qq{0,d}}$ and $t \deq \abs{p \cap \qq{d+1, d+m}}$. By definition of $P$, we have $s + t \geq 2$. By the inequality of arithmetic and geometric means, we have
\begin{equation*}
S_p \;\leq\; \max_k \sum_{i} \pb{\tilde \varrho_i^{(k)}}^s \frac{1}{(N \phi^{1/2})^t}\,.
\end{equation*}
Using \eqref{tho tilde bounds} it is easy to conclude that
\begin{equation*}
S_p \;\leq\;
\begin{cases}
1 & \text{if } (s,t) = (2,0)
\\
N^{-1/2} & \text{if } (s,t) \neq (2,0)\,.
\end{cases}
\end{equation*}
Moreover, since $d$ is even, at least one block of $P$ satisfies $(s,t) \neq (2,0)$.

Thus we find that
\begin{equation*}
\abs{\E \cal P} \;\prec\; \sum_{P} \prod_{p \in P} S_p + N^{-D} \;\leq\; C_{d+m} N^{-1/2} + N^{-D}\,.
\end{equation*}
Since $d+m \leq 2 C_0$, the proof is complete.
\end{proof}

In order to apply Lemma \ref{lem: gain of N1/2} to the terms $Y$ from \eqref{terms of degree three}, we need to expand $Y$ in terms of graded polynomials. This expansion is summarized in the following result, which gives the polynomializations of the coefficients of the terms from \eqref{terms of degree three}. For an arbitrary unit vector $\f v \in \R^M$ we define the control parameter
\begin{equation*}
\Psi^{\f v} \;\deq\; \Psi + (N^{-1} \norm{\f v}_{\infty})^{1/3} \,, \qquad \norm{\f v}_{\infty} \;\deq\; \max_i \abs{v(i)}\,.
\end{equation*}

\begin{lemma} \label{lem: polynom}
Fix $D > 0$. Then there exists $C_0 = C_0(D)$ such that for any unit vector $\f v \in \R^M$ we have
\begin{align}
T_{\f v \f v} &\;=\; T_{\f v \f v}^{[\beta]} + O_{\prec, \txt{even}}(\phi^{-1} (\Psi^{\f v})^2) + O_\prec(N^{-D})\,, \label{TTbeta_pol}
\\
T_{bb} &\;=\; O_{\prec, \txt{even}}(\phi^{-1/2}) + O_\prec(N^{-D})\,, \label{Tbb_pol}
\\
T_{\f w b} &\;=\; O_{\prec, \txt{even}}(\phi^{-1} \Psi_b) + O_\prec(N^{-D})\,, \label{Twb_pol}
\\
(T \bar X)_{\f v \beta} &\;=\; O_{\prec, \txt{odd}}(\phi^{-1/4} \Psi^{\f v}) + O_\prec(N^{-D})\,, \label{TQ_pol}
\\
(\bar X^* T \bar X)_{\beta \beta} &\;=\; O_{\prec, \txt{even}}(\phi^{1/2}) + O_\prec(N^{-D})\,, \label{QTQ_pol}
\end{align}
uniformly for $z \in \f S$.
\end{lemma}

\begin{proof}
We begin by noting that \eqref{bound: Rij isotropic gen} applied to $X^{[\mu]}$ and \eqref{im_m_swap} combined with a large deviation estimate (see \cite[Theorem B.1]{EKYY3}) yields
\begin{equation*}
(X^* G^{[\beta]} X)_{\beta \beta} \;=\; \phi^{1/2} m_{\phi^{-1}} + O_\prec(\phi^{-1/2} \Psi)\,.
\end{equation*}
Using \eqref{RS-N} and Lemma \ref{princ}, it is not hard to deduce that
\begin{equation*}
(\bar X^* T^{[\beta]} \bar X)_{\beta \beta} \;=\; \phi^{1/2} m_{\phi^{-1}} + O_\prec(\phi^{-1/2} \Psi)\,.
\end{equation*}
Thus for any fixed $n$ we may expand
\begin{align*}
-\frac{1}{1 + (\bar X^* T^{[\beta]} \bar X)_{\beta \beta}} &\;=\; - \frac{1}{1 + \phi^{1/2} m_{\phi^{-1}} - \pb{\phi^{1/2} m_{\phi^{-1}} - (\bar X^* T^{[\beta]} \bar X)_{\beta \beta}}}
\\
&\;=\; -\sum_{k = 0}^n (z m_\phi)^{k + 1} \pb{\phi^{1/2} m_{\phi^{-1}} - (\bar X^* T^{[\beta]} \bar X)_{\beta \beta}}^k + O_\prec(\phi^{1/2} \Psi^{n+1})\,,
\end{align*}
where in the second step we used \eqref{identity for m MP} and \eqref{mphi_minvphi}. Now we split
\begin{align*}
(\bar X^* T^{[\beta]} \bar X)_{\beta \beta} - \phi^{1/2} m_{\phi^{-1}} &\;=\; \sum_{i \neq j} T^{[\beta]}_{ij} \bar X_{i \beta} \bar X_{j \beta} + \sum_i T^{[\beta]}_{ii} \pbb{\bar X_{i \beta} - \frac{1}{N \phi^{1/2}}} + \sum_i \frac{1}{N \phi^{1/2}}  \pb{T_{ii}^{[\beta]} - m_{\phi^{-1}}}
\\
&\;=\; O_{\prec,2}(\phi^{-1/2} \Psi) + O_{\prec, \diamond}(\phi^{-1/2}) + O_{\prec, 0}(\phi^{-1/2} \Psi)
\\
&\;=\; O_{\prec, \txt{even}}(\phi^{-1/2})\,,
\end{align*}
where in the second step we used the estimates $\abs{T_{ij}^{[\beta]} - \delta_{ij} m_{\phi^{-1}}} \prec \phi^{-1} \Psi$  and $\abs{m_{\phi^{-1}}} \leq C \phi^{-1/2}$. Since $\abs{z m_\phi} \leq C \phi^{1/2}$, we therefore conclude that
\begin{equation*}
-\frac{1}{1 + (\bar X^* T^{[\beta]} \bar X)_{\beta \beta}} \;=\; O_{\prec, \txt{even}}(n \phi^{1/2}) + O_\prec(\phi^{1/2} \Psi^{n+1})\,.
\end{equation*}
From \eqref{def_Psi} and the definition of $\eta$, we readily find that $\Psi \leq N^{-c \tau}$ for some constant $c$. Therefore choosing $n \equiv n(\tau,D)$ large enough yields
\begin{equation} \label{R_polynom}
-\frac{1}{1 + (\bar X^* T^{[\beta]} \bar X)_{\beta \beta}} \;=\; O_{\prec, \txt{even}}(\phi^{1/2}) + O_\prec(\phi^{1/2} N^{-D})\,.
\end{equation}

Having established \eqref{R_polynom}, the remainder of the proof is relatively straightforward. From \eqref{R_mumu} and \eqref{G_X} we get
\begin{equation*}
(T \bar X)_{\f v \beta} \;=\; \frac{1}{1 + (\bar X^* T^{[\beta]} \bar X)_{\beta \beta}} (T^{[\beta]} \bar X)_{\f v \beta}\,.
\end{equation*}
Moreover, using $\Psi \geq c N^{-1/2}$ and $T_{\f v i}^{[\beta]} = v(i) m_{\phi^{-1}} + O_\prec(\phi^{-1} \Psi) = O_\prec(\phi^{-1/2} \abs{v(i)} + \phi^{-1} \Psi)$, we find
\begin{equation*}
\frac{1}{N \phi^{1/2}} \sum_i \absb{T_{\f v i}^{[\beta]}}^2 \;\prec\; \phi^{-3/2} \Psi^2\,, \qquad
\frac{1}{N^{3/2} \phi^{3/4}} \sum_i \absb{T_{\f v i}^{[\beta]}}^3 \;\prec\; N^{-1/2} \phi^{-9/4} (\Psi^{\f v})^3\,.
\end{equation*}
We conclude that
\begin{equation}  \label{TbetaQ}
(T^{[\beta]} \bar X)_{\f v \beta} \;=\; \sum_{i} T_{\f v i}^{[\beta]} \bar X_{i \beta} \;=\; O_{\prec, 1}(\phi^{-3/4} \Psi^{\f v})\,.
\end{equation}
Now \eqref{TQ_pol} follows easily from \eqref{TbetaQ} and \eqref{R_polynom}.

Moreover, \eqref{TTbeta_pol} and \eqref{Twb_pol} follow from  \eqref{Neumann} combined with \eqref{TbetaQ} and \eqref{R_polynom}. For \eqref{Twb_pol} we estimate the second term in \eqref{Neumann} by
\begin{equation*}
O_{\prec, \txt{even}}\pb{\phi^{1/2} \phi^{-3/2} \Psi^{\f w} \Psi^{\f e_b}} \;=\; O_{\prec, \txt{even}}\pb{\phi^{-1} (\Psi + N^{-1/3})^2} \;=\; O_{\prec, \txt{even}}\pb{\phi^{-1} \Psi}\,,
\end{equation*}
where in the last step we used that $\Psi \geq c N^{-1/2}$.
Moreover, \eqref{Tbb_pol} is a trivial consequence of \eqref{TTbeta_pol}. Finally, \eqref{QTQ_pol} follows from \eqref{Neumann} and \eqref{R_polynom} combined with
\begin{equation*}
(\bar X^* T^{[\beta]} \bar X)_{\beta \beta} \;=\; O_{\prec, 2}(1)\,.
\end{equation*}
This concludes the proof.
\end{proof}

Note that the upper bounds in Lemma \ref{lem: polynom} are the same as those of \eqref{uamza}, except that $\Psi$ is replaced with the larger quantity $\Psi^{\f v}$. In order to get back to $\Psi$ from $\Psi^{\f v}$, we use the following trivial result.
\begin{lemma} \label{lem: recover Psi}
We have
\begin{equation*}
\Psi^{\f v} \;\prec\; \Psi
\end{equation*}
if 
\begin{equation} \label{Psi_geq}
\Psi \;\geq\; N^{-1/3} \qquad \text{or} \qquad \norm{\f v}_{\infty} \;\prec\; N^{-1/2}\,.
\end{equation}
\end{lemma}
\begin{proof}
The claim follows immediately from the upper bound $\Psi \geq N^{-1/2}$, valid for all $z \in \f S$.
\end{proof}

In each application of Lemma \ref{lem: polynom}, we shall verify one of the conditions of \eqref{Psi_geq}. The first condition is verified for $\eta \leq N^{-2/3}$, which always holds for the coefficients of $x_1$, $x_2$, and $x_3$ (recall \eqref{defEI}).

The second condition of \eqref{Psi_geq} will be verified when computing the coefficients of $y_1$, $y_2$, and $y_3$. To that end, we make use of the freedom of the choice of basis when computing the trace in the definition of $J_1$, $J_2$, and $J_3$. We shall choose a basis that is completely delocalized. The following simple result guarantees the existence of such a basis.
\begin{lemma} \label{lem: deloc_basis}
There exists an orthonormal basis $\f w_1, \dots, \f w_M$ of $\R^M$ satisfying
\begin{equation} \label{deloc_basis}
\abs{w_i(j)} \;\prec\; M^{-1/2}
\end{equation}
uniformly in $i$ and $j$.
\end{lemma}
\begin{proof}
Let the matrix $[\f w_1 \cdots \f w_M]$ of orthonormal basis vectors be uniformly distributed on the orthogonal group $\r O(M)$. Then each $\f w_i$ is uniformly distributed on the unit sphere, and by standard Gaussian concentration arguments one finds that $\abs{w_i(j)} \prec M^{-1/2}$. In particular, there exists an orthonormal basis $\f w_1, \dots, \f w_M$ satisfying \eqref{deloc_basis}. In fact, a slightly more careful analysis shows that one can choose $\abs{w_i(j)} \leq (2 + \epsilon) (\log M)^{1/2} M^{-1/2}$ for any fixed $\epsilon > 0$ and large enough $M$.
\end{proof}

We may now derive estimates on the matrix $T^2$ by writing $(T^2)_{jk} = \sum_{i}T_{j \f w_i} T_{\f w_i k}$, where $\{\f w_i\}$ is a basis satisfying \eqref{deloc_basis}. From Lemmas \ref{lem: polynom} and \ref{lem: recover Psi} we get the following result.

\begin{lemma} \label{lem: polynom2}
Fix $D > 0$. Then there exists $C_0 = C_0(D)$ such that
\begin{align}
\tr T &\;=\; \tr T^{[\beta]} + O_{\prec, \txt{even}}(\phi^{-1} N \Psi^2) + O_\prec(N^{-D})\,, \label{TTbeta_pol2} 
\\
(T^2)_{bb} &\;=\; O_{\prec, \txt{even}}(N \phi^{-1} \Psi^2) + O_\prec(N^{-D})\,, \label{Tbb_pol2}
\\
(T^2 \bar X)_{b \beta} &\;=\; O_{\prec, \txt{odd}}(\phi^{-1/4} N \Psi^2) + O_\prec(N^{-D})\,, \label{TQ_pol2}
\\
(\bar X^* T^2 \bar X)_{\beta \beta} &\;=\; O_{\prec, \txt{even}}(\phi^{1/2} N \Psi^2) + O_\prec(N^{-D})\,, \label{QTQ_pol2}
\end{align}
uniformly for $z \in \f S$.
\end{lemma}

\begin{proof}
We prove \eqref{TQ_pol2}; the other estimates are proved similarly. We choose a basis $\f w_1, \dots, \f w_M$ as in Lemma \ref{lem: deloc_basis}, and write
\begin{equation*}
(T^2 \bar X)_{b \beta} \;=\; \sum_i T_{b \f w_i} (T\bar X)_{\f w_i \beta} \;=\; \sum_{i = 1}^{N\phi} O_{\prec, \txt{even}} \pb{\phi^{-1} \Psi + \phi^{-1/2} \abs{w_i(b)}} \, O_{\prec, \txt{odd}}(\phi^{-1/4} \Psi) + O_{\prec}(N^{-D})\,,
\end{equation*}
where we used \eqref{Twb_pol} with $\f w$ replaced by $\f w_i$, \eqref{TQ_pol}, and Lemma \ref{lem: recover Psi}. Summing over $i$, and recalling that $\Psi \geq N^{-1/2}$, it is easy to conclude \eqref{TQ_pol2}.
\end{proof}

In particular, as in \eqref{lghj2_est} we find
\begin{equation} \label{J_bounds_pol}
J_1,J_3 \;=\; O_{\prec, \txt{odd}}(\phi^{-1/4}N\Psi^2) + O_\prec(N^{-D})\,, \qquad J_2 \;=\; O_{\prec, \txt{even}} (N  \Psi^2) + O_{\prec}(N^{-D})\,,
\end{equation}
where the parity of $J_i$ follows easily from its definition.

The estimates from Lemma \ref{lem: polynom} are compatible with integration in the following sense. Suppose that $\cal P(s)$ depends on a parameter $s \in S$, where $S \subset \R^k$  has bounded volume, and that $\cal P(s) = O_{\prec, *}(A(s)) + O_\prec(N^{-D})$ uniformly in $s \in S$, where $A(s)$ is a deterministic function of $s$ and $* \in \{\txt{even}, \txt{odd}\}$ denotes the parity of $\cal P$.  Suppose in addition that $\cal P(s)$ is Lipschitz continuous with Lipschitz constant $N^C$. Then, analogously to Remark \ref{rem:all_z}, we have
\begin{equation*}
\int_S \cal P(s) \, \dd s \;=\; O_{\prec, *}\pbb{\int_S A(s) \, \dd s} + O_\prec\pbb{\int_S A(s) \, \dd s \, N^{-D}}\,.
\end{equation*}

Lemmas \ref{lem: polynom} and \ref{lem: polynom2} are the key estimates of the coefficients appearing in \eqref{terms of degree three}. We claim that all estimates of Lemma \ref{zddz}, along with \eqref{uamza}, remain valid, in the sense that an estimate of the form $\abs{u} \prec v$ is to be replaced with
\begin{equation*}
u \;=\; O_{\prec, *}(v) + O_\prec(N^{-D})\,,
\end{equation*}
where $* \in \{\txt{even}, \txt{odd}\}$ denotes the parity of polynomialization of $u$. Indeed, for the estimates \eqref{bzdbb} on $x_i$, we always have $\im z = \eta \leq N^{-2/3}$, so that by Lemma \ref{lem: recover Psi} we have $\Psi^{\f v} \prec \Psi$. Thus we get from Lemma \ref{lem: polynom} that
\begin{equation*}
x_1, x_3 \;=\; O_{\prec, \txt{odd}}\pb{\phi^{-1/4}N \Psi\Psi_b} + O_\prec(N^{-D})\,,
\qquad
x_2 \;=\; O_{\prec, \txt{even}} \pb{\phi^{-1/2}N  \Psi_b^2+N\Psi^2} + O_\prec(N^{-D})\,,
\end{equation*}
where the parity of $x_i$ may be easily deduced from their definitions. Moreover, for the estimates \eqref{bzdbby} we use \eqref{J_bounds_pol} to get
\begin{equation*}
y_1,y_3 \;=\; O_{\prec, \txt{odd}} \pb{\phi^{-1/4}N^{C\e}\kappa_a^{1/2}} + O_\prec(N^{-D})\,, \qquad
y_2  \;=\;  O_{\prec, \txt{even}} \pb{N^{C\e}\kappa_a^{1/2}} + O_\prec(N^{-D})\,.
\end{equation*}
Note that, thanks to Lemmas \ref{lem: polynom} and \ref{lem: polynom2}, we have obtained exactly the same upper bounds on the coefficients $x_i$ and $y_i$ as the ones obtained in Lemma \ref{zddz}, but we have in addition expressed them, up to a negligible error, as graded polynomials, to which Lemma \ref{lem: gain of N1/2} is applicable.

In addition to the coefficients $x_i$ and $y_i$, we have to control the coefficient $q^{(m)}(y_T)$ in the definition \eqref{def_Al} of $A_{\f l}$. We in fact claim that
\begin{equation}
q^{(m)}(y_T) \;=\; O_{\prec, \txt{even}}(N^{C \epsilon}) + O_\prec(N^{-D})\,.
\end{equation}
This follows from the estimate
\begin{equation*}
y_{T} \;=\; y_{T^{[\beta]}} + O_{\prec, \txt{even}}(N^{C \epsilon} \kappa_a) + O_\prec(N^{-D}) \;=\; O_{\prec, \txt{even}}(N^{C \epsilon}) + O_\prec(N^{-D})\,,
\end{equation*}
which may be derived from \eqref{TTbeta_pol2}, combined with a Taylor expansion of $q^{(m)}$. Similarly, we find that
\begin{equation}
h^{(k)}(A_{\f 0}) \;=\; O_{\prec, \txt{even}}(N^{C \epsilon}) + O_\prec(N^{-D}) \qquad (k = 1,2,3)\,.
\end{equation}

We may now put everything together. Noting that the degree of the polynomializations of the expressions \eqref{terms of degree three} is always odd, we obtain, in analogy to \eqref{E_bound_naive} that
\begin{equation*}
Y \;=\; O_{\prec, \txt{odd}} \pB{N^{C \epsilon} \pb{\kappa_a^{-1/2} \phi^{3/4} \abs{w(b)}^2 + \phi^{1/4} \abs{w(b)} + \phi^{-1/4} \kappa_a^{1/2}}} + O_\prec(N^{-D})
\end{equation*}
for $Y$ being any term of \eqref{terms of degree three}. Hence Lemma \ref{lem: 3m} follows from Lemma \ref{lem: gain of N1/2} and Young's inequality.

\subsection{Stability of level repulsion: proof of Lemma \ref{lemma:lr2}} \label{sec:pf_lrcomparison}
This is a Green function comparison argument, using the machinery introduced in Section \ref{sec: 81}. A similar comparison argument was given in Propositions 2.4 and 2.5 of \cite{KY1}. The details in the sample covariance case and for indices $a$ satisfying $a \leq K^{1 - \tau}$ follow an argument very similar to (in fact simpler than) the one from Sections \ref{sec: 81}--\ref{sec: 83}. As in the proofs of Propositions 2.4 and 2.5 of \cite{KY1}, one writes the level repulsion condition in terms of resolvents. In our case, one uses the representation \eqref{HS split 1} as the starting point. Then the machinery of Sections \ref{sec: 81}--\ref{sec: 83} may be applied with minor modifications. We omit the details.

\section{Extension to general $T$ and universality for the uncorrelated case} \label{sec:general_model}

In this section we relax the assumption \eqref{simpl_assumption}, and hence extend all arguments of Sections \ref{sec:prelim}--\ref{sec: QUE} to cover general $T$. We also prove the fixed-index joint eigenvector-eigenvalue universality of the matrix $H$ defined in \eqref{def_H}, for indices bounded by $K^{1 - \tau}$ for some $\tau > 0$.

Bearing the applications in the current paper in mind, we state the results of this section for the matrix $H$ from \eqref{def_H}, but it is a triviality that all results and their proofs carry over to case of arbitrary $Q$ from \eqref{wtH_cov} provided that $\Sigma = T T^* = I_M$.

\subsection{The isotropic Marchenko-Pastur law for $Y Y^*$}

We start with the singular value decomposition of $T$, which we write as
\begin{equation*}
T \;=\; O' (\Lambda, 0) O'' \;=\; O' \Lambda (I_M, 0) O''\,,
\end{equation*}
where $O' \in \r O(M)$ and $O'' \in \r O(M+r)$ are orthogonal matrices, $0$ is the $M \times r$ zero matrix, and $\Lambda$ is an $M \times M$ diagonal matrix containing the singular values of $T$. Setting
\begin{equation} \label{O_dep_T}
\Sigma^{1/2} \;=\; O' \Lambda (O')^* \,, \qquad O \;\deq\;
\begin{pmatrix}
O' & 0
\\
0 & I_r
\end{pmatrix}
O''\,,
\end{equation}
we have
\begin{equation*}
T \;=\; \Sigma^{1/2} (I_M, 0) O\,.
\end{equation*}
We conclude that
\begin{equation*}
Q \;=\; \Sigma^{1/2} H \Sigma^{1/2} \,,
\end{equation*}
where $H \deq Y Y^*$ and $Y \deq (I_M, 0) O X$ were defined in \eqref{def_H}. Comparing this to \eqref{Q_simplified},  we find that to relax the assumption \eqref{simpl_assumption} we have to generalize the arguments of Sections \ref{sec:prelim}--\ref{sec: QUE} by replacing $X X^*$ with $H = Y Y^*$.

The generalization of $G = (X X^* - z)^{-1}$ is the resolvent of $Y Y^*$,
\begin{equation*}
\wh G(z) \;\deq\; (Y Y^* - z)^{-1}\,.
\end{equation*}
We also abbreviate
\begin{equation*}
G' \;\deq\; (O X X^* O^* - z)^{-1}\,.
\end{equation*}
Throughout the following we identify $\f w \in \R^M$ with its natural embedding $\binom{\f w}{0} \in \R^{M+r}$. Thus, for example, for $\f v, \f w \in \R^M$ we may write $G'_{\f v \f w}$.

\begin{theorem}[Local laws for $Y Y^*$] \label{IMP for YY}
Theorem \ref{thm: IMP gen} remains valid with $G$ replaced by $\wh G$. Moreover, Theorems \ref{thm: I deloc gen} and \ref{thm: cov-rig} remain valid for $\f \zeta_i$ and $\lambda_i$ denoting the eigenvectors and eigenvalues of $Y Y^*$.
\end{theorem}

\begin{proof}
It suffices to prove the first sentence, since all claims in the second sentence follow from the isotropic law (see \cite{BEKYY} for more details). We only prove \eqref{bound: Rij isotropic gen} for $\wh G$; the other bound, \eqref{bound: Rij isotropic outside sc gen} for $\wh G$, is proved similarly. To simplify the presentation, we suppose that $r = 1$; the case $r \geq 2$ is a trivial extension. Abbreviate $\bar M \deq M + 1$. Noting that $Y_{i \mu} = \ind{i \neq \bar M} (OX)_{i \mu}$, we find from \cite[Definition 3.5 and Equation (3.7)]{BEKYY} that
\begin{equation} \label{wh_G_formula}
\wh G_{\f v \f w} \;=\; G'_{\f v \f w} - \frac{G'_{\f v \bar M} G'_{\bar M \f w}}{G'_{\bar M \bar M}}\,.
\end{equation}

For definiteness, we focus on \eqref{bound: Rij isotropic gen} for $\wh G$; the proof of \eqref{bound: Rij isotropic outside sc gen} for $\wh G$ is similar.
Since $G' = O G O^*$, we have $G'_{\f v \f w} = G_{O^* \f v \, O^* \f w}$. Hence, using \eqref{bound: Rij isotropic gen} and \eqref{wh_G_formula}, the proof will be complete provided we can show that
\begin{equation} \label{def_Phi}
\absbb{\frac{G'_{\f v \bar M} G'_{\bar M \f w}}{G'_{\bar M \bar M}}} \;\prec\; \Phi\,, \qquad  \Phi \;\deq\; \sqrt{\frac{\im m_{\phi^{-1}}(z)}{M \eta}} + \frac{1}{M \eta} \;\asymp\; \phi^{-1} \Psi\,,
\end{equation}
where we recall the definition \eqref{def_Psi} of $\Psi$.
In fact, from Lemma \ref{lemma: w} and \eqref{mphi_minvphi} we find that $\Phi / \abs{m_{\phi^{-1}}} \leq N^{-c}$ for some positive constant $c$ depending on $\tau$. Hence \eqref{bound: Rij isotropic gen} yields
\begin{equation*}
\absbb{\frac{G'_{\f v \bar M} G'_{\bar M \f w}}{G'_{\bar M \bar M}}} \;\prec\; \Phi^2 / \abs{m_{\phi^{-1}}} \;\leq\; \Phi\,.
\end{equation*}
This concludes the proof.
\end{proof}

Having established Theorem \ref{IMP for YY}, all arguments from Sections \ref{sec:prelim}--\ref{sec:bulk} that use it as input may be taken over verbatim, after replacing $G$ by $\wh G$. More precisely, all results from Sections \ref{sec:prelim}--\ref{sec:bulk} remain valid for a general $Q$, with the exception of Proposition \ref{prop: level repulsion}, Lemmas \ref{lemma:lr1} and \ref{lemma:lr2}, and Proposition \ref{prop:xie}. Therefore we have completed the proofs of all of our main results except Theorem \ref{thm: eigenvector law}.

In order to prove Theorem \ref{thm: eigenvector law}, we still have to prove Lemmas \ref{lemma:lr1} and \ref{lemma:lr2} and Proposition \ref{prop:xie} for $Y Y^*$ instead of $X X^*$. Lemma \ref{lemma:lr1} is easy: for Gaussian $X$ we have $Y \eqdist (I_M, 0) X = \wt X$, where $\wt X$ is the $M \times N$ matrix obtained from $X$ by deleting its bottom $r$ rows.

The proofs of Lemma \ref{lemma:lr2} and Proposition \ref{prop:xie} rely on Green function comparison. What remains, therefore, is to extend the argument of Section \ref{sec: QUE} from $H = X X^*$ to $H = Y Y^*$.

\subsection{Quantum unique ergodicity for $YY^*$} \label{sec: QUE Y}
In this subsection we prove Proposition \ref{prop:xie} for the eigenvectors $\f \zeta_a$ of $H = Y Y^*$. As explained in Section \ref{sec:pf_lrcomparison}, the proof of Lemma \ref{lemma:lr2} is analogous and therefore omitted. We proceed exactly as in Section \ref{sec: QUE}, replacing $G$ with $\wh G$. It suffices to prove the following result.

\begin{lemma} \label{lem: QUE Y}
Lemma \ref{CPZX} remains valid if $x(E)$ and $y(E)$ are replaced with $\wh x(E)$ and $\wh y(E)$, obtained from the definitions \eqref{defxE} and \eqref{defy428} by replacing $G$ with $\wh G$.
\end{lemma}
\begin{proof}
We take over the notation from the proof of Theorem \ref{IMP for YY}, and to simplify notation again assume that $r = 1$. As in Section \ref{sec: QUE}, we suppose for definiteness that $\phi \geq 1$. Defining $\f u \deq O \f w$ and $\f r \deq O \f e_{M+1}$, we have $\scalar{\f u}{\f r} = 0$ and, using \eqref{wh_G_formula},
\begin{equation*}
\wh G_{\f w \f w}\;=\; G_{\f u \f u} - \frac{G_{\f u \f r} G_{\f r \f u}}{G_{\f r \f r}}\,, \qquad
\tr \wh G \;=\; \tr G - \frac{(G^2)_{\f r \f r}}{G_{\f r \f r}}\,.
\end{equation*}
We conclude that
\begin{equation*}
\wh x(E) \;=\; x(E) - \frac{M}{\pi} \im \pbb{\frac{G_{\f u \f r} G_{\f r \f u}}{G_{\f r \f r}}}(E + \ii \eta)\,.
\end{equation*}
Recalling \eqref{bound: Rij isotropic gen} and \eqref{def_Psi}, we find that the second term is stochastically dominated by
\begin{equation*}
M \frac{\phi^{-2} \Psi^2}{\phi^{-1/2}} \;\leq\; N \Psi^2 \;\leq\; C N \frac{1}{N^2 \eta^2} \;=\;  \frac{C N^{4 \epsilon}}{N \Delta_a} \, \frac{1}{\Delta_a} \;\leq\; C N^{4 \epsilon} \kappa_a^{1/2} \frac{1}{\Delta_a}\,,
\end{equation*}
where in the second step we used that $\Psi \leq C (N \eta)^{-1}$, as follows from Lemma \ref{lemma: w} and the definition of $\eta$ in \eqref{defEI}.
Recalling the definitions from \eqref{defEI}, we therefore conclude that for small enough $\epsilon \equiv \epsilon(\tau)$ we have
\begin{equation} \label{x - x_hat}
\abs{I} \sup_{E \in I} \absb{\wh x(E) - x(E)} \;\prec\; N^{-c}
\end{equation}
for some positive constant $c$ depending on $\tau$.

Similarly, we have for any $z \in \f S$
\begin{equation} \label{trG-trGhat}
\absb{\tr G - \tr \wh G} \;\prec\; N \Psi^2 \;\leq\; N^{-c} \eta^{-1}
\end{equation}
for some positive constant $c$ depending on $\tau$. Plugging \eqref{trG-trGhat} into the definition of $\wh y(E)$ and estimating the error term using integration by parts, as in \eqref{866}, we get
\begin{equation*}
\sup_{E \in I} \absb{\wh y(E) - y(E)} \;\prec\; N^{-c}\,.
\end{equation*}
Using the mean value theorem and the bound $\abs{y(E)} \prec 1$, we therefore get
\begin{equation*}
h \qbb{ \int_{I} \wh x(E) \, q \pb{\wh y(E) }   \, \dd E } \;=\; h \qbb{ \int_{I} x(E) \, q \pb{y(E) }   \, \dd E } + O_\prec(N^{-c})\,.
\end{equation*}
Combined with \eqref{ee427}, this concludes the proof.
\end{proof}

This concludes the proof of Theorem \ref{thm: eigenvector law} for the case of general $T$.

\subsection{The joint eigenvalue-eigenvector universality of $YY^*$ near the spectral edges}

In this section we observe that the technology developed in Section \ref{sec: QUE} allows us to establish the universality of the joint eigenvalue-eigenvector distribution of $Q$ provided that $\Sigma = I_M$. Without loss of generality, we consider the case where $Q$ is given by $H = Y Y^*$ defined in \eqref{def_H}. This result applies to arbitrary eigenvalue and eigenvector indices which are bounded by $K^{1 - \tau}$, and does in particular not need to invoke eigenvalue correlation functions.

This result generalizes the quantum unique ergodicity from Proposition \ref{prop:xie} and its extension from Remark \ref{rem: gen QUE} by also including the distribution of the eigenvalues. The universality of both the eigenvalues and the eigenvectors is formulated in the sense of fixed indices. A result in a similar spirit was given in \cite[Theorem 1.6]{KY1}, except that the upper bound on the eigenvalue and eigenvector indices $(\log K)^{C \log \log N}$ from \cite{KY1} is improved all the way to $K^{1 - \tau}$, for any $\tau > 0$. A result covering all eigenvalue and eigenvector indices, i.e.\ with an index upper bound $K$, was given in \cite[Theorem 1.10]{KY1} and \cite[Theorem 8]{TV3}, but under the assumption of a four-moment matching assumption. Theorem \ref{thm:univ_H} is a true universality result in that it does not require any moment matching assumptions, but it does require an index upper bound of $K^{1 - \tau}$ instead of $K$ on the eigenvalue and eigenvector indices.

In addition, Theorem \ref{thm:univ_H} extends the previous results from \cite{KY1} and \cite{TV3} by considering arbitrary generalized components $\scalar{\f \zeta_a}{\f v}$ of the eigenvectors. Finally, Theorem \ref{thm:univ_H} holds for the general class of covariance matrices defined in \eqref{def_H}.

\begin{theorem}[Universality for the uncorrelated case] \label{thm:univ_H}
Fix $\tau > 0$, $k = 1,2,3, \dots$, and $r = 0,1,2, \dots$. Choose an observable $h \in C^4(\R^{2k})$ satisfying
\begin{equation*}
\abs{\partial^\alpha h(x)} \;\leq\; C(1 + \abs{x})^C
\end{equation*}
for some constant $C > 0$ and for all $\alpha \in \N^{2k}$ satisfying $\abs{\alpha} \leq 4$. Let $X$ be an $(M + r) \times N$ matrix, and define $H$ through \eqref{def_H} for some orthogonal $O \in \r O(M + r)$. Denote by $\lambda_1 \geq \cdots \geq \lambda_M$ the eigenvalues of $H$ and by $\f \zeta_1, \dots, \f \zeta_M$ the associated unit eigenvectors. Let $\E^{(1)}$ and $\E^{(2)}$ denote the expectations with respect to two laws on $X$, both of which satisfy \eqref{cond on entries of X} and \eqref{moments of X-1}. Recall the definition \eqref{def_Delta} of $\Delta_a$, the typical distance between $\lambda_a$ and $\lambda_{a+1}$, and \eqref{def:gamma_alpha} of the classical location $\gamma_a$.

Then for any indices $a_1, \dots, a_k, b_1, \dots, b_k \in \qq{1, K^{1 - \tau}}$ and deterministic unit vectors $\f u_1, \f w_1, \dots, \f u_k, \f w_k \in \R^M$ we have
\begin{equation*}
(\E^{(1)} - \E^{(2)}) \, h \pBB{\frac{\lambda_{a_1} - \gamma_{a_1}}{\Delta_{a_1}}, \dots, \frac{\lambda_{a_k} - \gamma_{a_k}}{\Delta_{a_k}}, M \scalar{\f u_1}{\f \zeta_{b_1}} \scalar{\f \zeta_{b_1}}{\f w_1}, \dots, M \scalar{\f u_k}{\f \zeta_{b_k}} \scalar{\f \zeta_{b_k}}{\f w_k}} \;=\; O(N^{-c})
\end{equation*}
for some constant $c \equiv c(\tau, k, r, h) > 0$.
\end{theorem}

\begin{proof}
The proof is a Green function comparison argument, a minor modification of that developed in Section \ref{sec: QUE}. We write the distribution of $\lambda_a - \gamma_a$ in terms of the resolvent $\wh G$, starting from the Helffer-Sj\"ostrand representation \eqref{HS split 1}, exactly as in \cite[Sections 4 and 5]{KY1}. We omit further details.
\end{proof}

\begin{remark} \label{rem: universality}
In particular, Theorem \ref{thm:univ_H} establishes the fixed-index universality of eigenvalues with indices bounded by $K^{1 - \tau}$. Indeed, we may choose $\E^{(2)}$ to be the expectation with respect to a Gaussian law, in which case $H \eqdist \wt X \wt X^*$, where $\wt X$ is a $M \times N$ and Gaussian. (For example, the top eigenvalue of $H$ is distributed according to the Tracy-Widom-1 distribution, etc.)

We note that even this fixed-index universality of eigenvalues is a new result, having previously only been established under the four-moment matching condition \cite{KY1, TV3} (in the context of Wigner matrices).
\end{remark}

\begin{remark}
We formulated Theorem \ref{thm:univ_H} for the real symmetric covariance matrices of the form \eqref{def_H}, but it and its proof remain valid for complex Hermitian covariance matrices, as well as Wigner matrices (both real symmetric and complex Hermitian).
\end{remark}

\begin{remark}
Assuming $\abs{\phi - 1} > \tau$, the condition $a \leq K^{1 - \tau}$ on the indices in Theorem \ref{thm:univ_H} may be replaced with $a \notin \qq{K^{1 - \tau}, K - K^{1 - \tau}}$.
\end{remark}

\begin{remark}\label{rem:univ_Q}
Combining Theorems \ref{thm:univ_H} and \ref{thm:sticking}, we get the following universality result for $Q$. Fix $\tau > 0$, $k = 1,2,3, \dots$, and $r = 0,1,2, \dots$.
For any continuous and bounded function $h$ on $\R^k$ we have
\begin{equation*}
\lim_{N \to \infty} \qBB{\E \, h\pbb{\frac{\mu_{s_+ + a_1} - \gamma_{a_1}}{\Delta_{a_1}}, \cdots, \frac{\mu_{s_+ + a_k} - \gamma_{a_k}}{\Delta_{a_k}}} - \E^{\txt{Wish}} \, h \pbb{\frac{\lambda_{a_1} - \gamma_{a_1}}{\Delta_{a_1}}, \cdots, \frac{\lambda_{a_k} - \gamma_{a_k}}{\Delta_{a_k}}}} \;=\; 0
\end{equation*}
for any indices $a_1, \dots, a_k \leq K^{1 - \tau} \alpha_+^3$. Here $\E^{\txt{Wish}}$ denotes expectation with respect to the Wishart case, where $r = 0$, $T = I_M$, and $X$ is Gaussian. A similar result holds near the left edge provided that $\abs{\phi - 1} \geq \tau$.

\end{remark}

\section{Extension to $\dot Q$ and proof of Theorem \ref{thm: Q dot}} \label{sec: Q dot}
In this section we explain how to extend our analysis from $Q$ defined in \eqref{wtH_cov} to $\dot Q$ defined in \eqref{def_dot_Q}, hence proving Theorem \ref{thm: Q dot}. We define the resolvent
\begin{equation*}
\dot G(z) \;\deq\; \pb{X (1 - \f e \f e^*) X^* - z}^{-1}\,,
\end{equation*}
which will replace $G(z) = (X X^* - z)^{-1}$ when analysing with $\dot Q$ instead of $Q$. We begin by noting that the isotropic local laws hold for also for $\dot G$.

\begin{theorem}[Local laws for $X (1 - \f e \f e^*) X^*$] \label{IMP for XeX}
Theorem \ref{thm: IMP gen} remains valid with $G$ replaced by $\dot G$. Moreover, Theorems \ref{thm: I deloc gen} and \ref{thm: cov-rig} remain valid for $\f \zeta_i$ and $\lambda_i$ denoting the eigenvectors and eigenvalues of $X (1 - \f e \f e^*) X^*$.
\end{theorem}

\begin{proof}
As in the proof of Theorem \ref{IMP for YY}, we only prove \eqref{bound: Rij isotropic gen} for $\dot G$.
Using the identity \eqref{woodbury} we get
\begin{equation} \label{dotG_G_id}
\dot G \;=\; \pb{X X^* - z - X \f e \f e^* X^*}^{-1} \;=\; G + \frac{1}{1 - (X^* G X)_{\f e \f e}} \, G X \f e \f e^* X^* G\,.
\end{equation}
Using \eqref{bound: Rij isotropic gen}, the proof will be complete provided we can show that
\begin{equation} \label{ee_iso_claim}
\absBB{\frac{(GX)_{\f v \f e} (X^* G)_{\f e \f w}}{1 - (X^* G X)_{\f e \f e}}} \;\prec\; \Phi
\end{equation}
for unit vectors $\f v, \f w \in \R^M$, where $\Phi$ was defined in \eqref{def_Phi}. Recall the definition $R(z) \deq (X^* X - z)^{-1}$. From the elementary identity $X^*G X = 1 + zR$ and Theorem \ref{thm: IMP gen} applied to $X^*$ instead of $X$, we get
\begin{equation*}
\absbb{\frac{1}{1 - (X^* G X)_{\f e \f e}}} \;\prec\; \frac{1}{\abs{z} \abs{m_\phi}} \;\leq\; \frac{C}{1 + \phi^{1/2}}\,,
\end{equation*}
where in the last step we used \eqref{bounds on mg} and \eqref{mphi_minvphi}. Using Lemma \ref{lem: fluct_avg} below with $\f x = \f e$, \eqref{ee_iso_claim} therefore follows provided we can prove that
\begin{equation*}
\phi^{1/2} (1 + \phi^{1/2}) \Phi^2 \;\leq\; C \Phi\,.
\end{equation*}
This is an immediate consequence of the estimate $(1 + \phi) \Phi \leq C$, which itself easily follows from the definition \eqref{def_S_theta} of $\f S$ and \eqref{im_m_swap}.
\end{proof}

Next, we deal with the quantum unique ergodicity of $X (1 - \f e \f e^*) X^*$. As explained in Section \ref{sec: QUE Y}, it suffices to prove the following result.

\begin{lemma} \label{lem: QUE eX}
Lemma \ref{CPZX} remains valid if $x(E)$ and $y(E)$ are replaced with $\dot x(E)$ and $\dot y(E)$, obtained from the definitions \eqref{defxE} and \eqref{defy428} by replacing $G$ with $\dot G$.
\end{lemma}
\begin{proof}
The proof mirrors closely that of Lemma \ref{lem: QUE Y}, using the identity \eqref{dotG_G_id} instead of \eqref{wh_G_formula} as input. We omit the details.
\end{proof}
Using Theorem \ref{IMP for XeX} and Lemma \ref{lem: QUE eX}, combined with the results of Section \ref{sec: QUE Y}, we conclude the proof of Theorem \ref{thm: Q dot}. To be precise, the arguments of Sections \ref{sec:general_model} and \ref{sec: Q dot} have to be successively combined so as to obtain the isotropic local laws and quantum unique ergodicity of the matrix $Y (1 - \f e \f e^*) Y^*$. This has to be done in the following order. First, using Theorem \ref{IMP for XeX} and Lemma \ref{lem: QUE eX}, one establishes the local laws and quantum unique ergodicity for $X (1 - \f e \f e^*) X^*$. Second, using these results as input, one repeats the arguments of Section \ref{sec:general_model}, except that $X X^*$ is replaced with $X (1 - \f e \f e^*) X^*$; this is a trivial modification of the arguments presented in Section \ref{sec:general_model}. Thus we get the local laws and quantum unique ergodicity for the matrix
\begin{equation*}
Y (1 - \f e \f e^*) Y^* \;=\; (I_M,0) O X (I_N - \f e \f e^*) X^* O^* (I_M,0)^*\,.
\end{equation*}
Moreover, we find that Theorem \ref{thm:univ_H} also holds if $H = Y Y^*$ from \eqref{def_H} is replaced with $Y (1 - \f e \f e^*) Y^*$.

All that remains is the proof of the following estimate, which generalizes \eqref{bound: XRv}.
\begin{lemma} \label{lem: fluct_avg}
For $z \in \f S$ and deterministic unit vectors $\f v \in \R^M$ and $\f x \in \R^N$, we have
\begin{equation} \label{GXvu}
\absb{(GX )_{\f v \f x} }  \;\prec\;  \phi^{1/4} (1 + \phi^{1/2}) \Phi\,,
\end{equation}
where $\Phi$ was defined in \eqref{def_Phi}.
\end{lemma}
\begin{proof}
In the case where $\f x = \f e_\mu$ is a standard unit vector of $\R^N$, \eqref{GXvu} is a trivial extension of \eqref{bound: XRv} (which was proved under the assumption that $\phi \geq 1$). For general $\f x$, the proof requires more work. Indeed, writing $(GX )_{\f v \f x} = \sum_\mu (GX )_{\f v \mu} \, u(\mu)$ and estimating $\abs{(GX )_{\f v \mu}}$ by $O_\prec(\phi^{1/4} (1 + \phi^{1/2}) \Phi)$ leads to a bound proportional to the $\ell^1$-norm of $\f x$ instead of its $\ell^2$-norm. In order to obtain the sharp bound, which is proportional to the $\ell^2$-norm, we need to exploit cancellations among the summands. This phenomenon is related to the \emph{fluctuation averaging} from \cite{EKY2}, and was previously exploited in \cite{BEKYY} to obtain the isotropic laws from Theorem \ref{thm: IMP gen}. It is best made use of by estimating the $p$-th moment for an even integer $p$,
\begin{equation} \label{EGXp}
\E \absb{(GX )_{\f v \f x}}^p \;=\; \abs{z}^p \sum_{\mu_1, \dots, \mu_p} x(\mu_1) \cdots x(\mu_p) \, \E \pB{R_{\mu_1\mu_1} \, (G^{[\mu_1]} X)_{\f v \mu_1} \cdots \ol{R_{\mu_p \mu_p} \, (G^{[\mu_p]} X)_{\f v \mu_p}} }\,;
\end{equation}
here we used the first identity of \eqref{G_X}. A similar argument was given in \cite[Section 5]{BEKYY}. The basic idea is to make all resolvents on the right-hand side of \eqref{EGXp} independent of the columns of $X$ indexed by $\{\mu_1, \dots, \mu_p\}$ (see \cite[Definition 3.7]{BEKYY}). As in \cite[Section 5]{BEKYY}, we do this using the identities from \cite[Lemma 3.8]{BEKYY} for the entries of $R$. In addition, for the entries of $G$ we use the identity (in the notation of \cite[Definition 3.7]{BEKYY})
\begin{equation}
G^{[T]}_{\f v \f w} \;=\; G^{[T \mu]}_{\f v \f w} + z R^{[T]}_{\mu \mu} \sum_{k,l = 1}^M G^{[T \mu]}_{\f v k} G^{[T \mu]}_{l \f w} X_{k \mu}  X_{l \mu}\,,
\end{equation}
which follows from \eqref{Neumann} and \eqref{R_mumu}. As in \cite[Section 5]{BEKYY}, the resulting expansion may be conveniently organized using graphs, and brutally truncated after a number of steps that depends only on $p$ and $\omega$ (here $\omega$ is the constant from $\f S$ in \eqref{def_S_theta}). The key observation is that, once the expansion is performed, we may take the pairing among the variables $\{X_{k \mu_i} \col k \in \qq{1,N}, i \in \qq{1,p}\}$; we find that each independent summation index $\mu_i$ comes with a weight bounded by $x(\mu_i)^2 + N^{-1}$, which sums to $O(1)$.  We refer to \cite{BEKYY} for the full details of the method, and leave the modifications outlined above to the reader.
\end{proof}

\appendix

\section{A few remarks on applications to statistics} \label{sec:stats}
In this appendix we give a few remarks on what our results imply for applications to statistics. We assume throughout that the population covariance matrix satisfies
\begin{equation} \label{Sigma_kl_bound}
\max_k \abs{\Sigma_{kk}} \leq C
\end{equation}
for some constant $C$.

We consider the following simple model problem. Suppose there is some (unknown) set $S \subset \{1, \dots, M\}$ whose associated variables $(a_k)_{k \in S}$ are strongly correlated. For simplicity, let us assume that the correlations are given by a single spike in $\Sigma$, i.e.
\begin{equation*}
\Sigma \;=\; I_M + (\sigma - 1) \f v \f v^*\,, \qquad \sigma - 1 \;=\; \phi^{1/2} d\,,
\end{equation*}
where the spike direction $\f v = (v(k))_{k = 1}^M$ is given by
\begin{equation*}
v(k) \;\deq\;
\begin{cases}
\abs{S}^{-1/2} & \text{if } k \in S
\\
0 & \text{if } k \notin S\,.
\end{cases}
\end{equation*}
(We choose this precise form for $\f v$ so as to simplify the presentation as much as possible. The following discussion also holds if $\f v$ is essentially supported on $S$, but the magnitude of its entries is not necessarily constant.)
Moreover, for simplicity we assume that $T = \Sigma^{1/2}$. In components, we have
\begin{equation} \label{Sigma_kl_example}
\Sigma_{kl} \;=\; \delta_{kl} + (\sigma - 1) v(k) v(l)\,.
\end{equation}

The goal is to recover the set $S$ from an observed realization of the sample covariance matrix $\cal Q$. (In the more general case where $\f v$ is not constant on $S$, one may easily recover its entries from the submatrix $(\cal Q_{kl})_{k,l \in S}$ once $S$ has been determined.)

The most naive way to proceed is to compare the entries of $\cal Q$ with those of $\Sigma$. Using \eqref{Sigma_kl_bound} it is not hard to conclude that
\begin{equation*}
\cal Q_{kl} \;=\; \Sigma_{kl} + O_\prec(N^{-1/2})\,.
\end{equation*}
%where $O_\prec(N^{-1/2})$ denotes a random error term that is bounded with high probability by $N^{\epsilon - 1/2}$ for all $\epsilon > 0$; see Definition \ref{def:stocdom} below for a precise definition.
We look at the off-diagonal terms of $\cal Q$ and infer that $k$ belongs to $S$ if there exists an index $l$ such that $\cal Q_{kl}$ is much larger than $N^{-1/2}$. For this approach to work, we require that $\abs{\Sigma_{kl}} \gg N^{-1/2}$, which reads $(\sigma - 1) v(k) v(l) \gg N^{-1/2}$. We conclude that this naive entrywise approach works provided that
\begin{equation} \label{cond_trivial}
\sigma - 1 \;\gg\; \frac{\abs{S}}{\sqrt{N}}\,.
\end{equation}

In contrast, according to Theorems \ref{thm: outlier locations} and \ref{thm: outlier eigenvectors}, looking at the principal components of $\cal Q$ allows us to determine $S$ from the top eigenvector $\f \xi_1$ provided the spike $\sigma$ is supercritical, i.e.\ gives rise to an outlier separated by a distance of order one from the bulk spectrum. This gives the condition
\begin{equation} \label{cond_PCA}
\sigma - 1 \;\gg\; \phi^{1/2}\,.
\end{equation}
The principal component analysis works and the naive componentwise approach does not if \eqref{cond_PCA} holds and \eqref{cond_trivial} does not. These conditions may be written as
\begin{equation*}
\phi^{1/2} \;\ll\; \sigma - 1 \;\ll\; \frac{\abs{S}}{\sqrt{M}}\,.
\end{equation*}
Hence the principal component analysis for this example is very effective when the family of correlated variables is quite large, $\abs{S} \gg \phi^{1/2} \sqrt{M}$.

More generally, the principal component approach may work in the regime
\begin{equation} \label{cond_PCA2}
\abs{S} \;\gg\; \phi^{1/2}\,
\end{equation}
and cannot work for smaller $\abs{S}$. Indeed, the assumption \eqref{Sigma_kl_bound} is satisfied for $\sigma \leq \abs{S}$, so that in the case of the strongest possible correlations, $\sigma \asymp \abs{S}$, the condition \eqref{cond_PCA} reduces to \eqref{cond_PCA2}. On the other hand, if $\abs{S} \ll \phi^{1/2}$, then from \eqref{Sigma_kl_example} and the assumption \eqref{Sigma_kl_bound} we find $\sigma - 1 \leq \abs{S} \ll \phi^{1/2}$, in contradiction to \eqref{cond_PCA}.

Clearly, by \eqref{cond_PCA}, for the purposes of statistical inference it is desirable to make $\phi$ as small as possible. This means that the number of samples per variables is as large as possible. It is therefore natural to attempt to make $M$ smaller so as to reduce $\phi$. Obviously, if we know a priori that $S$ is contained in some subset $Y$ of $\{1, \dots, M\}$ of size $M/2$, then we simply discard all variables indexed by $Y^c$ and consider the correlations restricted to $(a_i)_{i \in Y}$; we have halved $\phi$ in the process.

However, if we have no such a priori knowledge about $S$, discarding half of the variables $a_i$ is a bad idea. In this case, the best one can do is to choose $Y$ at random. Thus, suppose that $S$ is uniformly distributed among the subsets of $\{1, \dots, M\}$ of size $\abs{S}$. We cut the sample space $\{1, \dots, M\}$ in half by keeping only the $M/2$ first elements. We therefore obtain a new family of variables with dimensional parameters
\begin{equation*}
\wt M \;\deq\; M/2 \,, \qquad \wt \phi \;\deq\; \phi/2\,.
\end{equation*}
Let $\wt \Sigma$ be the $M/2 \times M/2$ matrix obtained from $\Sigma = 1 + \phi^{1/2} d \f v \f v^*$ by restricting it to the first $M/2$ elements, i.e. $\wt \Sigma \deq (\Sigma_{kl})_{k,l = 1}^{M/2}$. Note that $\wt \Sigma$ is again a rank-one perturbation of the identity. We write it in the form
\begin{equation*}
\wt \Sigma \;=\; I_{M/2} + \wt \phi^{1/2} \wt d \wt {\f v} \wt {\f v}^*\,,
\end{equation*}
where $\wt {\f v}$ is a unit vector.
Since
\begin{equation*}
\absb{\hb{k \col \wt v(k) \neq 0}} \;\approx\; \frac{\abs{S}}{2}
\end{equation*}
with high probability, we find that $\wt v(k) \approx \sqrt{2} v(k)$ for $k \in S$. Picking an entry $\Sigma_{kl} = \wt \Sigma_{kl}$ for $k,l \in S \cal \{1, \dots, M/2\}$, we obtain
\begin{equation*}
\phi^{1/2} \, d \, v(k) \, v(l) \;=\; \wt \phi^{1/2} \, \wt d \, \wt v(k) \, \wt v(l)\,.
\end{equation*}
Therefore,
\begin{equation*}
\wt d \;\approx\; d / \sqrt{2}\,.
\end{equation*}
We conclude that detecting spikes in the new problem is \emph{more difficult} than in the original problem, and the halving of sample space is therefore counterproductive unless one has some good a priori information about $\Sigma$.

We conclude this appendix with a remark about the use of the \emph{non-outlier} principal components for statistical inference.
Theorem \ref{thm: eigenvector law} implies that the non-outlier eigenvectors near the edge are all biased in the direction of $\f v_i$ provided that $d_i$ is near the BBP transition point $1$. Suppose for simplicity that $\phi$ is bounded. Then Theorem \ref{thm: eigenvector law} implies that for $d_i$ near $1$ we have $\scalar{\f w}{\f \xi_a}^2 \asymp (d_i - 1)^{-2} M^{-1}$ with high probability. We can therefore detect the spike $d_i$ through the resulting bias in the direction $\f v_i$ as long as $\abs{d_i - 1} \ll 1$. (Recall that the unbiased, or completely delocalized, case corresponds to $\scalar{\f w}{\f \xi_a}^2 \asymp M^{-1}$.) Hence, the non-outlier eigenvectors $\f \xi_a$ retain information even about \emph{subcritical} spikes. This is in stark contrast to the eigenvalues $\mu_a$, which, by Theorem \ref{thm:sticking}, retain no information about the subcritical spikes.

If $d_i \leq 1 - \tau$ for some fixed $\tau$, then the principal components of $Q$ contain no information about the spike $d_i$. We illustrate this using the simple model from \eqref{Sigma_kl_example}. Suppose we try to determine the set $S$ by choosing the components of $\f \xi_a$ that are much larger than the unbiased background value $M^{-1/2}$. This works provided that for $k \in S$ we have
\begin{equation*}
\xi_a(k)^2 \;\asymp\; \frac{(1 + \phi)^{1/2} v(k)^2}{M} \;\gg\; \frac{1}{M}\,.
\end{equation*}
Using $v(k)^2 = \abs{S}^{-1}$ for $k \in S$, we therefore get the condition $(1 + \phi)^{1/2} \gg \abs{S}$. This however cannot hold, since we have by assumption \eqref{Sigma_kl_bound} on $\Sigma$ that
\begin{equation*}
C \;\geq\; \Sigma_{kk} \;\asymp\; 1 + \phi^{1/2} \abs{S}^{-1}\,.
\end{equation*}

{\small

\providecommand{\bysame}{\leavevmode\hbox to3em{\hrulefill}\thinspace}
\providecommand{\MR}{\relax\ifhmode\unskip\space\fi MR }
\providecommand{\MRhref}[2]{%
  \href{http://www.ams.org/mathscinet-getitem?mr=#1}{#2}
}
\providecommand{\href}[2]{#2}

}

%{\small
%\subsection*{Acknowledgement}
%This material is based upon work supported by the National Science
%Foundation under agreement No.\ DMS-1128255. Any opinions, findings, and conclusions
%or recommendations expressed in this material are those of the author(s)
%and do not necessarily reflect the views of the National Science
%Foundation.
%}

\end{document}